\documentclass[12pt]{amsart}
\usepackage{amsmath,amssymb,url,color,tikz,colortbl,stmaryrd,bm}
\usetikzlibrary{patterns} 
\usepackage[all]{xy}
\usepackage{young}
\usepackage{enumitem}
\usepackage[margin=1in]{geometry}
\usepackage{hyperref}
\usepackage{amsfonts}
\usepackage{amsmath}
\usepackage{amssymb}
\usepackage{amsthm}
\usepackage{xcolor}

\makeatletter
\newcommand{\raisemath}[1]{\mathpalette{\raisem@th{#1}}}
\newcommand{\raisem@th}[3]{\raisebox{#1}{$#2#3$}}
\makeatother

\usepackage{graphicx}
\usepackage{listings}
\usepackage{tram}
\usepackage{manfnt}
\usepackage{url,color,tikz,euscript}

\usetikzlibrary{arrows,automata}
\usepackage[colorinlistoftodos,color=red!70]{todonotes}

\theoremstyle{theorem}
\newtheorem{thm}{Theorem}

\newtheorem{lem}[thm]{Lemma}
\newtheorem{lemma}[thm]{Lemma}
\newtheorem{prop}[thm]{Proposition}

\newtheorem{cor}[thm]{Corollary}

\newtheorem{conj}[thm]{Conjecture}

\theoremstyle{definition}
\newtheorem{defn}[thm]{Definition}
\newtheorem{definition}[thm]{Definition}

\newtheorem{notation}[thm]{Notation}
\newtheorem{ex}[thm]{Example}
\newtheorem{example}[thm]{Example}

\newtheorem{remark}[thm]{Remark}

\usepackage{scalerel}
\usepackage{mathtools}
\newcommand{\bigsig}{{\sum}}
\DeclareMathOperator*{\sumpar}{\bigsig^{\ \mathclap{\|}\mathclap{\rule[1.3pt]{6pt}{0.4pt}}}\:}
\def\parallelsum{\:\ \mathclap{\|}\mathclap{-}\ \:}

\newcommand{\sumparlimits}[2]{{\sum\limits_{#1}^#2}^{\hspace{-3 pt}\raisemath{-8.5pt}{\parallelsum}}\hspace{-1.7pt}}

\newcommand{\sumparlimitlower}[1]{{\sum\limits_{#1}}^{\hspace{-3 pt}\raisemath{-2.5pt}{\parallelsum}}\hspace{-2.9pt}}

\numberwithin{thm}{section}
\catcode`@=11
\def\m@th{\mathsurround\z@}
\def\cases#1{\left\{\,\vcenter{\normalbaselines\m@th
    \ialign{$##\hfil$&\quad##\hfil\crcr#1\crcr}}\right.}
\def\hang{\hangindent 24pt}
\def\d@nger{\medbreak\begingroup\clubpenalty=10000
  \def\par{\endgraf\endgroup\medbreak} %
  \noindent\hang\hangafter=-2
  \hbox to0pt{\hskip-\hangindent\dbend\hfill}}
\outer\def\danger{\d@nger}

\newcommand{\rr}{\mathbb{R}}
\newcommand{\zz}{\mathbb{Z}}
\newcommand{\qq}{\mathbb{Q}}

\newcommand{\kk}{\mathbb{K}}
\newcommand{\ff}{\mathbb{F}}
\newcommand{\bbs}{\mathbb{S}}

\newcommand{\e}{\varepsilon}

\newcommand{\tog}{\operatorname{Tog}}

\newcommand{\btog}{\operatorname{BTog}}
\newcommand{\BTog}{\operatorname{BTog}}
\newcommand{\ntog}{\operatorname{NTog}}

\newcommand{\mc}{\operatorname{MC}}

\newcommand{\rk}{\operatorname{rk}}

\newcommand{\ra}{\rightarrow}
\newcommand{\dra}{\dashrightarrow}
\newcommand{\down}{\nabla}
\newcommand{\up}{\Delta}
\newcommand{\BAR}{\operatorname{BAR}}
\newcommand{\BAG}{\operatorname{BAG}}
\newcommand{\BOG}{\operatorname{BOG}}
\newcommand{\BOR}{\operatorname{BOR}}

\newcommand{\NAR}{\operatorname{NAR}}
\newcommand{\NOR}{\operatorname{NOR}}
\newcommand{\sm}{\setminus}
\definecolor{green}{HTML}{006600}
\definecolor{orange}{HTML}{FF6200}
\definecolor{purple}{HTML}{990099}
\definecolor{coral}{HTML}{FF7F50}
\definecolor{mahogany}{HTML}{C04000}
\definecolor{gold}{HTML}{DAA541}
\definecolor{chocolate}{HTML}{D5691E}

\newcommand{\rkT}[1]{\boldsymbol{T}_{\rk=#1}}
\newcommand{\rktau}[1]{\boldsymbol{\tau}_{\rk=#1}}

\newcommand{\flatbin}{\mathbin{\flat}}

\catcode`@=12 
\renewcommand{\labelitemii}{\scriptsize\raise2pt\hbox{$\!{\blacktriangleright}$}} 

\newcommand{\cala}{\mathcal{A}}

\newcommand{\calc}{\mathcal{C}}

\newcommand{\calf}{\mathcal{F}}

\newcommand{\calj}{\mathcal{J}}

\newcommand{\calo}{\mathcal{O}}
\newcommand{\calp}{\mathcal{P}}

\newcommand{\calr}{\mathcal{R}}

\def\rowm{\rho}
\def\rowA{\rowm_{\mathcal{A}}}
\def\rowJ{\rowm_{\mathcal{J}}}

\newcommand{\ds}{\displaystyle}
\newenvironment{absolutelynopagebreak}
  {\par\nobreak\vfil\penalty0\vfilneg
   \vtop\bgroup}
  {\par\xdef\tpd{\the\prevdepth}\egroup
   \prevdepth=\tpd}

\newcommand\mydots{\ifmmode\makebox[1.2em][c]{$\cdot$\hfil$\cdot$\hfil$\cdot$}\fi}

\title[Birational and noncommutative antichain rowmotion]{Birational and noncommutative
lifts of antichain toggling and rowmotion}
\author[Joseph]{Michael Joseph}
\address{Department of Technology and Mathematics, Dalton State College, 650 College Dr., Dalton, GA 30720, USA}
\email{mjosephmath@gmail.com}
\author[Roby]{Tom Roby}
\address{Department of Mathematics, University of Connecticut, Storrs, CT 06269-1009, USA}
\email{tom.roby@uconn.edu}
\date{\today}

\subjclass[2010]{05E18, 06A07, 12E15}
\keywords{antichain,
birational rowmotion, 
dynamical algebraic combinatorics,
graded poset,
homomesy, 
isomorphism,
noncommutative algebra,
periodicity, 
rowmotion,
toggle group,
transfer map. 
}

\begin{document}

\begin{abstract}
The rowmotion action on order ideals or on antichains of a finite partially ordered set has been studied (under a variety of names) by many authors.  Depending on the poset, one finds unexpectedly interesting orbit structures, instances of (small order) periodicity, cyclic sieving, and homomesy. Many of these nice features still hold when the action is extended to $[0,1]$-labelings of the poset or (via detropicalization) to labelings by rational functions (the \textit{birational} setting).  

In this work, we parallel the birational lifting already done for order-ideal rowmotion to antichain rowmotion.  We give explicit equivariant bijections between the birational toggle groups and between their respective liftings.  We further extend all of these notions to labellings by \textit{noncommutative} rational functions, setting an unpublished periodicity conjecture of Grinberg in a broader context. 


\end{abstract}

\maketitle
\vspace{-0.3 in}
\section{Introduction}\label{sec:intro}

\textit{Combinatorial rowmotion} is a well-studied action on the set of order ideals
$\calj(P)$ or on the set of antichains $\cala(P)$ of a finite poset $P$. 
It was first studied as a map on $\cala(P)$
by Brouwer and Schrijver~\cite{brouwer1974period}, and goes by several names.  In recent literature, the name 
``rowmotion,'' due to Striker and Williams~\cite{strikerwilliams} (who summarize the
history), has stuck.  An updated historical survey is available in Thomas and
Williams~\cite[\S7]{rowmotion-in-slow-motion}.

Rowmotion has proven to be of great interest in dynamical algebraic combinatorics.  On
several ``nice'' posets (e.g., positive root posets or minuscule posets, such as products of
two chains), rowmotion exhibits various phenomena including \emph{periodicity} (of a relatively
small order), \emph{cyclic sieving} (as defined by Reiner, Stanton, and White~\cite{csp}),
\emph{homomesy} (where a natural statistic, e.g.~cardinality, has the same average over every
orbit)~\cite{ast,shahrzad,indepsetspaper,panyushev,propproby,rush-wang-homomesy-minuscule,vorland-3-chains,vorland-thesis},
and \emph{resonance} (see~\cite{dpsresonance,dilks-striker-vorland}).
Rowmotion is related to Auslander--Reiten translation on certain quivers~\cite{Yil17}.   

Quite surprisingly, some of these features extend to the piecewise-linear (order polytope)
level and can be lifted further to the birational level~\cite{einpropp}.  One sometimes
gets periodicity of the same order as the combinatorial map, and often homomesy extends as
well~\cite{GrRo16,GrRo15,musiker-roby}. Just as one example, the file-cardinality homomesy
for order-ideal rowmowtion on rectangular posets~\cite[Thm.~19ff]{propproby} lifts to a birational
homomesy~\cite[Thm.~2.16]{musiker-roby};  Rush and Wang's extension of this result to all minuscule
posets~\cite[Thm.~1.2]{rush-wang-homomesy-minuscule} has recently been lifted to the birational
realm by S.~Okada~\cite{okada-minuscule}.

It is a continuing mystery why certain properties of combinatorial rowmotion on
many families of posets lift to the birational realm.
This question is still far from answered, but Hopkins has shown
some properties must lift to the birational realm on posets which satisfy the tCDE property and have a grid-like structure~\cite{hopkins2019minuscule}.
Birational rowmotion is related to $Y$-systems of type $A_m \times A_n$ described in
Zamolodchikov periodicity~\cite[\S4.4]{robydac}.

The lifting of order-ideal rowmotion (herein denoted $\rowJ$) to BOR-motion (birational order rowmotion) proceeds by first writing $\rowJ$ as a product of
involutions called \textit{toggles}, each of which acts on $\calj(P)$, the set of order ideals of a
poset. These toggles are then extended to Stanley's \textit{order polytope} $\calo\calp(P)$,
and then lifted further to toggles at
the birational level~\cite{einpropp}, following the lead of Kirillov and Berenstein~\cite{berenstein}.
Letting $\kk$ be a field of characteristic zero, we lift from a piecewise-linear map to a birational map through detropicalization of operations.
Any equality of rational expressions (such as periodicity or
homomesy) that does not contain subtraction or
additive inverses 
also holds in the piecewise-linear realm (by tropicalization) and furthermore in the
combinatorial realm (by restriction); see~\cite[Remark~10]{GrRo16}.

As part of a broader study of toggling in general, Striker defined \emph{antichain toggles}
that act on $\cala (P)$~\cite{strikergentog}.  The first
author gave an explicit isomorphism between these two different toggle groups (on $\calj (P)$
and on $\cala(P)$) for the same
poset $P$, and extended these results to the piecewise-linear level~\cite{antichain-toggling}, 
where $\cala (P)$
extends to Stanley's \emph{chain polytope} $\calc (P)$~\cite{Sta86}.  These toggles can be
used to define the \emph{antichain rowmotion} of~\cite{brouwer1974period} and its extension
to all of $\calc (P)$.

Our goal in this work is to study the parallel lifting of this map on $\calc (P)$ to the
birational level, which we call  \emph{Birational Antichain Rowmotion} or \emph{BAR-motion}
(Definition~\ref{def:BAR}) for short. We construct equivariant bijections between this
action and the previously studied BOR-motion, allowing us to deduce properties of one from
the other.  We also describe a noncommutative analogue of both these maps (the first
originally discovered by Darij Grinberg, unpublished), and prove that these bijections
extend even to this realm.

The paper is organized as follows.  In Section~\ref{sec:jungle rollers}, we include the
necessary background on rowmotion at the combinatorial, piecewise-linear, and birational
levels, including ways to view them as compositions of transfer maps and as products of
toggles.  This positions us in Section~\ref{sec:batar} to define \textit{birational
antichain toggling} and \textit{BAR-motion} and construct the explicit bijection between the
two different toggle groups at the birational level. 

Section~\ref{sec:graded} contains our results that pertain specifically to graded
posets, which are the only ones known thus far to exhibit homomesy or periodicity in the birational realm. 
Since toggles within the same rank commute with each other, we can toggle them ``all at
once''.  Toggling by ranks (hence the term ``rowmotion'') from top to bottom gives one map,
but we can also toggle first all even ranks, then all odd ones, giving a map
called \textit{gyration} by Striker~\cite{strikerRS}. 
Grinberg and the second author worked with graded rescalings of
poset labelings in several proofs in~\cite{GrRo16,GrRo15}.
We discuss the analogues of these ideas under the antichain perspective.

In Section~\ref{sec:noncommutative}, we lift our birational actions
further to the \emph{noncommutative realm}, where we do not assume commutativity
of multiplication.  This setting has not appeared in the literature before, but is based on
unpublished definitions and conjectures of Darij Grinberg.  
In this paper, we show that \emph{NOR-motion (Noncommutative Order Rowmotion)} and
\emph{NAR-motion (Noncommutative Antichain Rowmotion)} 
always exhibit the same order on any given finite poset.
We use ``noncommutative realm'' as a short term, but we really mean
``not-necessarily-commutative birational realm'' as fields are skew fields. 

We defer the proofs of several results in Section~\ref{sec:batar} since they follow from their
noncommutative analogues in Section~\ref{sec:noncommutative}. 
(We originally proved the results in this realm before realizing we could extend
them to the noncommutative setting.)  
Toggles are no longer involutions, so it is surprising that many other key properties do continue to hold
(with suitably modified definitions), such as
the isomorphism between the order and antichain toggle groups. 

To summarize, we have four realms (combinatorial, piecewise-linear, birational, and
noncommutative) and two rowmotion maps (order-ideal and antichain).  Our new work here
involves lifting the latter map to the birational and noncommutative realms, and giving
explicit equivariant isomorphisms connecting these two maps at the two highest levels.  
Beyond the inherent interest of showing that Stanley's transfer maps between $\calo (P)$ and
$\calc (P)$ lift nicely to the birational and noncommutative birational settings, we hope
that having different approaches to these maps will help shed light on some of their
tantalizing properties.  In particular, on several minuscule and root posets, BOR-motion has
the same order as combinatorial
order-ideal rowmowtion, and this is conjectured to extend to NOR-motion as well.  Our
results show that a proof for NAR-motion would automatically imply it for NOR-motion as
well.  

To prove refined versions of homomesy in the product of two chain posets, J.~Propp and the
second author used an equivariant bijection discovered (less formally) by R.~Stanley and
H.~Thomas~\cite[\S 7]{propproby}.  
In a separate paper~\cite{STword} we explore the lifting of this ``Stanley--Thomas word'' to the
piecewise-linear, birational, and noncommutative realms. Although the map is no longer a
bijection, so cannot be used to prove periodicity directly, it still gives enough
information to prove the homomesy at the piecewise-linear and birational levels (a result
previously shown by D.~Grinberg, S.~Hopkins, and S.~Okada). Even at the noncommutative
level, the Stanley--Thomas word of a poset labeling rotates cyclically with the lifting of
antichain rowmotion.

\subsection{Acknowledgements}
The authors are grateful for useful conversations with
David Einstein,
Darij Grinberg,
Sam Hopkins,
Soichi Okada,
James Propp, 
Richard Stanley, 
Jessica Striker,
Corey Vorland, 
Nathan Williams.  We also thank an anonymous referee, who made a number of useful suggestions 
incorporated into this version.

\section{Rowmotion in the combinatorial, piecewise-linear, and birational realms}\label{sec:jungle rollers}

This section contains the necessary background for this paper.
We discuss the toggle group of a poset $P$, rowmotion on order ideals and on antichains, and
define their
generalizations to the piecewise-linear realm.  We also discuss the lifting of order-ideal
rowmotion to the birational realm. 
Our new results begin in Section~\ref{sec:batar} with the birational lifting of
antichain rowmotion.

\subsection{Rowmotion in the combinatorial realm}

We assume familiarity with basic notions from the theory of posets,
as discussed in \cite[Ch. 3]{ec1ed2}.
Throughout this paper $P$ will denote a finite poset.

Following the notation of Einstein-Propp~\cite{einpropp}, we can define rowmotion via the
following natural bijections between the set $\calj(P)$ of all 
\emph{order ideals} of $P$, the set $\calf(P)$ of all \emph{order filters} of $P$, and the set $\cala(P)$ of
all \emph{antichains} of $P$. 
\begin{itemize}
\item The map $\Theta: 2^P \ra 2^P$ where $\Theta(S) = P\sm S$ is the complement of $S$ (so $\Theta$ sends order ideals to order filters and vice versa).
\item The \textbf{up-transfer} $\up:\calj(P) \ra \cala(P)$, where $\up(I)$ is the set of maximal elements of $I$.  For an antichain $A\in\cala(P)$,
$\up^{-1}(A)=\{x\in P: x\leq y \text{ for some } y\in A\}$ (``downward saturation'').
\item The \textbf{down-transfer} $\down:\calf(P) \ra \cala(P)$, where $\down(F)$ is the set of minimal elements of $F$.  For an antichain $A\in\cala(P)$,
$\down^{-1}(A)=\{x\in P: x\geq y \text{ for some } y\in A\}$ (``upward saturation'').
\end{itemize}

\begin{definition}\label{def:OIrow}
\textbf{Order-ideal rowmotion} is the map $\rowJ: \calj(P) \ra \calj(P)$ given by the composition
$\rowJ = \up^{-1}\circ\down\circ\Theta$.
\textbf{Antichain rowmotion} is the map $\rowA: \cala(P) \ra \cala(P)$ given by the composition
$\rowA = \down\circ\Theta\circ\up^{-1}$.
\textbf{Order-filter rowmotion} is the map $\rho_{\calf}: \calf (P)\ra \calf (P)$ given by the composition $\rho_\calf = \Theta\circ\up^{-1}\circ\down$.
\end{definition}

\begin{example}\label{ex:a3-row}
Below we show examples of $\rowJ$  and $\rowA$ on the positive root poset
$\Phi^+(A_3)$. In each step, the elements of the subset of the poset are given by the
filled-in circles.
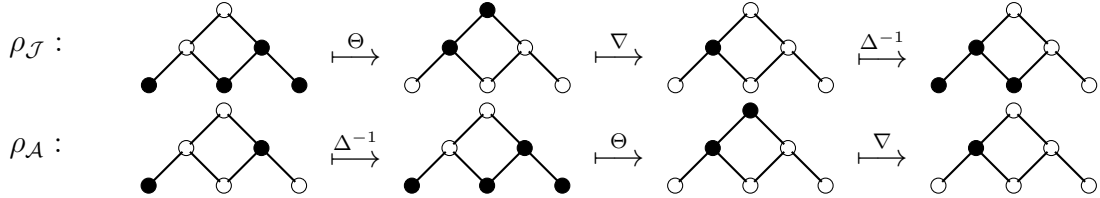
\begin{center}
\begin{tikzpicture}[xscale=0.5,yscale=0.5]
\node at (-5,-3) {$\rowJ:$};
\begin{scope}[shift={(0,-4)}]
\draw[thick] (-0.1, 1.9) -- (-0.9, 1.1);
\draw[thick] (0.1, 1.9) -- (0.9, 1.1);
\draw[thick] (-1.1, 0.9) -- (-1.9, 0.1);
\draw[thick] (-0.9, 0.9) -- (-0.1, 0.1);
\draw[thick] (0.9, 0.9) -- (0.1, 0.1);
\draw[thick] (1.1, 0.9) -- (1.9, 0.1);
\draw (0,2) circle [radius=0.2];
\draw (-1,1) circle [radius=0.2];
\draw[fill] (1,1) circle [radius=0.2];
\draw[fill] (-2,0) circle [radius=0.2];
\draw[fill] (0,0) circle [radius=0.2];
\draw[fill] (2,0) circle [radius=0.2];
\end{scope}
\node at (3.5,-3) {$\stackrel{\Theta}{\longmapsto}$};
\begin{scope}[shift={(7,-4)}]
\draw[thick] (-0.1, 1.9) -- (-0.9, 1.1);
\draw[thick] (0.1, 1.9) -- (0.9, 1.1);
\draw[thick] (-1.1, 0.9) -- (-1.9, 0.1);
\draw[thick] (-0.9, 0.9) -- (-0.1, 0.1);
\draw[thick] (0.9, 0.9) -- (0.1, 0.1);
\draw[thick] (1.1, 0.9) -- (1.9, 0.1);
\draw[fill] (0,2) circle [radius=0.2];
\draw[fill] (-1,1) circle [radius=0.2];
\draw (1,1) circle [radius=0.2];
\draw (-2,0) circle [radius=0.2];
\draw (0,0) circle [radius=0.2];
\draw (2,0) circle [radius=0.2];
\end{scope}
\node at (10.5,-3) {$\stackrel{\down}{\longmapsto}$};
\begin{scope}[shift={(14,-4)}]
\draw[thick] (-0.1, 1.9) -- (-0.9, 1.1);
\draw[thick] (0.1, 1.9) -- (0.9, 1.1);
\draw[thick] (-1.1, 0.9) -- (-1.9, 0.1);
\draw[thick] (-0.9, 0.9) -- (-0.1, 0.1);
\draw[thick] (0.9, 0.9) -- (0.1, 0.1);
\draw[thick] (1.1, 0.9) -- (1.9, 0.1);
\draw (0,2) circle [radius=0.2];
\draw[fill] (-1,1) circle [radius=0.2];
\draw (1,1) circle [radius=0.2];
\draw (-2,0) circle [radius=0.2];
\draw (0,0) circle [radius=0.2];
\draw (2,0) circle [radius=0.2];
\end{scope}
\node at (17.5,-3) {$\stackrel{\up^{-1}}{\longmapsto}$};
\begin{scope}[shift={(21,-4)}]
\draw[thick] (-0.1, 1.9) -- (-0.9, 1.1);
\draw[thick] (0.1, 1.9) -- (0.9, 1.1);
\draw[thick] (-1.1, 0.9) -- (-1.9, 0.1);
\draw[thick] (-0.9, 0.9) -- (-0.1, 0.1);
\draw[thick] (0.9, 0.9) -- (0.1, 0.1);
\draw[thick] (1.1, 0.9) -- (1.9, 0.1);
\draw (0,2) circle [radius=0.2];
\draw[fill] (-1,1) circle [radius=0.2];
\draw (1,1) circle [radius=0.2];
\draw[fill] (-2,0) circle [radius=0.2];
\draw[fill] (0,0) circle [radius=0.2];
\draw (2,0) circle [radius=0.2];
\end{scope}
\end{tikzpicture}
\end{center}

\begin{center}
\begin{tikzpicture}[xscale=0.5,yscale=0.5]
\node at (-5,1) {$\rowA:$};
\begin{scope}
\draw[thick] (-0.1, 1.9) -- (-0.9, 1.1);
\draw[thick] (0.1, 1.9) -- (0.9, 1.1);
\draw[thick] (-1.1, 0.9) -- (-1.9, 0.1);
\draw[thick] (-0.9, 0.9) -- (-0.1, 0.1);
\draw[thick] (0.9, 0.9) -- (0.1, 0.1);
\draw[thick] (1.1, 0.9) -- (1.9, 0.1);
\draw (0,2) circle [radius=0.2];
\draw (-1,1) circle [radius=0.2];
\draw[fill] (1,1) circle [radius=0.2];
\draw[fill] (-2,0) circle [radius=0.2];
\draw (0,0) circle [radius=0.2];
\draw (2,0) circle [radius=0.2];
\end{scope}
\node at (3.5,1) {$\stackrel{\up^{-1}}{\longmapsto}$};
\begin{scope}[shift={(7,0)}]
\draw[thick] (-0.1, 1.9) -- (-0.9, 1.1);
\draw[thick] (0.1, 1.9) -- (0.9, 1.1);
\draw[thick] (-1.1, 0.9) -- (-1.9, 0.1);
\draw[thick] (-0.9, 0.9) -- (-0.1, 0.1);
\draw[thick] (0.9, 0.9) -- (0.1, 0.1);
\draw[thick] (1.1, 0.9) -- (1.9, 0.1);
\draw (0,2) circle [radius=0.2];
\draw (-1,1) circle [radius=0.2];
\draw[fill] (1,1) circle [radius=0.2];
\draw[fill] (-2,0) circle [radius=0.2];
\draw[fill] (0,0) circle [radius=0.2];
\draw[fill] (2,0) circle [radius=0.2];
\end{scope}
\node at (10.5,1) {$\stackrel{\Theta}{\longmapsto}$};
\begin{scope}[shift={(14,0)}]
\draw[thick] (-0.1, 1.9) -- (-0.9, 1.1);
\draw[thick] (0.1, 1.9) -- (0.9, 1.1);
\draw[thick] (-1.1, 0.9) -- (-1.9, 0.1);
\draw[thick] (-0.9, 0.9) -- (-0.1, 0.1);
\draw[thick] (0.9, 0.9) -- (0.1, 0.1);
\draw[thick] (1.1, 0.9) -- (1.9, 0.1);
\draw[fill] (0,2) circle [radius=0.2];
\draw[fill] (-1,1) circle [radius=0.2];
\draw (1,1) circle [radius=0.2];
\draw (-2,0) circle [radius=0.2];
\draw (0,0) circle [radius=0.2];
\draw (2,0) circle [radius=0.2];
\end{scope}
\node at (17.5,1) {$\stackrel{\down}{\longmapsto}$};
\begin{scope}[shift={(21,0)}]
\draw[thick] (-0.1, 1.9) -- (-0.9, 1.1);
\draw[thick] (0.1, 1.9) -- (0.9, 1.1);
\draw[thick] (-1.1, 0.9) -- (-1.9, 0.1);
\draw[thick] (-0.9, 0.9) -- (-0.1, 0.1);
\draw[thick] (0.9, 0.9) -- (0.1, 0.1);
\draw[thick] (1.1, 0.9) -- (1.9, 0.1);
\draw (0,2) circle [radius=0.2];
\draw[fill] (-1,1) circle [radius=0.2];
\draw (1,1) circle [radius=0.2];
\draw (-2,0) circle [radius=0.2];
\draw (0,0) circle [radius=0.2];
\draw (2,0) circle [radius=0.2];
\end{scope}
\end{tikzpicture}
\end{center}
\end{example}

\subsection{The order-ideal toggle group}
The map $\rowJ$ can also be written a composition of involutions on $\calj (P)$ called
\emph{toggles}, as first shown by Cameron and Fon-Der-Flaass~\cite{cameronfonder}. 

\begin{definition}[\cite{cameronfonder}]\label{def:comb-row}
For $v\in P$, the \textbf{order-ideal toggle} corresponding to $v$ is the map
$T_v: \calj(P)\ra \calj(P)$
defined by
$$T_v(I)=\left\{\begin{array}{ll}
I\cup\{v\} &\text{if $v\not\in I$ and $I\cup\{v\}\in \calj(P)$,}\\
I\sm\{v\} &\text{if $v\in I$ and $I\sm\{v\}\in \calj(P)$,}\\
I &\text{otherwise.}
\end{array}\right.$$
Let $\tog_\calj(P)$ denote the \textbf{toggle group} of $\calj(P)$, i.e., the subgroup of
$\frak{S}_{\calj (P)}$  (the symmetric group on $\calj (P)$) generated by the toggles $\{T_v : v\in P\}$.
\end{definition}

The toggle $T_v$ either adds or removes $v$ from the order ideal if the resulting set is still an order ideal, and otherwise does nothing.
 

\begin{defn}[{\cite[\S3.5]{ec1ed2}}]
A sequence $(x_1,x_2,\dots,x_n)$ containing all of the elements of a finite poset $P$ exactly once is called a \textbf{linear extension} of $P$ if it is order-preserving, that is, whenever $x_i<x_j$ in $P$ then $i<j$.
\end{defn}

\begin{prop}[\cite{cameronfonder}]\label{prop:row-toggles}
For any linear extension $(x_1,x_2,\dots,x_n)$ of $P$, order-ideal rowmotion is given
by $\rowJ = T_{x_1} T_{x_2} \cdots T_{x_n}$.
\end{prop}

\begin{example}\label{ex:row-toggles}
For the poset
$P=[2]\times[3]$
of Example~\ref{ex:a3-row}, as labeled below, $(a,b,c,d,e,f)$ gives a linear extension.
We show the effect of applying $T_aT_bT_cT_dT_eT_f$ to the order ideal considered in Example~\ref{ex:a3-row}.
In each step, we indicate the element whose toggle we apply next in {\color{red}red}.
Notice that the outcome is the same order ideal we obtained by the three step process,
demonstrating Proposition~\ref{prop:row-toggles}. 
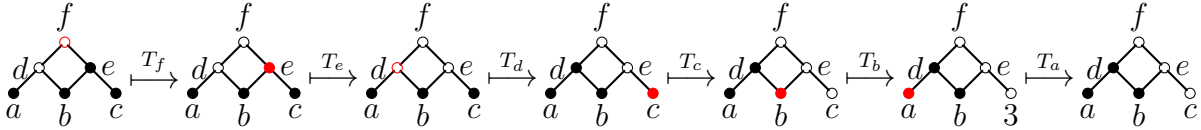
\begin{center}
\begin{tikzpicture}[scale=0.34]
\begin{scope}[shift={(0,-4)}]
\draw[thick] (-0.1, 1.9) -- (-0.9, 1.1);
\draw[thick] (0.1, 1.9) -- (0.9, 1.1);
\draw[thick] (-1.1, 0.9) -- (-1.9, 0.1);
\draw[thick] (-0.9, 0.9) -- (-0.1, 0.1);
\draw[thick] (0.9, 0.9) -- (0.1, 0.1);
\draw[thick] (1.1, 0.9) -- (1.9, 0.1);
\draw[red] (0,2) circle [radius=0.2];
\draw (-1,1) circle [radius=0.2];
\draw[fill] (1,1) circle [radius=0.2];
\draw[fill] (-2,0) circle [radius=0.2];
\draw[fill] (0,0) circle [radius=0.2];
\draw[fill] (2,0) circle [radius=0.2];
\node[above] at (0,2) {$f$};
\node[left] at (-1,1) {$d$};
\node[right] at (1,1) {$e$};
\node[below] at (-2,0) {$a$};
\node[below] at (0,0) {$b$};
\node[below] at (2,0) {$c$};
\end{scope}
\node at (3.5,-3) {$\stackrel{T_f}{\longmapsto}$};
\begin{scope}[shift={(7,-4)}]
\draw[thick] (-0.1, 1.9) -- (-0.9, 1.1);
\draw[thick] (0.1, 1.9) -- (0.9, 1.1);
\draw[thick] (-1.1, 0.9) -- (-1.9, 0.1);
\draw[thick] (-0.9, 0.9) -- (-0.1, 0.1);
\draw[thick] (0.9, 0.9) -- (0.1, 0.1);
\draw[thick] (1.1, 0.9) -- (1.9, 0.1);
\draw (0,2) circle [radius=0.2];
\draw (-1,1) circle [radius=0.2];
\draw[red,fill] (1,1) circle [radius=0.2];
\draw[fill] (-2,0) circle [radius=0.2];
\draw[fill] (0,0) circle [radius=0.2];
\draw[fill] (2,0) circle [radius=0.2];
\node[above] at (0,2) {$f$};
\node[left] at (-1,1) {$d$};
\node[right] at (1,1) {$e$};
\node[below] at (-2,0) {$a$};
\node[below] at (0,0) {$b$};
\node[below] at (2,0) {$c$};
\end{scope}
\node at (10.5,-3) {$\stackrel{T_e}{\longmapsto}$};
\begin{scope}[shift={(14,-4)}]
\draw[thick] (-0.1, 1.9) -- (-0.9, 1.1);
\draw[thick] (0.1, 1.9) -- (0.9, 1.1);
\draw[thick] (-1.1, 0.9) -- (-1.9, 0.1);
\draw[thick] (-0.9, 0.9) -- (-0.1, 0.1);
\draw[thick] (0.9, 0.9) -- (0.1, 0.1);
\draw[thick] (1.1, 0.9) -- (1.9, 0.1);
\draw (0,2) circle [radius=0.2];
\draw[red] (-1,1) circle [radius=0.2];
\draw (1,1) circle [radius=0.2];
\draw[fill] (-2,0) circle [radius=0.2];
\draw[fill] (0,0) circle [radius=0.2];
\draw[fill] (2,0) circle [radius=0.2];
\node[above] at (0,2) {$f$};
\node[left] at (-1,1) {$d$};
\node[right] at (1,1) {$e$};
\node[below] at (-2,0) {$a$};
\node[below] at (0,0) {$b$};
\node[below] at (2,0) {$c$};
\end{scope}
\node at (17.5,-3) {$\stackrel{T_d}{\longmapsto}$};
\begin{scope}[shift={(21,-4)}]
\draw[thick] (-0.1, 1.9) -- (-0.9, 1.1);
\draw[thick] (0.1, 1.9) -- (0.9, 1.1);
\draw[thick] (-1.1, 0.9) -- (-1.9, 0.1);
\draw[thick] (-0.9, 0.9) -- (-0.1, 0.1);
\draw[thick] (0.9, 0.9) -- (0.1, 0.1);
\draw[thick] (1.1, 0.9) -- (1.9, 0.1);
\draw (0,2) circle [radius=0.2];
\draw[fill] (-1,1) circle [radius=0.2];
\draw (1,1) circle [radius=0.2];
\draw[fill] (-2,0) circle [radius=0.2];
\draw[fill] (0,0) circle [radius=0.2];
\draw[red,fill] (2,0) circle [radius=0.2];
\node[above] at (0,2) {$f$};
\node[left] at (-1,1) {$d$};
\node[right] at (1,1) {$e$};
\node[below] at (-2,0) {$a$};
\node[below] at (0,0) {$b$};
\node[below] at (2,0) {$c$};
\end{scope}
\node at (24.5,-3) {$\stackrel{T_c}{\longmapsto}$};
\begin{scope}[shift={(28,-4)}]
\draw[thick] (-0.1, 1.9) -- (-0.9, 1.1);
\draw[thick] (0.1, 1.9) -- (0.9, 1.1);
\draw[thick] (-1.1, 0.9) -- (-1.9, 0.1);
\draw[thick] (-0.9, 0.9) -- (-0.1, 0.1);
\draw[thick] (0.9, 0.9) -- (0.1, 0.1);
\draw[thick] (1.1, 0.9) -- (1.9, 0.1);
\draw (0,2) circle [radius=0.2];
\draw[fill] (-1,1) circle [radius=0.2];
\draw (1,1) circle [radius=0.2];
\draw[fill] (-2,0) circle [radius=0.2];
\draw[red,fill] (0,0) circle [radius=0.2];
\draw (2,0) circle [radius=0.2];
\node[above] at (0,2) {$f$};
\node[left] at (-1,1) {$d$};
\node[right] at (1,1) {$e$};
\node[below] at (-2,0) {$a$};
\node[below] at (0,0) {$b$};
\node[below] at (2,0) {$c$};
\end{scope}
\node at (31.5,-3) {$\stackrel{T_b}{\longmapsto}$};
\begin{scope}[shift={(35,-4)}]
\draw[thick] (-0.1, 1.9) -- (-0.9, 1.1);
\draw[thick] (0.1, 1.9) -- (0.9, 1.1);
\draw[thick] (-1.1, 0.9) -- (-1.9, 0.1);
\draw[thick] (-0.9, 0.9) -- (-0.1, 0.1);
\draw[thick] (0.9, 0.9) -- (0.1, 0.1);
\draw[thick] (1.1, 0.9) -- (1.9, 0.1);
\draw (0,2) circle [radius=0.2];
\draw[fill] (-1,1) circle [radius=0.2];
\draw (1,1) circle [radius=0.2];
\draw[red,fill] (-2,0) circle [radius=0.2];
\draw[fill] (0,0) circle [radius=0.2];
\draw (2,0) circle [radius=0.2];
\node[above] at (0,2) {$f$};
\node[left] at (-1,1) {$d$};
\node[right] at (1,1) {$e$};
\node[below] at (-2,0) {$a$};
\node[below] at (0,0) {$b$};
\node[below] at (2,0) {3};
\end{scope}
\node at (38.5,-3) {$\stackrel{T_a}{\longmapsto}$};
\begin{scope}[shift={(42,-4)}]
\draw[thick] (-0.1, 1.9) -- (-0.9, 1.1);
\draw[thick] (0.1, 1.9) -- (0.9, 1.1);
\draw[thick] (-1.1, 0.9) -- (-1.9, 0.1);
\draw[thick] (-0.9, 0.9) -- (-0.1, 0.1);
\draw[thick] (0.9, 0.9) -- (0.1, 0.1);
\draw[thick] (1.1, 0.9) -- (1.9, 0.1);
\draw (0,2) circle [radius=0.2];
\draw[fill] (-1,1) circle [radius=0.2];
\draw (1,1) circle [radius=0.2];
\draw[fill] (-2,0) circle [radius=0.2];
\draw[fill] (0,0) circle [radius=0.2];
\draw (2,0) circle [radius=0.2];
\node[above] at (0,2) {$f$};
\node[left] at (-1,1) {$d$};
\node[right] at (1,1) {$e$};
\node[below] at (-2,0) {$a$};
\node[below] at (0,0) {$b$};
\node[below] at (2,0) {$c$};
\end{scope}
\end{tikzpicture}
\end{center}
\end{example}

\subsection{The antichain toggle group}
Toggling makes sense in a broader context, as formalized by Striker~\cite{strikergentog}. We can define \emph{antichain toggles} on $\cala(P)$, by replacing $\calj(P)$ with $\cala(P)$ in the definition.  Removing any element from an antichain always yields an antichain, giving a simpler second case.  
\begin{defn}[\cite{strikergentog}]\label{def:comb-tau}
Let $v\in P$.  Then the \textbf{antichain toggle} corresponding to $v$ is the map $\tau_v: \cala(P)\ra \cala(P)$
defined by
$$\tau_v(A)=\left\{\begin{array}{ll}
A\cup\{v\} &\text{if $v\not\in A$ and $A\cup\{v\}\in\cala(P)$,}\\
A\sm\{v\} &\text{if $v\in A$,}\\
A &\text{otherwise.}
\end{array}\right.$$
Let $\tog_\cala(P)$ denote the \textbf{toggle group} of $\cala(P)$, i.e., the subgroup of
$\frak{S}_{\cala (P)}$  (the symmetric group on $\cala (P)$) generated by the toggles $\{\tau_v : v\in P\}$.
\end{defn}
The first author constructed an explicit isomorphism between the toggle groups $\tog_\calj(P)$ and
$\tog_\cala(P)$~\cite{antichain-toggling}.
A consequence of this isomorphism is that antichain rowmotion is also a product of antichain toggles in an order specified by a linear extension, but starting at the bottom and moving upwards.

\begin{prop}[{\cite[Prop.~2.24]{antichain-toggling}}]\label{prop:row-toggles-anti}
For any linear extension $(x_1,x_2,\dots,x_n)$ of $P$,  antichain rowmotion is given by
$\rowA = \tau_{x_n} \cdots \tau_{x_2} \tau_{x_1}$.
\end{prop}

\begin{example}\label{ex:row-toggles-anti}
For the poset
$P=[2]\times[3]$, $(a,b,c,d,e,f)$ gives a linear extension.
We show the effect of applying $\tau_f \tau_e \tau_d \tau_c \tau_b \tau_a$ to the order ideal considered in Example~\ref{ex:a3-row}.
In each step, we indicate the element whose toggle we apply next in {\color{red}red}.
Notice that the outcome is the same order ideal we obtained by the three step process, demonstrating Proposition~\ref{prop:row-toggles-anti}.

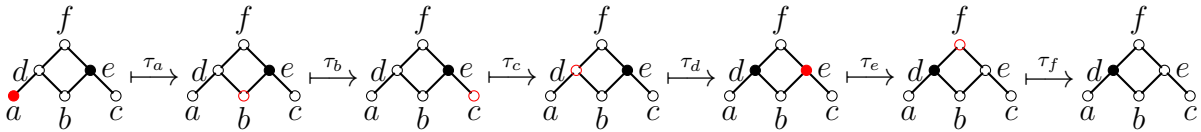
\begin{center}
\begin{tikzpicture}[scale=0.34]
\begin{scope}[shift={(0,-4)}]
\draw[thick] (-0.1, 1.9) -- (-0.9, 1.1);
\draw[thick] (0.1, 1.9) -- (0.9, 1.1);
\draw[thick] (-1.1, 0.9) -- (-1.9, 0.1);
\draw[thick] (-0.9, 0.9) -- (-0.1, 0.1);
\draw[thick] (0.9, 0.9) -- (0.1, 0.1);
\draw[thick] (1.1, 0.9) -- (1.9, 0.1);
\draw (0,2) circle [radius=0.2];
\draw (-1,1) circle [radius=0.2];
\draw[fill] (1,1) circle [radius=0.2];
\draw[red,fill] (-2,0) circle [radius=0.2];
\draw (0,0) circle [radius=0.2];
\draw (2,0) circle [radius=0.2];
\node[above] at (0,2) {$f$};
\node[left] at (-1,1) {$d$};
\node[right] at (1,1) {$e$};
\node[below] at (-2,0) {$a$};
\node[below] at (0,0) {$b$};
\node[below] at (2,0) {$c$};
\end{scope}
\node at (3.5,-3) {$\stackrel{\tau_a}{\longmapsto}$};
\begin{scope}[shift={(7,-4)}]
\draw[thick] (-0.1, 1.9) -- (-0.9, 1.1);
\draw[thick] (0.1, 1.9) -- (0.9, 1.1);
\draw[thick] (-1.1, 0.9) -- (-1.9, 0.1);
\draw[thick] (-0.9, 0.9) -- (-0.1, 0.1);
\draw[thick] (0.9, 0.9) -- (0.1, 0.1);
\draw[thick] (1.1, 0.9) -- (1.9, 0.1);
\draw (0,2) circle [radius=0.2];
\draw (-1,1) circle [radius=0.2];
\draw[fill] (1,1) circle [radius=0.2];
\draw (-2,0) circle [radius=0.2];
\draw[red] (0,0) circle [radius=0.2];
\draw (2,0) circle [radius=0.2];
\node[above] at (0,2) {$f$};
\node[left] at (-1,1) {$d$};
\node[right] at (1,1) {$e$};
\node[below] at (-2,0) {$a$};
\node[below] at (0,0) {$b$};
\node[below] at (2,0) {$c$};
\end{scope}
\node at (10.5,-3) {$\stackrel{\tau_b}{\longmapsto}$};
\begin{scope}[shift={(14,-4)}]
\draw[thick] (-0.1, 1.9) -- (-0.9, 1.1);
\draw[thick] (0.1, 1.9) -- (0.9, 1.1);
\draw[thick] (-1.1, 0.9) -- (-1.9, 0.1);
\draw[thick] (-0.9, 0.9) -- (-0.1, 0.1);
\draw[thick] (0.9, 0.9) -- (0.1, 0.1);
\draw[thick] (1.1, 0.9) -- (1.9, 0.1);
\draw (0,2) circle [radius=0.2];
\draw (-1,1) circle [radius=0.2];
\draw[fill] (1,1) circle [radius=0.2];
\draw (-2,0) circle [radius=0.2];
\draw (0,0) circle [radius=0.2];
\draw[red] (2,0) circle [radius=0.2];
\node[above] at (0,2) {$f$};
\node[left] at (-1,1) {$d$};
\node[right] at (1,1) {$e$};
\node[below] at (-2,0) {$a$};
\node[below] at (0,0) {$b$};
\node[below] at (2,0) {$c$};
\end{scope}
\node at (17.5,-3) {$\stackrel{\tau_c}{\longmapsto}$};
\begin{scope}[shift={(21,-4)}]
\draw[thick] (-0.1, 1.9) -- (-0.9, 1.1);
\draw[thick] (0.1, 1.9) -- (0.9, 1.1);
\draw[thick] (-1.1, 0.9) -- (-1.9, 0.1);
\draw[thick] (-0.9, 0.9) -- (-0.1, 0.1);
\draw[thick] (0.9, 0.9) -- (0.1, 0.1);
\draw[thick] (1.1, 0.9) -- (1.9, 0.1);
\draw (0,2) circle [radius=0.2];
\draw[red] (-1,1) circle [radius=0.2];
\draw[fill] (1,1) circle [radius=0.2];
\draw (-2,0) circle [radius=0.2];
\draw (0,0) circle [radius=0.2];
\draw (2,0) circle [radius=0.2];
\node[above] at (0,2) {$f$};
\node[left] at (-1,1) {$d$};
\node[right] at (1,1) {$e$};
\node[below] at (-2,0) {$a$};
\node[below] at (0,0) {$b$};
\node[below] at (2,0) {$c$};
\end{scope}
\node at (24.5,-3) {$\stackrel{\tau_d}{\longmapsto}$};
\begin{scope}[shift={(28,-4)}]
\draw[thick] (-0.1, 1.9) -- (-0.9, 1.1);
\draw[thick] (0.1, 1.9) -- (0.9, 1.1);
\draw[thick] (-1.1, 0.9) -- (-1.9, 0.1);
\draw[thick] (-0.9, 0.9) -- (-0.1, 0.1);
\draw[thick] (0.9, 0.9) -- (0.1, 0.1);
\draw[thick] (1.1, 0.9) -- (1.9, 0.1);
\draw (0,2) circle [radius=0.2];
\draw[fill] (-1,1) circle [radius=0.2];
\draw[red,fill] (1,1) circle [radius=0.2];
\draw (-2,0) circle [radius=0.2];
\draw (0,0) circle [radius=0.2];
\draw (2,0) circle [radius=0.2];
\node[above] at (0,2) {$f$};
\node[left] at (-1,1) {$d$};
\node[right] at (1,1) {$e$};
\node[below] at (-2,0) {$a$};
\node[below] at (0,0) {$b$};
\node[below] at (2,0) {$c$};
\end{scope}
\node at (31.5,-3) {$\stackrel{\tau_e}{\longmapsto}$};
\begin{scope}[shift={(35,-4)}]
\draw[thick] (-0.1, 1.9) -- (-0.9, 1.1);
\draw[thick] (0.1, 1.9) -- (0.9, 1.1);
\draw[thick] (-1.1, 0.9) -- (-1.9, 0.1);
\draw[thick] (-0.9, 0.9) -- (-0.1, 0.1);
\draw[thick] (0.9, 0.9) -- (0.1, 0.1);
\draw[thick] (1.1, 0.9) -- (1.9, 0.1);
\draw[red] (0,2) circle [radius=0.2];
\draw[fill] (-1,1) circle [radius=0.2];
\draw (1,1) circle [radius=0.2];
\draw (-2,0) circle [radius=0.2];
\draw (0,0) circle [radius=0.2];
\draw (2,0) circle [radius=0.2];
\node[above] at (0,2) {$f$};
\node[left] at (-1,1) {$d$};
\node[right] at (1,1) {$e$};
\node[below] at (-2,0) {$a$};
\node[below] at (0,0) {$b$};
\node[below] at (2,0) {$c$};
\end{scope}
\node at (38.5,-3) {$\stackrel{\tau_f}{\longmapsto}$};
\begin{scope}[shift={(42,-4)}]
\draw[thick] (-0.1, 1.9) -- (-0.9, 1.1);
\draw[thick] (0.1, 1.9) -- (0.9, 1.1);
\draw[thick] (-1.1, 0.9) -- (-1.9, 0.1);
\draw[thick] (-0.9, 0.9) -- (-0.1, 0.1);
\draw[thick] (0.9, 0.9) -- (0.1, 0.1);
\draw[thick] (1.1, 0.9) -- (1.9, 0.1);
\draw (0,2) circle [radius=0.2];
\draw[fill] (-1,1) circle [radius=0.2];
\draw (1,1) circle [radius=0.2];
\draw (-2,0) circle [radius=0.2];
\draw (0,0) circle [radius=0.2];
\draw (2,0) circle [radius=0.2];
\node[above] at (0,2) {$f$};
\node[left] at (-1,1) {$d$};
\node[right] at (1,1) {$e$};
\node[below] at (-2,0) {$a$};
\node[below] at (0,0) {$b$};
\node[below] at (2,0) {$c$};
\end{scope}
\end{tikzpicture}
\end{center}
\end{example}

\subsection{Piecewise-linear dynamics}
Now we generalize our actions from subsets of $P$ (i.e., $\{0,1 \}$-labelings of $P$)
to $\rr$-labelings of the elements of $P$; let $\rr^P:=\{f:P\rightarrow \rr \}$ denote the set of such labelings.
The toggling perspective allows us to extend these maps from the combinatorial realm
(on finite sets) to the piecewise-linear realm (polytopes whose vertices correspond to these
sets), and then lift to the birational realm by detropicalizing the operations~\cite{einpropp}.
The study of piecewise-linear dynamics begins with two polytopes introduced by Stanley~\cite{Sta86}, the \textit{order polytope} and the \textit{chain polytope} of $P$.
The vertices of these polytopes are the sets $\calf(P)$ of order filters and $\cala(P)$ of antichains (associating a subset of $P$ with its indicator function
labeling).  Einstein and Propp defined piecewise-linear toggle operations on the order
polytope that match the order-ideal toggle $T_v$ when restricted to the vertices (though here
we use order \emph{filters} instead of order ideals)~\cite[\S3,4]{einpropp}.

\begin{definition}[\cite{Sta86}]
Within $\rr^{P}$ the \textbf{order polytope} of $P$ is the set $\calo\calp(P)$ of labelings $f:P\ra [0,1]$ that are order-preserving (i.e., if $a\leq b$ in $P$, then $f(a)\leq f(b)$).
The \textbf{chain polytope} of $P$ is the set $\calc(P)$ of labelings $f:P\ra [0,1]$ such that the sum of the labels across every chain is at most 1.
\end{definition}

By associating a subset of $P$ with its indicator functions, the sets $\calf(P)$ and $\cala(P)$ describe the vertices of $\calo\calp(P)$ and $\calc(P)$ respectively~\cite{Sta86}.
Similarly, we can define an \textbf{order-reversing polytope} $\calo\calr(P)$ to consist of all labelings $f:P\ra [0,1]$ that are order-reversing (i.e., if $a\leq b$ in $P$, then $f(a)\geq f(b)$).  The vertices of $\calo\calr(P)$ are the order ideals of $P$.
To define toggles, there is no important difference in defining them over $\calo\calr(P)$ as opposed to $\calo\calp(P)$.  The piecewise-linear order toggles are typically defined on $\calo\calp(P)$.

\begin{defn}[\cite{einpropp}]
Given a finite poset $P$, let $\widehat{P}$ denote the poset $P$ with the addition of two
elements, $\widehat{0}$ and $\widehat{1}$, satisfying $\widehat{0}< v$ and $\widehat{1}>v$
for all $v\in P$.  
Let $v\in P$ and $f\in\calo\calp(P)$.
The \textbf{piecewise-linear order toggle} $T_v:\calo\calp(P)\ra \calo\calp(P)$ is
$$\big(T_v(f)\big)(x) =
\left\{\begin{array}{ll}
f(x) &\text{if }x\not=v\\
\max\limits_{y\lessdot v} f(y) + 
\min\limits_{y\gtrdot v} f(y) - f(v)
&\text{if }x=v\end{array}\right.
$$
where we set $f\big(\widehat{0}\big)=0$ and $f\big(\widehat{1}\big)=1$.
We use the notation $x\lessdot y$ to mean ``$y$ covers $x$'' and $x\gtrdot y$ to mean ``$x$
covers $y$''.  By using cover relations in $\widehat{P}$ we ensure that every element
of $P$ covers some element of $\widehat{P}$ and is covered by an element of $\widehat{P}$. 
The effect of this involution at $x=v$ is to replace the label at $v$ with the value obtained
by reflecting the allowable $\rr$-interval $\Big[\max\limits_{y\lessdot v} f(y), \min\limits_{y\gtrdot
v} f(y)\Big]$ through its midpoint.
\end{defn}

The first author defined the following generalization of antichain toggles to the chain
polytope $\calc(P)$~\cite[\S3]{antichain-toggling}, which matches our earlier definition
of $\tau_v$ when restricted to the vertices $\cala(P)$. 

\begin{defn}[\cite{antichain-toggling}]\label{def:PL-ant-tog}
For $v\in P$, let $\mc_v(P)$ denote the set of all maximal chains of $P$ through $v$.
The \textbf{piecewise-linear antichain toggle}
(or {\bf chain polytope toggle}) $\tau_v:\calc(P)\ra \calc(P)$ is
$$\big(\tau_v(g)\big)(x) =
\left\{\begin{array}{ll}
1-\max\left\{ \left.\sum\limits_{i=1}^k g(y_i) \right| (y_1,\dots,y_k)\in \mc_v(P) \right\}
&\text{if }x=v\\
g(x) &\text{if }x\not=v
\end{array}\right..
$$
\end{defn}

We have the following generalizations of complementation $\Theta$,
down-transfer $\down$, and up-transfer $\up$ (and their inverses)
to the piecewise-linear realm.
In fact the term ``down-transfer'' was chosen as it is equivalent here to
Stanley's transfer map, used to transfer properties (such as volume formulas) between $\calc(P)$
and $\calo\calp(P)$~\cite{Sta86}.

\begin{defn}[{\cite[\S4]{einpropp}}]
The maps $\Theta: \rr^P \ra \rr^P$, $\down: \calo\calp(P) \ra \calc(P)$,
$\up: \calo\calr(P) \ra \calc(P)$, and their inverses are given as follows.
\begin{align*}
    (\Theta f)(x) &= 1-f(x)\\
    (\down f)(x) &= f(x)-{\max\limits_{y\lessdot x} f(y)} \quad \big(\text{with } f\big(\widehat{0}\big)=0\big)\\
    (\up f)(x) &= 
    f(x) - {\max\limits_{y\gtrdot x} f(y)} \quad \big(\text{with } f\big(\widehat{1}\big)=0\big)\\
    \left(\down^{-1}f\right)(x) &=
    \max\left\{ f(y_1)+f(y_2)+\cdots +f(y_k):
    \widehat{0}\lessdot y_1\lessdot y_2 \lessdot \cdots \lessdot y_k=x \right\}\\
    &=f(x)+ \max\limits_{y\lessdot x}\left(\down^{-1}f\right)(y)\\
    \left(\up^{-1}f\right)(x) &=
    \max\left\{ f(y_1)+ f(y_2)+\cdots +f(y_k):x=y_1
\lessdot y_2\lessdot \cdots \lessdot y_k\lessdot \widehat{1}\right\}\\
    &=f(x)+ \max\limits_{y\gtrdot x}\left(\up^{-1}f\right)(y)
\end{align*}
\end{defn}

For any linear extension $(x_1,x_2,\dots,x_n)$ of $P$,
Einstein and Propp showed that piecewise-linear order rowmotion
defined as $T_{x_1}T_{x_2}\cdots T_{x_n}$
is equivalent to the composition $\Theta \circ \up^{-1} \circ \down$ (a consequence of
their proof at the birational level). 
The first author showed that piecewise-linear antichain rowmotion can be defined either as
$\tau_{x_n}\cdots \tau_{x_2}\tau_{x_1}$ or as
$\down\circ\Theta \circ \up^{-1}$, as in the combinatorial realm.
For details, see~\cite[\S4, \S6]{einpropp} and~\cite[\S3.3, \S3.4]{antichain-toggling}.

\subsection{Birational dynamics}

We now detropicalize the piecewise-linear order toggles to birational
toggles over an arbitrary field $\kk$ of characteristic zero in the usual way: replacing
the $\max$ operation with addition, 
addition with multiplication, 
subtraction with division, and
the additive identity 0 with the multiplicative identity 1.
Additionally, we replace 1 with a generic fixed constant $C\in\kk$.
(The definition
in~\cite{einpropp} is slightly more general than we need here; set $\alpha=1$ and
$\omega=C$ in their version to get ours.  In this paper, we will be primarily interested in
birational antichain rowmotion, in which the two arbitrary constants $\alpha$ and $\omega$
would always appear together as $\omega/\alpha$; thus, one constant $C$ is sufficient.) 

\begin{defn}[{\cite[Definition~5.1]{einpropp}}]
Let $\kk^P:=\{f:P\rightarrow \kk \}$ be the set of $\kk$-labelings of the elements of $P$.
For $v\in P$, the \textbf{birational order toggle} at $v$ is the
birational map
$T_v: \kk^P \dra \kk^P$ given by
\[
\ds \big(T_v(f)\big)(x) =
\left\{\begin{array}{ll}
f(x) &\text{if }x\not=v\\\vspace{-7 pt}\\
\dfrac{\sum\limits_{y\in\widehat{P},y\lessdot v} f(y)}
{f(v) \sum\limits_{y\in\widehat{P},y\gtrdot v} \frac{1}{f(y)}}
&\text{if }x=v\end{array}\right.
\]
where we set $f\big(\widehat{0}\big)=1$ and $f\big(\widehat{1}\big)=C$.
\end{defn}
The birational order toggles $\{T_v: v\in P\}$
generate the {\bf birational order toggle group}, a subgroup of the group of
birational automorphisms of $\kk^{P}$, which we denote $\btog_O(P)$.
The following shows that basic properties of order-ideal toggles lift to the birational realm.

\begin{prop}[\cite{einpropp,GrRo15}]\label{prop:BTogO-toggle-inv-commute}
Each toggle $T_v$ is an involution (i.e., $T_v^2$ is the identity).  Two toggles $T_u,T_v$ commute if and only if neither $u$ nor $v$ covers the other.
\end{prop}

\begin{defn}[{\cite[Definition~5.2]{einpropp}}]
Let $(x_1,x_2,\dots,x_n)$ be any linear extension of $P$.
The birational analogue of order filter rowmotion, which we will call
\textbf{birational order rowmotion} (or \textbf{BOR-motion}),
is $\BOR= T_{x_1} T_{x_2} \cdots T_{x_n}$.  (Compare with
Proposition~\ref{prop:row-toggles}.)
\end{defn}
An annoying technicality here is that for some choices of labels it is possible that these
maps could lead to division by zero.  But for ``generic'' choices of labels (say with respect to the
Zariski topology) iterates of this map will be well-defined. This is discussed
carefully in \cite[\S1]{GrRo16} and \cite[\S3]{GrRo15}.  Alternatively, Einstein and Propp
make the choice to consider only \textit{positive} labelings, i.e., $(\rr_{>
0})^{\widehat{P}}$.

\subsection{Birational transfer maps}
By detropicalizing the operations in the transfer maps of the previous subsection, we get
their birational analogues.  Einstein and Propp prove~\cite[\S6]{einpropp} that, under the definitions below,
$\BOR=\Theta\circ\up^{-1}\circ\down$. 
These maps are composed in the order which lifts 
combinatorial rowmotion on \textit{order filters}.

\begin{defn}[{\cite[\S6]{einpropp}}]\label{def:bir-trans}
Let $f\in \kk^P$ and $x\in P$.  We define the following birational maps. 
Again, we call $\Theta$ {\bf complement}, $\down$ {\bf down transfer},
and $\up$ {\bf up transfer}.

\begin{align*}
    (\Theta f)(x) &= \frac{C}{f(x)}\\
    (\down f)(x) &= \frac{f(x)}{\sum\limits_{y\lessdot x} f(y)} \quad \big(\text{with } f\big(\widehat{0}\big)=1\big)\\
    (\up f)(x) &= 
    \frac{f(x)}{\sum\limits_{y\gtrdot x} f(y)} \quad \big(\text{with } f\big(\widehat{1}\big)=1\big)\\
    \left(\down^{-1}f\right)(x) &=
    \sum_{\widehat{0}\lessdot y_1\lessdot y_2 \lessdot \cdots \lessdot y_k=x} f(y_1)f(y_2)\cdots f(y_k)
            =f(x) \sum\limits_{y\lessdot x}\left(\down^{-1}f\right)(y)\\
    \left(\up^{-1}f\right)(x) &=
    \sum_{x=y_1 \lessdot y_2\lessdot \cdots \lessdot y_k\lessdot \widehat{1}} f(y_1)f(y_2)\cdots f(y_k)
              =f(x) \sum\limits_{y\gtrdot x}\left(\up^{-1}f\right)(y).
\end{align*}

\end{defn}

We use the same symbols in each realm (combinatorial, piecewise-linear, birational, and noncommutative), allowing context to clarify which is meant.
Examples of each map $\down$, $\down^{-1}$, $\up$, and $\up^{-1}$ are given in Figure~\ref{fig:transfer-example}.

\begin{figure}
\begin{tikzpicture}[xscale=13/9,yscale=8/9]
\begin{scope}
\draw[thick] (-0.3, 1.7) -- (-0.7, 1.3);
\draw[thick] (0.3, 1.7) -- (0.7, 1.3);
\draw[thick] (-1.3, 0.7) -- (-1.7, 0.3);
\draw[thick] (-0.7, 0.7) -- (-0.3, 0.3);
\draw[thick] (0.7, 0.7) -- (0.3, 0.3);
\draw[thick] (1.3, 0.7) -- (1.7, 0.3);
\node at (0,2) {$  z  $};
\node at (-1,1) {$  x  $};
\node at (1,1) {$  y  $};
\node at (-2,0) {$  u  $};
\node at (0,0) {$  v  $};
\node at (2,0) {$  w  $};
\node at (2.75,1) {$\stackrel{\down}{\longmapsto}$};
\end{scope}
\begin{scope}[shift={(5.5,0)}]
\draw[thick] (-0.3, 1.7) -- (-0.7, 1.3);
\draw[thick] (0.3, 1.7) -- (0.7, 1.3);
\draw[thick] (-1.3, 0.7) -- (-1.7, 0.3);
\draw[thick] (-0.7, 0.7) -- (-0.3, 0.3);
\draw[thick] (0.7, 0.7) -- (0.3, 0.3);
\draw[thick] (1.3, 0.7) -- (1.7, 0.3);
\node at (0,2) {\large$  \frac{z}{x+y}  $};
\node at (-1,1) {\large$  \frac{x}{u+v}  $};
\node at (1,1) {\large$  \frac{y}{v+w}  $};
\node at (-2,0) {$  u  $};
\node at (0,0) {$  v  $};
\node at (2,0) {$  w  $};
\end{scope}
\begin{scope}[shift={(0,-3.5)}]
\draw[thick] (-0.3, 1.7) -- (-0.7, 1.3);
\draw[thick] (0.3, 1.7) -- (0.7, 1.3);
\draw[thick] (-1.3, 0.7) -- (-1.7, 0.3);
\draw[thick] (-0.7, 0.7) -- (-0.3, 0.3);
\draw[thick] (0.7, 0.7) -- (0.3, 0.3);
\draw[thick] (1.3, 0.7) -- (1.7, 0.3);
\node at (0,2) {$  z(ux+vx+vy+wy)  $};
\node at (-1,1) {$  x(u+v)  $};
\node at (1,1) {$  y(v+w)  $};
\node at (-2,0) {$  u  $};
\node at (0,0) {$  v  $};
\node at (2,0) {$  w  $};
\node at (2.75,1) {$\stackrel{\down^{-1}}{\longmapsfrom}$};
\end{scope}
\begin{scope}[shift={(5.5,-3.5)}]
\draw[thick] (-0.3, 1.7) -- (-0.7, 1.3);
\draw[thick] (0.3, 1.7) -- (0.7, 1.3);
\draw[thick] (-1.3, 0.7) -- (-1.7, 0.3);
\draw[thick] (-0.7, 0.7) -- (-0.3, 0.3);
\draw[thick] (0.7, 0.7) -- (0.3, 0.3);
\draw[thick] (1.3, 0.7) -- (1.7, 0.3);
\node at (0,2) {$  z  $};
\node at (-1,1) {$  x  $};
\node at (1,1) {$  y  $};
\node at (-2,0) {$  u  $};
\node at (0,0) {$  v  $};
\node at (2,0) {$  w  $};
\end{scope}
\begin{scope}[shift={(0,-7)}]
\draw[thick] (-0.3, 1.7) -- (-0.7, 1.3);
\draw[thick] (0.3, 1.7) -- (0.7, 1.3);
\draw[thick] (-1.3, 0.7) -- (-1.7, 0.3);
\draw[thick] (-0.7, 0.7) -- (-0.3, 0.3);
\draw[thick] (0.7, 0.7) -- (0.3, 0.3);
\draw[thick] (1.3, 0.7) -- (1.7, 0.3);
\node at (0,2) {$  z  $};
\node at (-1,1) {$  x  $};
\node at (1,1) {$  y  $};
\node at (-2,0) {$  u  $};
\node at (0,0) {$  v  $};
\node at (2,0) {$  w  $};
\node at (2.75,1) {$\stackrel{\up}{\longmapsto}$};
\end{scope}
\begin{scope}[shift={(5.5,-7)}]
\draw[thick] (-0.3, 1.7) -- (-0.7, 1.3);
\draw[thick] (0.3, 1.7) -- (0.7, 1.3);
\draw[thick] (-1.3, 0.7) -- (-1.7, 0.3);
\draw[thick] (-0.7, 0.7) -- (-0.3, 0.3);
\draw[thick] (0.7, 0.7) -- (0.3, 0.3);
\draw[thick] (1.3, 0.7) -- (1.7, 0.3);
\node at (0,2) {$  z  $};
\node at (-1,1) {\large$  \frac{x}{z}  $};
\node at (1,1) {\large$  \frac{y}{z}  $};
\node at (-2,0) {\large$  \frac{u}{x}  $};
\node at (0,0) {\large$  \frac{v}{x+y}  $};
\node at (2,0) {\large$  \frac{w}{y}  $};
\end{scope}
\begin{scope}[shift={(0,-10.5)}]
\draw[thick] (-0.3, 1.7) -- (-0.7, 1.3);
\draw[thick] (0.3, 1.7) -- (0.7, 1.3);
\draw[thick] (-1.3, 0.7) -- (-1.7, 0.3);
\draw[thick] (-0.7, 0.7) -- (-0.3, 0.3);
\draw[thick] (0.7, 0.7) -- (0.3, 0.3);
\draw[thick] (1.3, 0.7) -- (1.7, 0.3);
\node at (0,2) {$  z  $};
\node at (-1,1) {$  xz  $};
\node at (1,1) {$  yz  $};
\node at (-2,0) {$  uxz  $};
\node at (0,0) {$  v(x+y)z $};
\node at (2,0) {$  wyz  $};
\node at (2.75,1) {$\stackrel{\up^{-1}}{\longmapsfrom}$};
\end{scope}
\begin{scope}[shift={(5.5,-10.5)}]
\draw[thick] (-0.3, 1.7) -- (-0.7, 1.3);
\draw[thick] (0.3, 1.7) -- (0.7, 1.3);
\draw[thick] (-1.3, 0.7) -- (-1.7, 0.3);
\draw[thick] (-0.7, 0.7) -- (-0.3, 0.3);
\draw[thick] (0.7, 0.7) -- (0.3, 0.3);
\draw[thick] (1.3, 0.7) -- (1.7, 0.3);
\node at (0,2) {$  z  $};
\node at (-1,1) {$  x  $};
\node at (1,1) {$  y  $};
\node at (-2,0) {$  u  $};
\node at (0,0) {$  v  $};
\node at (2,0) {$  w  $};
\end{scope}
\end{tikzpicture}
\caption{An example of each of the birational
maps $\down$, $\down^{-1}$, $\up$, and $\up^{-1}$ on the positive root poset $\Phi^+(A_3)$.}
\label{fig:transfer-example}
\end{figure}
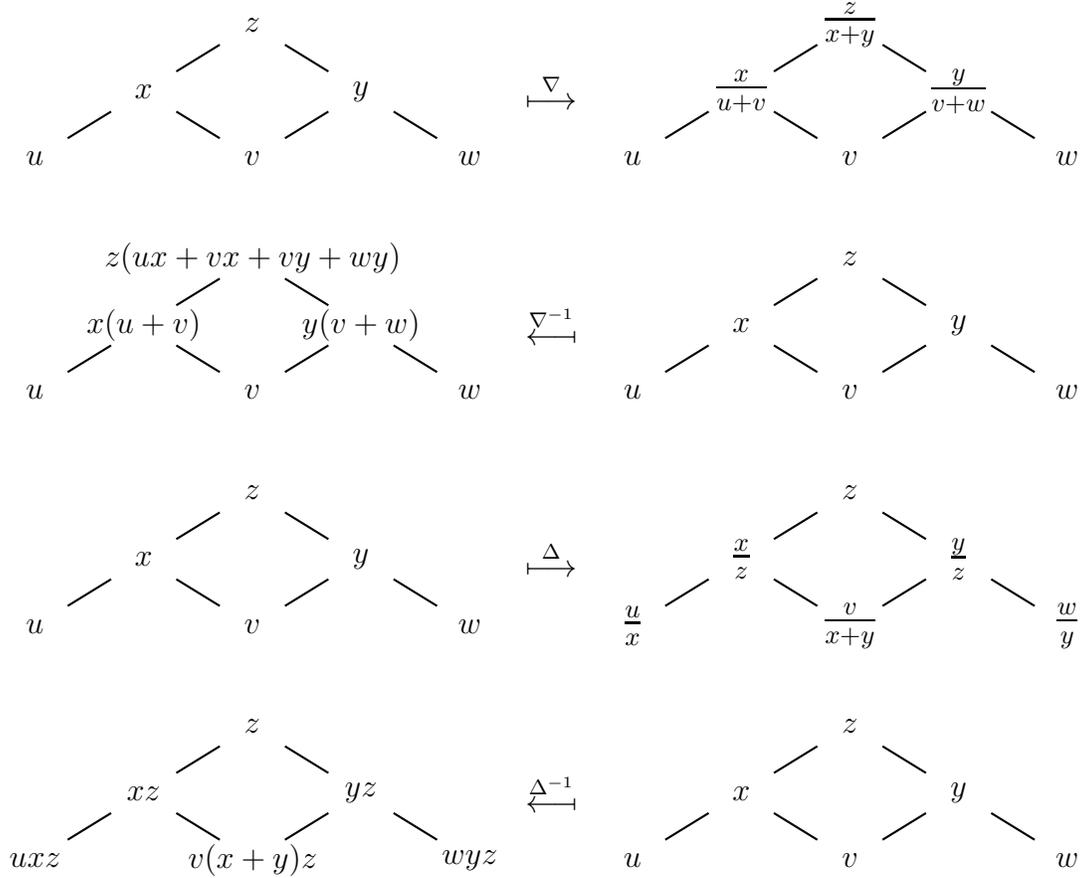

\section{Birational antichain toggling and rowmotion}\label{sec:batar}
\subsection{The birational antichain toggle group}
Now we will combine the different generalizations of toggling and study a new birational analogue of antichain toggling and rowmotion.
Again, we fix a field $\kk$ of characteristic zero.
In the combinatorial realm, we could define rowmotion on $\calj(P)$ or $\cala(P)$ either as
a composition of the three maps 
$\Theta$, $\up^{-1}$, and $\down$ (as is commonly done)
or in terms of toggles, since these approaches are proven equivalent.
BOR is usually defined in terms of toggles, and we take the analogous approach to defining
BAR below, lifting the definition in 
Proposition~\ref{prop:row-toggles-anti}.

\begin{defn}
For $v\in P$ and $g\in \kk^P$, set
$$\Upsilon_v g=
\sum\limits_{(y_1,\dots,y_k)\in \mc_v(P)} g(y_1) \cdots g(y_k),$$
where we recall $\mc_v(P)$ is the set of all maximal chains of $P$ through $v$.
\end{defn}

\begin{defn}
Let $v\in P$.  The \textbf{birational antichain toggle} is the birational map $\tau_v:\kk^P \dra \kk^P$ defined as follows:

$$\big(\tau_v(g)\big)(x) = \left\{\begin{array}{ll}
\dfrac{C}{\sum\limits_{(y_1,\dots,y_k)\in \mc_v(P)} g(y_1) \cdots g(y_k)} = \dfrac{C}{\Upsilon_v g} &\text{if $x=v$}\\\vspace{-7 pt}\\
g(x) &\text{if $x\not=v$}
\end{array}\right.
$$
\end{defn}

This definition is what is obtained from Definition~\ref{def:PL-ant-tog}
through detropicalization of operations.
As with antichain toggles in the combinatorial and piecewise-linear 
realms~\cite{strikergentog,antichain-toggling}, birational antichain toggles do not commute
as frequently as order toggles. 

\begin{prop}\label{prop:basic-tau-prop}
Let $u,v\in P$.
\begin{enumerate}
\item Each toggle $\tau_v$ is an involution, i.e., $\tau_v^2$ is the identity.
\item If $u\parallel v$ (i.e., $u$ and $v$ are incomparable), then $\tau_u \tau_v = \tau_v \tau_u$.
\end{enumerate}
\end{prop}

\begin{proof}
(1)
Let $g\in \kk^P$ be a generic labeling.
To show $\tau_v$ is an involution, we wish to show
$\tau_v^2(g)=g$.
Since $\tau_v$ can only change the label at $v$, we need only show
$\big(\tau_v^2(g)\big)(v)=g(v)$.
Every chain in $\mc_v(P)$ can be split into segments: below $v$, $v$ itself, and above $v$.  As we can take the sums of products on each segment,
\begin{equation}\label{eq:alt_tau_e}
\big(\tau_v(g)\big)(v) =
\frac{C}{\left(\sum\limits_{u\lessdot v}
(\down^{-1}g)(v)\right) g(v) \left(\sum\limits_{u\gtrdot v}
(\up^{-1}g)(v)\right)}
\end{equation}
for any $g\in\kk^P$.
In Eq.~(\ref{eq:alt_tau_e}), we regard
$\sum\limits_{u\lessdot v}
(\down^{-1}g)(v)=1$ if $v$ is minimal in $P$ and likewise $\sum\limits_{u\gtrdot v}
(\up^{-1}g)(v)=1$ if $v$ is maximal in $P$
(since the sums are nonempty when working in $\widehat{P}$).

Using Eq.~(\ref{eq:alt_tau_e}),
\begin{align*}
    \big(\tau_v^2(g)\big)(v) &=
    \frac{C}{\left(\sum\limits_{u\lessdot v}
\big(\down^{-1}\tau_v(g)\big)(v)\right) \big(\tau_v(g)\big)(v) \left(\sum\limits_{u\gtrdot v}
\big(\up^{-1}\tau_v(g)\big)(v)\right)}\\
&=\frac{C}{\left(\sum\limits_{u\lessdot v}
(\down^{-1}g)(v)\right) \frac{C}{\left(\sum\limits_{u\lessdot v}
(\down^{-1}g)(v)\right) g(v) \left(\sum\limits_{u\gtrdot v}
(\up^{-1}g)(v)\right)} \left(\sum\limits_{u\gtrdot v}
(\up^{-1}g)(v)\right)}\\
&=g(v).
\end{align*}

(2) 
Suppose $u \parallel v$.
Only the label of $u$ can be changed by $\tau_u$ and only the label of $v$ can be changed by $\tau_v$.
No chain contains both $u$ and $v$, so the label of $u$ has no effect on what $\tau_v$ does and the label of $v$ has no effect on what $\tau_u$ does.  So $\tau_u\tau_v=\tau_v\tau_u$.

\end{proof}

Proposition~\ref{prop:basic-tau-prop}
shows that the following definition is
well-defined, since any two linear extensions of a
poset differ by a sequence of transpositions of incomparable elements~\cite{etienne-84}.

\begin{defn}\label{def:BAR}
Let $(x_1,x_2,\dots,x_n)$ be any linear extension of a finite poset $P$.  Then the birational
map $\BAR=\tau_{x_n}\cdots \tau_{x_2} \tau_{x_1}$, i.e., toggling once at each element of $P$
from bottom to top, is called \textbf{birational antichain rowmotion (BAR-motion)}.
\end{defn}

\begin{ex}\label{ex:2x3}
Consider the poset $P=[2]\times[3]$ below, with the generic labeling $g\in\kk^P$ by 
$u,v,w,x,y,z\in\kk$. 
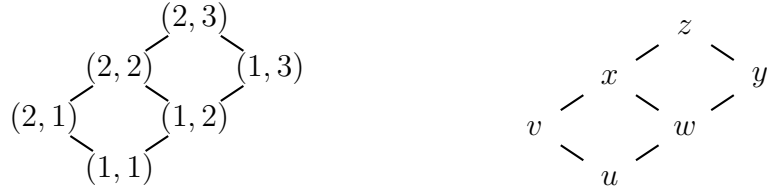
\begin{center}
\begin{tikzpicture}[yscale=2/3]
\node at (0,0) {$(1,1)$};
\node at (-1,1) {$(2,1)$};
\node at (1,1) {$(1,2)$};
\node at (0,2) {$(2,2)$};
\node at (2,2) {$(1,3)$};
\node at (1,3) {$(2,3)$};
\draw[thick] (-0.35,0.35) -- (-0.65,0.65);
\draw[thick] (0.35,0.35) -- (0.65,0.65);
\draw[thick] (1.35,1.35) -- (1.65,1.65);
\draw[thick] (0.35,2.35) -- (0.65,2.65);
\draw[thick] (-0.35,1.65) -- (-0.65,1.35);
\draw[thick] (0.35,1.65) -- (0.65,1.35);
\draw[thick] (1.35,2.65) -- (1.65,2.35);
\end{tikzpicture}
\hskip 1in
\begin{tikzpicture}[yscale=2/3]
\node at (0,0) {$u$};
\node at (-1,1) {$v$};
\node at (1,1) {$w$};
\node at (0,2) {$x$};
\node at (2,2) {$y$};
\node at (1,3) {$z$};
\draw[thick] (-0.35,0.35) -- (-0.65,0.65);
\draw[thick] (0.35,0.35) -- (0.65,0.65);
\draw[thick] (1.35,1.35) -- (1.65,1.65);
\draw[thick] (0.35,2.35) -- (0.65,2.65);
\draw[thick] (-0.35,1.65) -- (-0.65,1.35);
\draw[thick] (0.35,1.65) -- (0.65,1.35);
\draw[thick] (1.35,2.65) -- (1.65,2.35);
\end{tikzpicture}
\end{center}



To compute $\BAR$ along the linear extension $(1,1),(2,1),(1,2),(2,2),(1,3),(2,3)$, 
we first toggle at $(1,1)$.  There are three maximal chains through this bottom element:

\begin{itemize}
\item $(1,1)\lessdot(2,1)\lessdot(2,2)\lessdot(2,3)$ with product of labels $uvxz$
\item $(1,1)\lessdot(1,2)\lessdot(2,2)\lessdot(2,3)$ with product of labels $uwxz$
\item $(1,1)\lessdot(1,2)\lessdot(1,3)\lessdot(2,3)$ with product of labels $uwyz$
\end{itemize}

For $\Upsilon_{(1,1)}g$, we add up the products of the labels on these three maximal chains, and get $uvxz+uwxz+uwyz=u(vx+wx+wy)z$.
Then to apply the toggle $\tau_{(1,1)}$,
we change the label of $(1,1)$ from $u$ to
$\frac{C}{u(vx+wx+wy)z}$.

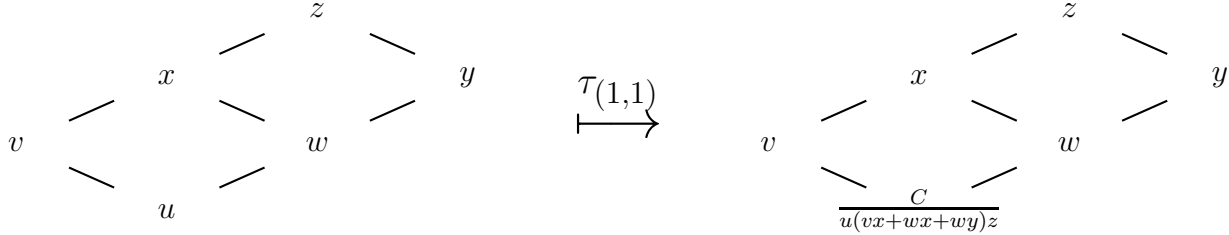
\begin{center}
\begin{tikzpicture}[xscale=2,yscale=8/9]
\node at (0,0) {$u$};
\node at (-1,1) {$v$};
\node at (1,1) {$w$};
\node at (0,2) {$x$};
\node at (2,2) {$y$};
\node at (1,3) {$z$};
\draw[thick] (-0.35,0.35) -- (-0.65,0.65);
\draw[thick] (0.35,0.35) -- (0.65,0.65);
\draw[thick] (1.35,1.35) -- (1.65,1.65);
\draw[thick] (0.35,2.35) -- (0.65,2.65);
\draw[thick] (-0.35,1.65) -- (-0.65,1.35);
\draw[thick] (0.35,1.65) -- (0.65,1.35);
\draw[thick] (1.35,2.65) -- (1.65,2.35);
\node at (3,1.6) {{\LARGE$\stackrel{\tau_{(1,1)}}{\longmapsto}$}};
\begin{scope}[shift={(5,0)}]
\node at (0,0) {$\frac{C}{u(vx+wx+wy)z}$};
\node at (-1,1) {$v$};
\node at (1,1) {$w$};
\node at (0,2) {$x$};
\node at (2,2) {$y$};
\node at (1,3) {$z$};
\draw[thick] (-0.35,0.35) -- (-0.65,0.65);
\draw[thick] (0.35,0.35) -- (0.65,0.65);
\draw[thick] (1.35,1.35) -- (1.65,1.65);
\draw[thick] (0.35,2.35) -- (0.65,2.65);
\draw[thick] (-0.35,1.65) -- (-0.65,1.35);
\draw[thick] (0.35,1.65) -- (0.65,1.35);
\draw[thick] (1.35,2.65) -- (1.65,2.35);
\end{scope}
\end{tikzpicture}
\end{center}

Now we apply $\tau_{(2,1)}$ to $\tau_{(1,1)}g$ (our above result).  There is only one maximal chain through $(2,1)$.  The product of labels along that maximal chain is $\Upsilon_{(2,1)}(\tau_{(1,1)}g)=\frac{C}{u(vx+wx+wy)z}vxz$.  Thus we change the label of $(2,1)$
from $v$ to $\frac{C}{\frac{C}{u(vx+wx+wy)z}vxz}=\frac{u(vx+wx+wy)}{vx}$.

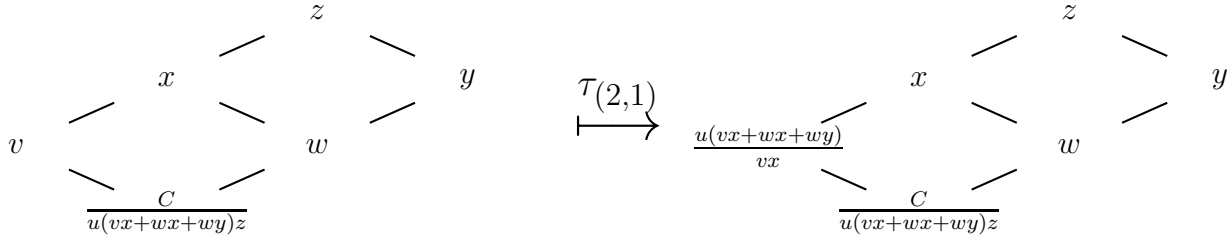
\begin{center}
\begin{tikzpicture}[xscale=2,yscale=8/9]
\node at (0,0) {$\frac{C}{u(vx+wx+wy)z}$};
\node at (-1,1) {$v$};
\node at (1,1) {$w$};
\node at (0,2) {$x$};
\node at (2,2) {$y$};
\node at (1,3) {$z$};
\draw[thick] (-0.35,0.35) -- (-0.65,0.65);
\draw[thick] (0.35,0.35) -- (0.65,0.65);
\draw[thick] (1.35,1.35) -- (1.65,1.65);
\draw[thick] (0.35,2.35) -- (0.65,2.65);
\draw[thick] (-0.35,1.65) -- (-0.65,1.35);
\draw[thick] (0.35,1.65) -- (0.65,1.35);
\draw[thick] (1.35,2.65) -- (1.65,2.35);
\node at (3,1.6) {{\LARGE$\stackrel{\tau_{(2,1)}}{\longmapsto}$}};
\begin{scope}[shift={(5,0)}]
\node at (0,0) {$\frac{C}{u(vx+wx+wy)z}$};
\node at (-1,1) {$\frac{u(vx+wx+wy)}{vx}$};
\node at (1,1) {$w$};
\node at (0,2) {$x$};
\node at (2,2) {$y$};
\node at (1,3) {$z$};
\draw[thick] (-0.35,0.35) -- (-0.65,0.65);
\draw[thick] (0.35,0.35) -- (0.65,0.65);
\draw[thick] (1.35,1.35) -- (1.65,1.65);
\draw[thick] (0.35,2.35) -- (0.65,2.65);
\draw[thick] (-0.35,1.65) -- (-0.65,1.35);
\draw[thick] (0.35,1.65) -- (0.65,1.35);
\draw[thick] (1.35,2.65) -- (1.65,2.35);
\end{scope}
\end{tikzpicture}
\end{center}

Next there are two maximal chains through $(1,2)$, one of which goes through the element
labeled $x$ and the other through the element labeled $y$.  Dividing $C$ by the sum of the
products of the labels for both chains gives $$\frac{C}{\frac{C}{u(vx+wx+wy)z} w (x+y) z} =
\frac{u(vx+wx+wy)}{w(x+y)}
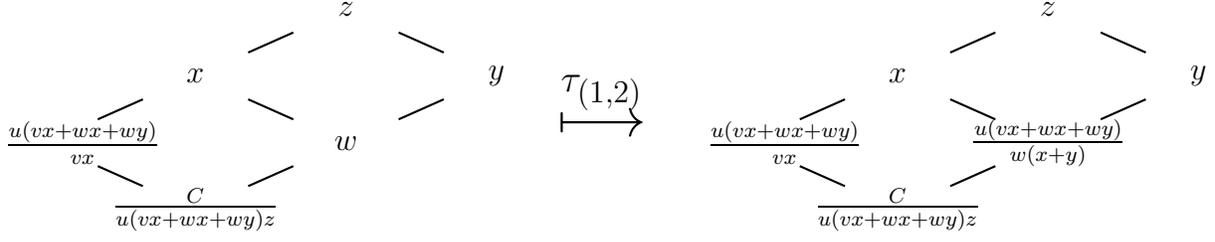
$$ 

\begin{center}
\begin{tikzpicture}[xscale=2,yscale=8/9]
\node at (0,0) {$\frac{C}{u(vx+wx+wy)z}$};
\node at (-0.75,1) {$\frac{u(vx+wx+wy)}{vx}$};
\node at (1,1) {$w$};
\node at (0,2) {$x$};
\node at (2,2) {$y$};
\node at (1,3) {$z$};
\draw[thick] (-0.35,0.35) -- (-0.65,0.65);
\draw[thick] (0.35,0.35) -- (0.65,0.65);
\draw[thick] (1.35,1.35) -- (1.65,1.65);
\draw[thick] (0.35,2.35) -- (0.65,2.65);
\draw[thick] (-0.35,1.65) -- (-0.65,1.35);
\draw[thick] (0.35,1.65) -- (0.65,1.35);
\draw[thick] (1.35,2.65) -- (1.65,2.35);
\node at (2.7,1.6) {{\LARGE$\stackrel{\tau_{(1,2)}}{\longmapsto}$}};
\begin{scope}[shift={(14/3,0)}]
\node at (0,0) {$\frac{C}{u(vx+wx+wy)z}$};
\node at (-0.75,1) {$\frac{u(vx+wx+wy)}{vx}$};
\node at (1,1) {$\frac{u(vx+wx+wy)}{w(x+y)}$};
\node at (0,2) {$x$};
\node at (2,2) {$y$};
\node at (1,3) {$z$};
\draw[thick] (-0.35,0.35) -- (-0.65,0.65);
\draw[thick] (0.35,0.35) -- (0.65,0.65);
\draw[thick] (1.35,1.35) -- (1.65,1.65);
\draw[thick] (0.35,2.35) -- (0.65,2.65);
\draw[thick] (-0.35,1.65) -- (-0.65,1.35);
\draw[thick] (0.35,1.65) -- (0.65,1.35);
\draw[thick] (1.35,2.65) -- (1.65,2.35);
\end{scope}
\end{tikzpicture}
\end{center}

Similarly, the remaining 3 toggles give

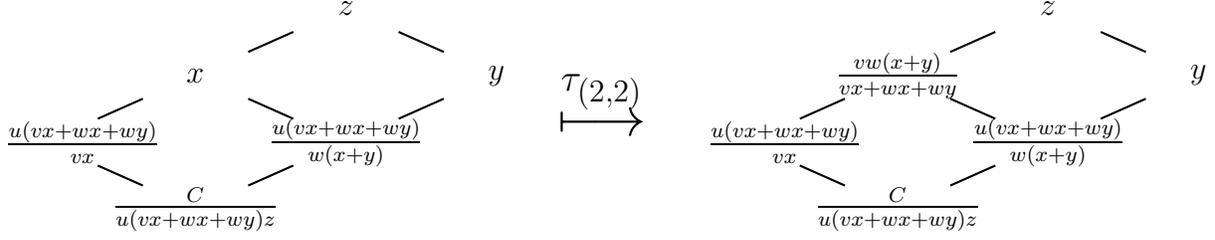
\begin{center}
\begin{tikzpicture}[xscale=2,yscale=8/9]
\node at (0,0) {$\frac{C}{u(vx+wx+wy)z}$};
\node at (-0.75,1) {$\frac{u(vx+wx+wy)}{vx}$};
\node at (1,1) {$\frac{u(vx+wx+wy)}{w(x+y)}$};
\node at (0,2) {$x$};
\node at (2,2) {$y$};
\node at (1,3) {$z$};
\draw[thick] (-0.35,0.35) -- (-0.65,0.65);
\draw[thick] (0.35,0.35) -- (0.65,0.65);
\draw[thick] (1.35,1.35) -- (1.65,1.65);
\draw[thick] (0.35,2.35) -- (0.65,2.65);
\draw[thick] (-0.35,1.65) -- (-0.65,1.35);
\draw[thick] (0.35,1.65) -- (0.65,1.35);
\draw[thick] (1.35,2.65) -- (1.65,2.35);
\node at (2.7,1.6) {{\LARGE$\stackrel{\tau_{(2,2)}}{\longmapsto}$}};
\begin{scope}[shift={(14/3,0)}]
\node at (0,0) {$\frac{C}{u(vx+wx+wy)z}$};
\node at (-0.75,1) {$\frac{u(vx+wx+wy)}{vx}$};
\node at (1,1) {$\frac{u(vx+wx+wy)}{w(x+y)}$};
\node at (0,2) {$\frac{vw(x+y)}{vx+wx+wy}$};
\node at (2,2) {$y$};
\node at (1,3) {$z$};
\draw[thick] (-0.35,0.35) -- (-0.65,0.65);
\draw[thick] (0.35,0.35) -- (0.65,0.65);
\draw[thick] (1.35,1.35) -- (1.65,1.65);
\draw[thick] (0.35,2.35) -- (0.65,2.65);
\draw[thick] (-0.35,1.65) -- (-0.65,1.35);
\draw[thick] (0.35,1.65) -- (0.65,1.35);
\draw[thick] (1.35,2.65) -- (1.65,2.35);
\end{scope}
\end{tikzpicture}
\end{center}

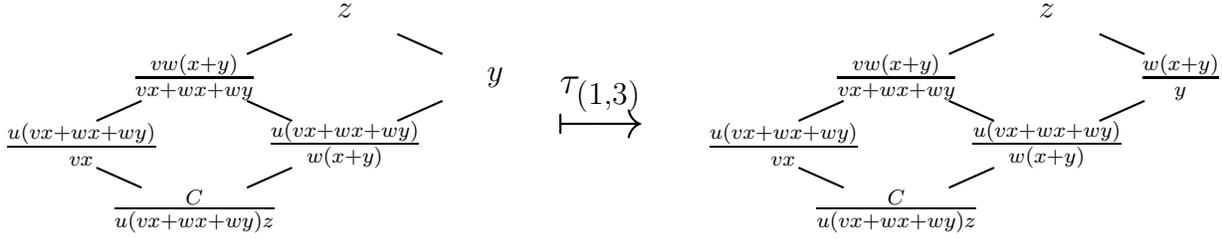
\begin{center}
\begin{tikzpicture}[xscale=2,yscale=8/9]
\node at (0,0) {$\frac{C}{u(vx+wx+wy)z}$};
\node at (-0.75,1) {$\frac{u(vx+wx+wy)}{vx}$};
\node at (1,1) {$\frac{u(vx+wx+wy)}{w(x+y)}$};
\node at (0,2) {$\frac{vw(x+y)}{vx+wx+wy}$};
\node at (2,2) {$y$};
\node at (1,3) {$z$};
\draw[thick] (-0.35,0.35) -- (-0.65,0.65);
\draw[thick] (0.35,0.35) -- (0.65,0.65);
\draw[thick] (1.35,1.35) -- (1.65,1.65);
\draw[thick] (0.35,2.35) -- (0.65,2.65);
\draw[thick] (-0.35,1.65) -- (-0.65,1.35);
\draw[thick] (0.35,1.65) -- (0.65,1.35);
\draw[thick] (1.35,2.65) -- (1.65,2.35);
\node at (2.7,1.6) {{\LARGE$\stackrel{\tau_{(1,3)}}{\longmapsto}$}};
\begin{scope}[shift={(14/3,0)}]
\node at (0,0) {$\frac{C}{u(vx+wx+wy)z}$};
\node at (-0.75,1) {$\frac{u(vx+wx+wy)}{vx}$};
\node at (1,1) {$\frac{u(vx+wx+wy)}{w(x+y)}$};
\node at (0,2) {$\frac{vw(x+y)}{vx+wx+wy}$};
\node at (1.89,2) {$\frac{w(x+y)}{y}$};
\node at (1,3) {$z$};
\draw[thick] (-0.35,0.35) -- (-0.65,0.65);
\draw[thick] (0.35,0.35) -- (0.65,0.65);
\draw[thick] (1.35,1.35) -- (1.65,1.65);
\draw[thick] (0.35,2.35) -- (0.65,2.65);
\draw[thick] (-0.35,1.65) -- (-0.65,1.35);
\draw[thick] (0.35,1.65) -- (0.65,1.35);
\draw[thick] (1.35,2.65) -- (1.65,2.35);
\end{scope}
\end{tikzpicture}
\end{center}

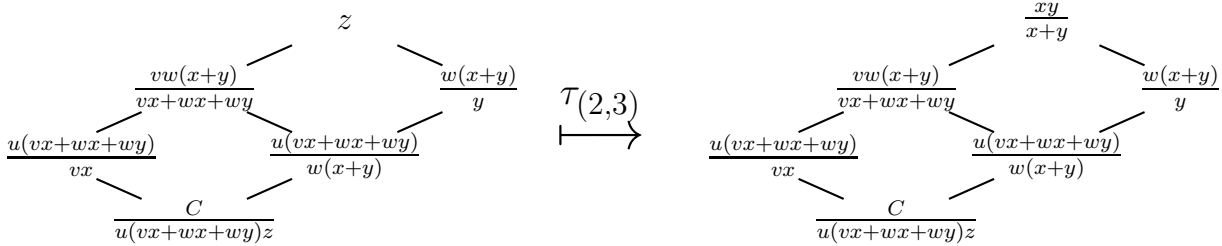
\begin{center}
\begin{tikzpicture}[xscale=2,yscale=8/9]
\node at (0,0) {$\frac{C}{u(vx+wx+wy)z}$};
\node at (-0.75,1) {$\frac{u(vx+wx+wy)}{vx}$};
\node at (1,1) {$\frac{u(vx+wx+wy)}{w(x+y)}$};
\node at (0,2) {$\frac{vw(x+y)}{vx+wx+wy}$};
\node at (1.89,2) {$\frac{w(x+y)}{y}$};
\node at (1,3) {$z$};
\draw[thick] (-0.35,0.35) -- (-0.65,0.65);
\draw[thick] (0.35,0.35) -- (0.65,0.65);
\draw[thick] (1.35,1.35) -- (1.65,1.65);
\draw[thick] (0.35,2.35) -- (0.65,2.65);
\draw[thick] (-0.35,1.65) -- (-0.65,1.35);
\draw[thick] (0.35,1.65) -- (0.65,1.35);
\draw[thick] (1.35,2.65) -- (1.65,2.35);
\node at (2.7,1.6) {{\LARGE$\stackrel{\tau_{(2,3)}}{\longmapsto}$}};
\begin{scope}[shift={(14/3,0)}]
\node at (0,0) {$\frac{C}{u(vx+wx+wy)z}$};
\node at (-0.75,1) {$\frac{u(vx+wx+wy)}{vx}$};
\node at (1,1) {$\frac{u(vx+wx+wy)}{w(x+y)}$};
\node at (0,2) {$\frac{vw(x+y)}{vx+wx+wy}$};
\node at (1.89,2) {$\frac{w(x+y)}{y}$};
\node at (1,3) {$\frac{xy}{x+y}$};
\draw[thick] (-0.35,0.35) -- (-0.65,0.65);
\draw[thick] (0.35,0.35) -- (0.65,0.65);
\draw[thick] (1.35,1.35) -- (1.65,1.65);
\draw[thick] (0.35,2.35) -- (0.65,2.65);
\draw[thick] (-0.35,1.65) -- (-0.65,1.35);
\draw[thick] (0.35,1.65) -- (0.65,1.35);
\draw[thick] (1.35,2.65) -- (1.65,2.35);
\end{scope}
\end{tikzpicture}
\end{center}

Note this composition of the six toggles is \textbf{one} iteration of BAR-motion on $g$.
\end{ex}

\begin{thm}\label{thm:chi-R-commute-diagram}
On a finite poset $P$, $\BAR= \down\circ\Theta\circ\up^{-1}$, and the diagram
below commutes.

\begin{center}
\begin{tikzpicture}[yscale=2/3]
\node at (0,1.8) {$\kk^P$};
\node at (0,0) {$\kk^P$};
\node at (3.25,1.8) {$\kk^P$};
\node at (3.25,0) {$\kk^P$};
\draw[semithick, dashed, ->] (0,1.3) -- (0,0.5);
\node[left] at (0,0.9) {$\down$};
\draw[semithick, dashed, ->] (0.7,0) -- (2.5,0);
\node[below] at (1.5,0) {$\BAR$};
\draw[semithick, dashed, ->] (0.7,1.8) -- (2.5,1.8);
\node[above] at (1.5,1.8) {$\BOR$};
\draw[semithick, dashed, ->] (3.25,1.3) -- (3.25,0.5);
\node[right] at (3.25,0.9) {$\down$};
\end{tikzpicture}
\end{center}
\end{thm}

Compare the example in Figure~\ref{fig:2x3-BAR}
to Example~\ref{ex:2x3}.
We will not prove Theorem~\ref{thm:chi-R-commute-diagram} now
as it follows immediately from the more general
noncommutative version: Theorem~\ref{thm:NAR-transfer}.

\begin{figure}[b]
\begin{center}
\begin{tikzpicture}[xscale=1.87,yscale=1.25]
\begin{scope}[yscale=4/7]
\node at (0,0) {$u$};
\node at (-1,1) {$v$};
\node at (1,1) {$w$};
\node at (0,2) {$x$};
\node at (2,2) {$y$};
\node at (1,3) {$z$};
\draw[thick] (-0.2,0.2) -- (-0.8,0.8);
\draw[thick] (0.2,0.2) -- (0.8,0.8);
\draw[thick] (1.2,1.2) -- (1.8,1.8);
\draw[thick] (0.2,2.2) -- (0.8,2.8);
\draw[thick] (-0.2,1.8) -- (-0.8,1.2);
\draw[thick] (0.2,1.8) -- (0.8,1.2);
\draw[thick] (1.2,2.8) -- (1.8,2.2);
\end{scope}
\draw[ultra thick, ->] (0,-0.5) -- (0,-1.1);
\node[left] at (0,-0.8) {$\BAR$};
\draw[ultra thick, ->] (2.5,0.5) -- (3.5,0.5);
\node[above] at (3,0.5) {$\up^{-1}$};
\node[above] at (3,-2.3) {$\down$};
\draw[ultra thick, ->] (3.5,-2.3) -- (2.5,-2.3);
\draw[ultra thick, ->] (5,-0.5) -- (5,-1.1);
\node[left] at (5,-0.8) {$\Theta$};
\begin{scope}[yscale=4/7,shift={(5,0)},xscale=5/6]
\node at (0,0) {$u(vx+wx+wy)z$};
\node at (-1,1) {$vxz$};
\node at (1,1) {$w(x+y)z$};
\node at (0,2) {$xz$};
\node at (2,2) {$yz$};
\node at (1,3) {$z$};
\draw[thick] (-0.3,0.3) -- (-0.7,0.7);
\draw[thick] (0.3,0.3) -- (0.7,0.7);
\draw[thick] (1.3,1.3) -- (1.7,1.7);
\draw[thick] (0.3,2.3) -- (0.7,2.7);
\draw[thick] (-0.3,1.7) -- (-0.7,1.3);
\draw[thick] (0.3,1.7) -- (0.7,1.3);
\draw[thick] (1.3,2.7) -- (1.7,2.3);
\end{scope}
\begin{scope}[shift={(0,-3.3)},yscale=13/18]
\node at (0,0) {$\frac{C}{u(vx+wx+wy)z}$};
\node at (-1,1) {$\frac{u(vx+wx+wy)}{vx}$};
\node at (1,1) {$\frac{u(vx+wx+wy)}{w(x+y)}$};
\node at (0,2) {$\frac{vw(x+y)}{vx+wx+wy}$};
\node at (2,2) {$\frac{w(x+y)}{y}$};
\node at (1,3) {$\frac{xy}{x+y}$};
\draw[thick] (-0.3,0.3) -- (-0.7,0.7);
\draw[thick] (0.3,0.3) -- (0.7,0.7);
\draw[thick] (1.3,1.3) -- (1.7,1.7);
\draw[thick] (0.3,2.3) -- (0.7,2.7);
\draw[thick] (-0.3,1.7) -- (-0.7,1.3);
\draw[thick] (0.3,1.7) -- (0.7,1.3);
\draw[thick] (1.3,2.7) -- (1.7,2.3);
\end{scope}
\begin{scope}[shift={(5,-3.3)},xscale=5/6,yscale=13/18]
\node at (0,0) {$\frac{C}{u(vx+wx+wy)z}$};
\node at (-1,1) {$\frac{C}{vxz}$};
\node at (1,1) {$\frac{C}{w(x+y)z}$};
\node at (0,2) {$\frac{C}{xz}$};
\node at (2,2) {$\frac{C}{yz}$};
\node at (1,3) {$\frac{C}{z}$};
\draw[thick] (-0.3,0.3) -- (-0.7,0.7);
\draw[thick] (0.3,0.3) -- (0.7,0.7);
\draw[thick] (1.3,1.3) -- (1.7,1.7);
\draw[thick] (0.3,2.3) -- (0.7,2.7);
\draw[thick] (-0.3,1.7) -- (-0.7,1.3);
\draw[thick] (0.3,1.7) -- (0.7,1.3);
\draw[thick] (1.3,2.7) -- (1.7,2.3);
\end{scope}
\end{tikzpicture}
\end{center}
\caption{One iteration of $\BAR$-motion on $[2]\times[3]$ as the composition of the three maps.}
\label{fig:2x3-BAR}
\end{figure}
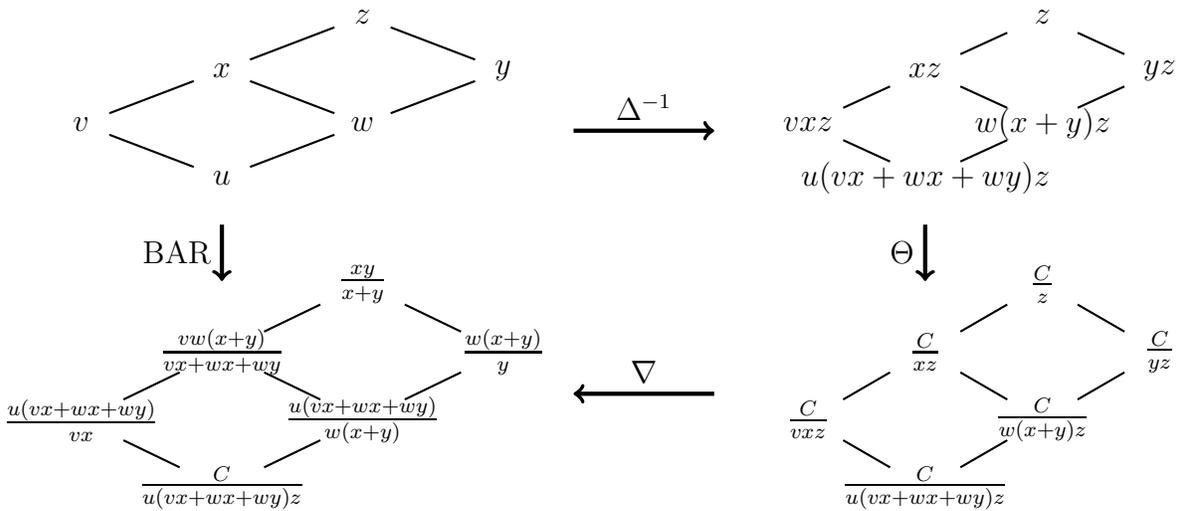

\begin{ex}\label{ex:BAR[2]x[3]}
Since $\BOR=\down^{-1} \circ \BAR \circ \down$,
$\BAR$ has the same order as $\BOR$ on any given
finite poset $P$.
On a general poset $P$, these birational rowmotion maps usually have infinite order, but on several nice families of posets described in~\cite{GrRo16,GrRo15,hopkins2019minuscule}, the order is finite (and the same as the order in the combinatorial realm).
For example, on a rectangle $[a]\times[b]$,
birational rowmotion has order $a+b$.
Figure~\ref{fig:BAR[2]x[3]} illustrates this with a generic orbit of $\BAR$ on $P=[2]\times[3]$, observing that the order is 5.
\end{ex}

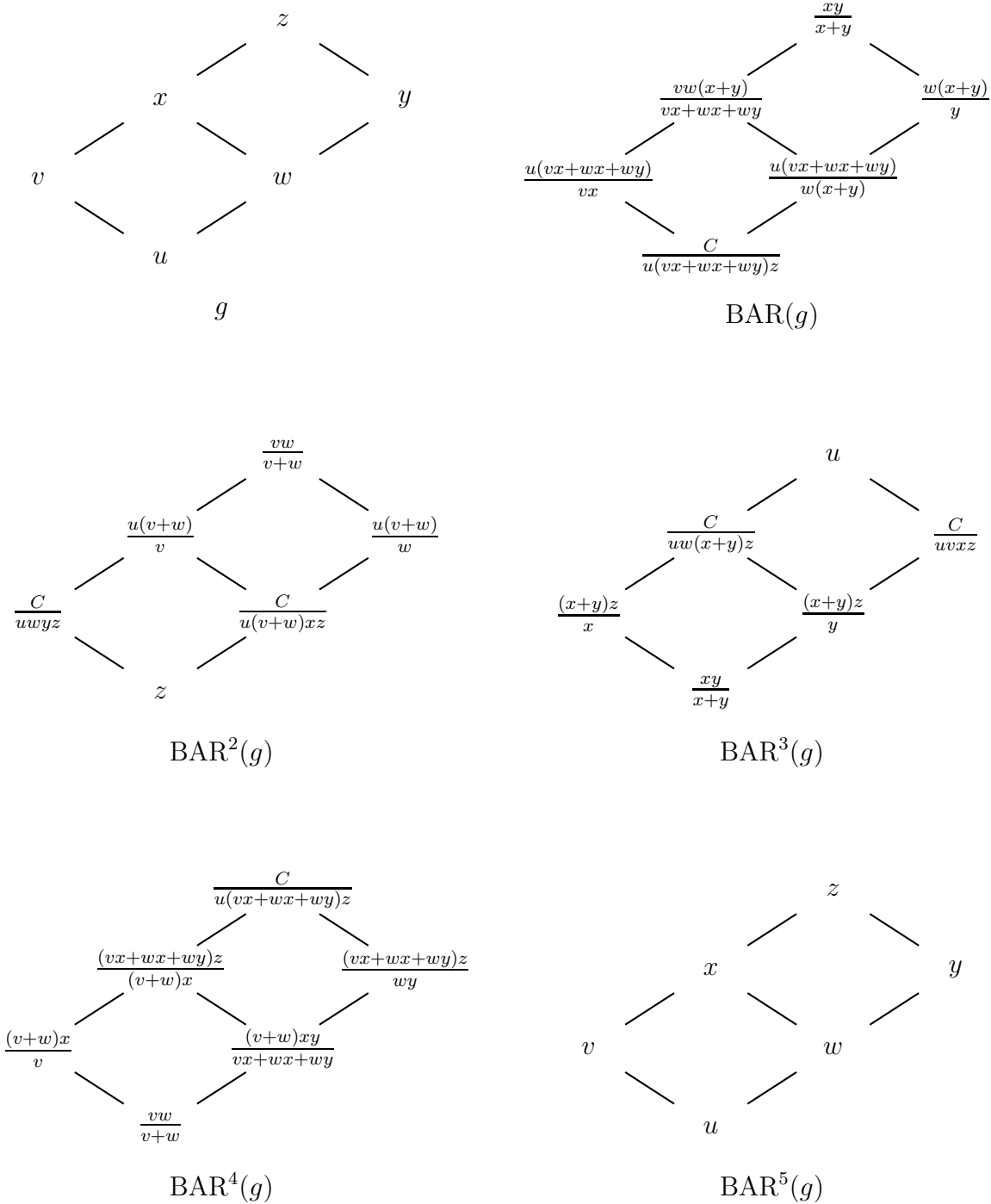
\begin{figure}
\centering
\begin{tikzpicture}[xscale=17/9, yscale=11/9]
\begin{scope}
\draw[thick] (-0.3, 1.7) -- (-0.7, 1.3);
\draw[thick] (0.3, 1.7) -- (0.7, 1.3);
\draw[thick] (-0.7, 0.7) -- (-0.3, 0.3);
\draw[thick] (0.7, 0.7) -- (0.3, 0.3);
\draw[thick] (0.7, 2.7) -- (0.3, 2.3);
\draw[thick] (1.3, 2.7) -- (1.7, 2.3);
\draw[thick] (1.7, 1.7) -- (1.3, 1.3);
\node at (1,3) {$  z  $};
\node at (0,2) {$  x  $};
\node at (2,2) {$  y  $};
\node at (-1,1) {$  v  $};
\node at (1,1) {$  w  $};
\node at (0,0) {$  u  $};
\node[below] at (0.5,-0.45) {$g$};
\end{scope}
\begin{scope}[shift={(4.5,0)}]
\draw[thick] (-0.3, 1.7) -- (-0.7, 1.3);
\draw[thick] (0.3, 1.7) -- (0.7, 1.3);
\draw[thick] (-0.7, 0.7) -- (-0.3, 0.3);
\draw[thick] (0.7, 0.7) -- (0.3, 0.3);
\draw[thick] (0.7, 2.7) -- (0.3, 2.3);
\draw[thick] (1.3, 2.7) -- (1.7, 2.3);
\draw[thick] (1.7, 1.7) -- (1.3, 1.3);
\node at (1,3) {$  \frac{xy}{x+y}  $};
\node at (0,2) {$  \frac{vw(x+y)}{vx+wx+wy}  $};
\node at (2,2) {$  \frac{w(x+y)}{y}  $};
\node at (-1,1) {$  \frac{u(vx+wx+wy)}{vx}  $};
\node at (1,1) {$  \frac{u(vx+wx+wy)}{w(x+y)}  $};
\node at (0,0) {$  \frac{C}{u(vx+wx+wy)z}  $};
\node[below] at (0.5,-0.45) {$\BAR(g)$};
\end{scope}
\begin{scope}[shift={(0,-5.5)}]
\draw[thick] (-0.3, 1.7) -- (-0.7, 1.3);
\draw[thick] (0.3, 1.7) -- (0.7, 1.3);
\draw[thick] (-0.7, 0.7) -- (-0.3, 0.3);
\draw[thick] (0.7, 0.7) -- (0.3, 0.3);
\draw[thick] (0.7, 2.7) -- (0.3, 2.3);
\draw[thick] (1.3, 2.7) -- (1.7, 2.3);
\draw[thick] (1.7, 1.7) -- (1.3, 1.3);
\node at (1,3) {$  \frac{vw}{v+w}  $};
\node at (0,2) {$  \frac{u(v+w)}{v}  $};
\node at (2,2) {$  \frac{u(v+w)}{w}  $};
\node at (-1,1) {$  \frac{C}{uwyz}  $};
\node at (1,1) {$  \frac{C}{u(v+w)xz}  $};
\node at (0,0) {$  z  $};
\node[below] at (0.5,-0.45) {$\BAR^2(g)$};
\end{scope}
\begin{scope}[shift={(4.5,-5.5)}]
\draw[thick] (-0.3, 1.7) -- (-0.7, 1.3);
\draw[thick] (0.3, 1.7) -- (0.7, 1.3);
\draw[thick] (-0.7, 0.7) -- (-0.3, 0.3);
\draw[thick] (0.7, 0.7) -- (0.3, 0.3);
\draw[thick] (0.7, 2.7) -- (0.3, 2.3);
\draw[thick] (1.3, 2.7) -- (1.7, 2.3);
\draw[thick] (1.7, 1.7) -- (1.3, 1.3);
\node at (1,3) {$  u  $};
\node at (0,2) {$  \frac{C}{uw(x+y)z}  $};
\node at (2,2) {$  \frac{C}{uvxz}  $};
\node at (-1,1) {$  \frac{(x+y)z}{x}  $};
\node at (1,1) {$  \frac{(x+y)z}{y}  $};
\node at (0,0) {$  \frac{xy}{x+y}  $};
\node[below] at (0.5,-0.45) {$\BAR^3(g)$};
\end{scope}
\begin{scope}[shift={(0,-11)}]
\draw[thick] (-0.3, 1.7) -- (-0.7, 1.3);
\draw[thick] (0.3, 1.7) -- (0.7, 1.3);
\draw[thick] (-0.7, 0.7) -- (-0.3, 0.3);
\draw[thick] (0.7, 0.7) -- (0.3, 0.3);
\draw[thick] (0.7, 2.7) -- (0.3, 2.3);
\draw[thick] (1.3, 2.7) -- (1.7, 2.3);
\draw[thick] (1.7, 1.7) -- (1.3, 1.3);
\node at (1,3) {$  \frac{C}{u(vx+wx+wy)z}  $};
\node at (0,2) {$  \frac{(vx+wx+wy)z}{(v+w)x}  $};
\node at (2,2) {$  \frac{(vx+wx+wy)z}{wy}  $};
\node at (-1,1) {$  \frac{(v+w)x}{v}  $};
\node at (1,1) {$  \frac{(v+w)xy}{vx+wx+wy}  $};
\node at (0,0) {$  \frac{vw}{v+w}  $};
\node[below] at (0.5,-0.45) {$\BAR^4(g)$};
\end{scope}
\begin{scope}[shift={(4.5,-11)}]
\draw[thick] (-0.3, 1.7) -- (-0.7, 1.3);
\draw[thick] (0.3, 1.7) -- (0.7, 1.3);
\draw[thick] (-0.7, 0.7) -- (-0.3, 0.3);
\draw[thick] (0.7, 0.7) -- (0.3, 0.3);
\draw[thick] (0.7, 2.7) -- (0.3, 2.3);
\draw[thick] (1.3, 2.7) -- (1.7, 2.3);
\draw[thick] (1.7, 1.7) -- (1.3, 1.3);
\node at (1,3) {$  z  $};
\node at (0,2) {$  x  $};
\node at (2,2) {$  y  $};
\node at (-1,1) {$  v  $};
\node at (1,1) {$  w  $};
\node at (0,0) {$  u  $};
\node[below] at (0.5,-0.45) {$\BAR^5(g)$};
\end{scope}
\end{tikzpicture}
\caption{An orbit of $\BAR$ starting with a generic labeling $g\in\kk^P$, for
$P=[2]\times[3]$.  We observe that $\BAR^5(g)=g$, so the order of $\BAR$ is $5=2+3$ on this poset.}
\label{fig:BAR[2]x[3]}
\end{figure}

\subsection{Isomorphism between the two birational
toggle groups}\label{ss:isoBR}

The first author constructed an explicit isomorphism between
the combinatorial toggle groups of order ideals and of antichains, and then lifted it to the
piecewise-linear realm~\cite{antichain-toggling}.  Here we further lift it to the birational realm.  It turns out that
this isomorphism can be lifted even further to the noncommutative realm, with modified
definitions because toggles are no longer involutions there. 
We do this in Section~\ref{sec:noncommutative}.  All the proofs over skew fields
imply the commutative birational counterparts by restriction. So in this subsection we merely
state these (new) results.  

\begin{defn}
For $v\in P$, let $T_v^* =\tau_{v_1}\tau_{v_2}\cdots \tau_{v_k} \tau_v \tau_{v_k}\cdots
\tau_{v_2}\tau_{v_1}$ where $v_1,\dots,v_k$ are the elements of $P$ covered by $v$.  (In the
case that $v$ is a minimal element of $P$, $k=0$ and $T_v^*=\tau_v$.) So $T_v^*$ is the
conjugation of $\tau_{v}$ by $\prod_{w\lessdot v}\tau_{w}$, which within $\BTog_{A}(P)$ mimics the effect of
the order toggle $T_{v}$.  

\end{defn}

\begin{absolutelynopagebreak}
\begin{thm}[{Analogue of~\cite[Thm.~2.15]{antichain-toggling}, generalized in Thm.~\ref{thm:T-star-NC}}]\label{thm:T-star}
Let $v\in P$.  Then the following diagram commutes on the domains in which the maps are defined.
\begin{center}
\begin{tikzpicture}[yscale=.6]
\node at (0,1.8) {$\kk^P$};
\node at (0,0) {$\kk^P$};
\node at (0,-1.8) {$\kk^P$};
\node at (3.25,1.8) {$\kk^P$};
\node at (3.25,0) {$\kk^P$};
\node at (3.25,-1.8) {$\kk^P$};
\draw[semithick, dashed, ->] (0,1.3) -- (0,0.5);
\node[left] at (0,0.9) {$\up^{-1}$};
\draw[semithick, dashed, ->] (0.7,-1.8) -- (2.5,-1.8);
\node[below] at (1.5,-1.8) {$T_v$};
\draw[semithick, dashed, ->] (0.7,1.8) -- (2.5,1.8);
\node[above] at (1.5,1.8) {$T_v^*$};
\draw[semithick, dashed, ->] (3.25,1.3) -- (3.25,0.5);
\node[right] at (3.25,0.9) {$\up^{-1}$};
\draw[semithick, dashed, ->] (0,-0.5) -- (0,-1.3);
\node[left] at (0,-0.9) {$\Theta$};
\draw[semithick, dashed, ->] (3.25,-0.5) -- (3.25,-1.3);
\node[right] at (3.25,-0.9) {$\Theta$};
\end{tikzpicture}
\end{center}
\end{thm}
\end{absolutelynopagebreak}

\begin{ex}\label{ex:T-star}
Again we consider the poset $P=[2]\times[3]$ with elements named as in Example~\ref{ex:2x3}.
Then $(2,2)$ covers $(2,1)$ and $(1,2)$, so
$T_{(2,2)}^*=\tau_{(2,1)}\tau_{(1,2)} \tau_{(2,2)} \tau_{(1,2)}\tau_{(2,1)}$.
Since
$(2,1)\parallel (1,2)$, the toggles $\tau_{(2,1)}$ and $\tau_{(1,2)}$ commute by
Proposition~\ref{prop:basic-tau-prop}, so we can apply them ``simultaneously''.
We verify
Theorem~\ref{thm:T-star} holds for this example in Figure~\ref{fig:T-star}.
\end{ex}

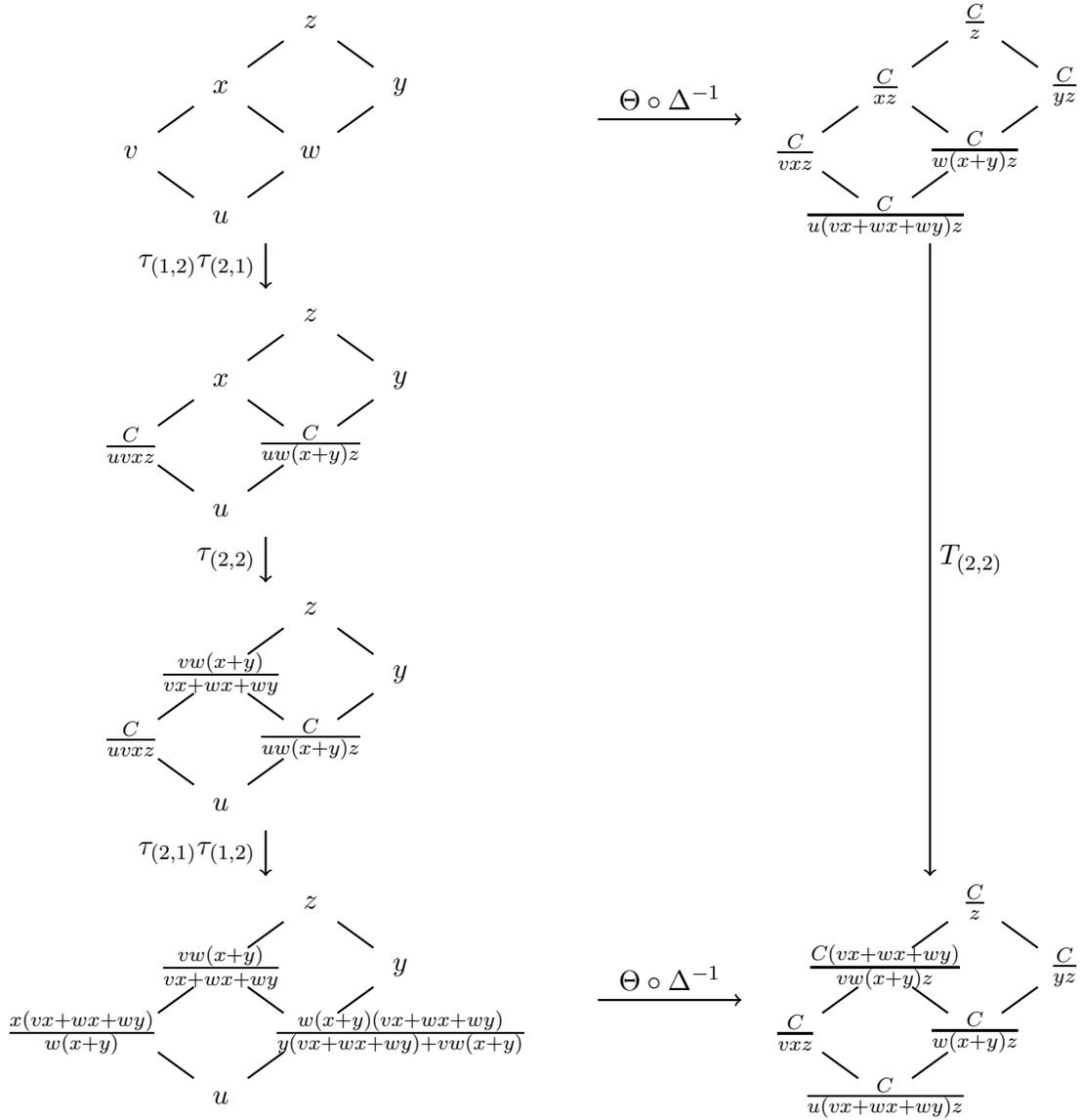
\begin{figure}
\begin{center}
\begin{small}
\begin{tikzpicture}[xscale=11/9,yscale=8/9]
\node at (0,0) {$u$};
\node at (-1,1) {$v$};
\node at (1,1) {$w$};
\node at (0,2) {$x$};
\node at (2,2) {$y$};
\node at (1,3) {$z$};
\draw[thick] (-0.3,0.3) -- (-0.7,0.7);
\draw[thick] (0.3,0.3) -- (0.7,0.7);
\draw[thick] (1.3,1.3) -- (1.7,1.7);
\draw[thick] (0.3,2.3) -- (0.7,2.7);
\draw[thick] (-0.3,1.7) -- (-0.7,1.3);
\draw[thick] (0.3,1.7) -- (0.7,1.3);
\draw[thick] (1.3,2.7) -- (1.7,2.3);
\draw[thick, ->] (0.5,-0.4) -- (0.5,-1.1);
\node[left] at (0.5,-0.75) {$\tau_{(1,2)}\tau_{(2,1)}$};
\draw[thick, ->] (4.2,1.5) -- (5.8,1.5);
\node[above] at (5,1.5) {$\Theta \circ \up^{-1}$};
\begin{scope}[shift={(0,-4.5)}]
\node at (0,0) {$u$};
\node at (-1,1) {$\frac{C}{uvxz}$};
\node at (1,1) {$\frac{C}{uw(x+y)z}$};
\node at (0,2) {$x$};
\node at (2,2) {$y$};
\node at (1,3) {$z$};
\draw[thick] (-0.3,0.3) -- (-0.7,0.7);
\draw[thick] (0.3,0.3) -- (0.7,0.7);
\draw[thick] (1.3,1.3) -- (1.7,1.7);
\draw[thick] (0.3,2.3) -- (0.7,2.7);
\draw[thick] (-0.3,1.7) -- (-0.7,1.3);
\draw[thick] (0.3,1.7) -- (0.7,1.3);
\draw[thick] (1.3,2.7) -- (1.7,2.3);
\draw[thick, ->] (0.5,-0.4) -- (0.5,-1.1);
\node[left] at (0.5,-0.75) {$\tau_{(2,2)}$};
\end{scope}
\begin{scope}[shift={(0,-9)}]
\node at (0,0) {$u$};
\node at (-1,1) {$\frac{C}{uvxz}$};
\node at (1,1) {$\frac{C}{uw(x+y)z}$};
\node at (0,2) {$\frac{vw(x+y)}{vx+wx+wy}$};
\node at (2,2) {$y$};
\node at (1,3) {$z$};
\draw[thick] (-0.3,0.3) -- (-0.7,0.7);
\draw[thick] (0.3,0.3) -- (0.7,0.7);
\draw[thick] (1.3,1.3) -- (1.7,1.7);
\draw[thick] (0.3,2.3) -- (0.7,2.7);
\draw[thick] (-0.3,1.7) -- (-0.7,1.3);
\draw[thick] (0.3,1.7) -- (0.7,1.3);
\draw[thick] (1.3,2.7) -- (1.7,2.3);
\draw[thick, ->] (0.5,-0.4) -- (0.5,-1.1);
\node[left] at (0.5,-0.75) {$\tau_{(2,1)}\tau_{(1,2)}$};
\end{scope}
\begin{scope}[shift={(0,-13.5)}]
\node at (0,0) {$u$};
\node at (-1.55,1) {$\frac{x(vx+wx+wy)}{w(x+y)}$};
\node at (2,1) {$\frac{w(x+y)(vx+wx+wy)}{y(vx+wx+wy)+vw(x+y)}$};
\node at (0,2) {$\frac{vw(x+y)}{vx+wx+wy}$};
\node at (2,2) {$y$};
\node at (1,3) {$z$};
\draw[thick] (-0.3,0.3) -- (-0.7,0.7);
\draw[thick] (0.3,0.3) -- (0.7,0.7);
\draw[thick] (1.3,1.3) -- (1.7,1.7);
\draw[thick] (0.3,2.3) -- (0.7,2.7);
\draw[thick] (-0.3,1.7) -- (-0.7,1.3);
\draw[thick] (0.3,1.7) -- (0.7,1.3);
\draw[thick] (1.3,2.7) -- (1.7,2.3);
\draw[thick, ->] (4.2,1.5) -- (5.8,1.5);
\node[above] at (5,1.5) {$\Theta \circ \up^{-1}$};
\end{scope}
\begin{scope}[shift={(7.4,-13.5)}]
\node at (0,0) {$\frac{C}{u(vx+wx+wy)z}$};
\node at (-1,1) {$\frac{C}{vxz}$};
\node at (1,1) {$\frac{C}{w(x+y)z}$};
\node at (0,2) {$\frac{C(vx+wx+wy)}{vw(x+y)z}$};
\node at (2,2) {$\frac{C}{yz}$};
\node at (1,3) {$\frac{C}{z}$};
\draw[thick] (-0.3,0.3) -- (-0.7,0.7);
\draw[thick] (0.3,0.3) -- (0.7,0.7);
\draw[thick] (1.3,1.3) -- (1.7,1.7);
\draw[thick] (0.3,2.3) -- (0.7,2.7);
\draw[thick] (-0.3,1.7) -- (-0.7,1.3);
\draw[thick] (0.3,1.7) -- (0.7,1.3);
\draw[thick] (1.3,2.7) -- (1.7,2.3);
\end{scope}
\begin{scope}[shift={(7.4,0)}]
\node at (0,0) {$\frac{C}{u(vx+wx+wy)z}$};
\node at (-1,1) {$\frac{C}{vxz}$};
\node at (1,1) {$\frac{C}{w(x+y)z}$};
\node at (0,2) {$\frac{C}{xz}$};
\node at (2,2) {$\frac{C}{yz}$};
\node at (1,3) {$\frac{C}{z}$};
\draw[thick] (-0.3,0.3) -- (-0.7,0.7);
\draw[thick] (0.3,0.3) -- (0.7,0.7);
\draw[thick] (1.3,1.3) -- (1.7,1.7);
\draw[thick] (0.3,2.3) -- (0.7,2.7);
\draw[thick] (-0.3,1.7) -- (-0.7,1.3);
\draw[thick] (0.3,1.7) -- (0.7,1.3);
\draw[thick] (1.3,2.7) -- (1.7,2.3);
\draw[thick, ->] (0.5,-0.4) -- (0.5,-10.1);
\node[right] at (0.5,-5.25) {$T_{(2,2)}$};
\end{scope}
\end{tikzpicture}
\end{small}
\end{center}
\caption{An illustration of Theorem~\ref{thm:T-star}.
See Example~\ref{ex:T-star}.}
\label{fig:T-star}
\end{figure}

\begin{defn}
Let $v\in P$ and let $(x_1,\dots,x_k)$ be a linear extension of the subposet $\{x\in
P\;|\;x < v\}$ of $P$.  Define 
$\eta_v=T_{x_1} T_{x_2}\cdots T_{x_k}$
and
$\tau_v^* = \eta_v T_v \eta_v^{-1}.$
(Note that $\eta_v$ is well-defined since
toggles corresponding to incomparable elements commute.)
\end{defn}

As the next theorem formalizes, $\tau_v^*$ mimics the antichain toggle $\tau_v$ in terms of order toggles.

\begin{absolutelynopagebreak}
\begin{thm}[{Analogue of~\cite[Thm.~2.19]{antichain-toggling}, generalized in Thm.~\ref{thm:tau-star-NC}}]\label{thm:tau-star}
Let $v\in P$.  Then the following diagram commutes on the domains in which the maps are defined.
\begin{center}
\begin{tikzpicture}[yscale=.6]
\node at (0,1.8) {$\kk^P$};
\node at (0,0) {$\kk^P$};
\node at (0,-1.8) {$\kk^P$};
\node at (3.25,1.8) {$\kk^P$};
\node at (3.25,0) {$\kk^P$};
\node at (3.25,-1.8) {$\kk^P$};
\draw[semithick, dashed, ->] (0,1.3) -- (0,0.5);
\node[left] at (0,0.9) {$\up^{-1}$};
\draw[semithick, dashed, ->] (0.7,-1.8) -- (2.5,-1.8);
\node[below] at (1.5,-1.8) {$\tau_v^*$};
\draw[semithick, dashed, ->] (0.7,1.8) -- (2.5,1.8);
\node[above] at (1.5,1.8) {$\tau_v$};
\draw[semithick, dashed, ->] (3.25,1.3) -- (3.25,0.5);
\node[right] at (3.25,0.9) {$\up^{-1}$};
\draw[semithick, dashed, ->] (0,-0.5) -- (0,-1.3);
\node[left] at (0,-0.9) {$\Theta$};
\draw[semithick, dashed, ->] (3.25,-0.5) -- (3.25,-1.3);
\node[right] at (3.25,-0.9) {$\Theta$};
\end{tikzpicture}
\end{center}
\end{thm}
\end{absolutelynopagebreak}

Theorems~\ref{thm:T-star} and~\ref{thm:tau-star}
yield the following corollary.

\begin{cor}
There is an isomorphism from $\btog_A(P)$ to
$\btog_O(P)$ given by $\tau_v \mapsto \tau_v^*$
with inverse given by $T_v \mapsto T_v^*$.
\end{cor}

\section{Birational antichain toggles on graded posets}\label{sec:graded}

\subsection{Toggling by ranks}

Although rowmotion can be defined on any finite poset,
only graded posets are currently known to have nice behavior (periodicity, homomesy, cyclic sieving, resonance).
In fact, the name ``rowmotion'' references the natural factorization of this map as a
product of rank (``row'') toggles in the graded case~\cite{strikerwilliams}.

\begin{defn}[\cite{ec1ed2}]
A poset $P$ is \textbf{graded} if it has a well-defined \textbf{rank function} $\rk: P\ra \zz_{\geq 0}$ satisfying \begin{itemize}
\item $\rk(x)=0$ for any minimal element $x$,
\item $\rk(y)=\rk(x)+1$ if $y\gtrdot x$,
\item every maximal element $x$ of $P$ has $\rk(x)=r$, where $r$ is the {\bf rank} of $P$.
\end{itemize}
For $x\in P$, we call $\rk(x)$ the {\bf rank} of $x$.
\end{defn}

In this section, $P$ will always refer to a finite graded poset of rank $r$.

\begin{defn}
Define birational antichain \textbf{rank toggles} as follows: 
$$\rkT{i}:=\prod\limits_{\rk(x)=i} T_x,\hspace{0.3 in}\rktau{i}:=\prod\limits_{\rk(x)=i}
\tau_x,\hspace{0.3 in}\rkT{i}^*:=\prod\limits_{\rk(x)=i} T_x^*,\hspace{0.3
in}\rktau{i}^*:=\prod\limits_{\rk(x)=i} \tau_x^*.$$ 
Since poset elements of the same rank are pairwise incomparable, each product is of commuting toggles.
Thus, the products above are all well-defined involutions.
\end{defn}

Clearly $\BAR=\rktau{r}\rktau{r-1}\cdots \rktau{1}\rktau{0}$ and $\BOR = \rkT{0} \rkT{1} \cdots \rkT{r-1} \rkT{r}.$

Under the isomorphism between $\btog_A(P)$ and $\btog_O(P)$, the rank toggles get sent to products of rank toggles, as the following proposition shows.
The proof in the piecewise-linear realm~\cite[Propositions~2.30,~2.31]{antichain-toggling}
goes through unchanged, so we omit it here. 

\begin{prop}
Let $v\in P$ with $\rk(v)=i$.
We have the following identities involving rank toggles
(where the empty product $\rktau{-1}$ is the identity). 
\begin{itemize}
    \item $T_v^*=\rktau{i-1}\tau_v\rktau{i-1}$,
    \item $\rkT{i}^*=
\rktau{i-1}\rktau{i}\rktau{i-1}$,
    \item $\tau_v^*=\rkT{0}\rkT{1}\cdots \rkT{i-1} T_v \rkT{i-1}\cdots \rkT{1}\rkT{0}$,
    \item $\rktau{i}^*=\rkT{0}\rkT{1}\cdots \rkT{i-1}\rkT{i}\rkT{i-1}\cdots \rkT{1}\rkT{0}$.
\end{itemize}
\end{prop}

\subsection{Graded rescaling}\label{ss:g-rescaling}

Grinberg and the second author analyzed the effect of BOR (and also birational order rank
toggles) on graded rescalings of poset labelings, which highlighted the advantages of working with
graded posets~\cite{GrRo16}.  In this section, we give analogous results for the antichain
analogues, describing how $\BAR$ and birational antichain rank toggles $\rktau{i}$ act on
graded rescalings.

\begin{defn}[{\cite[\S6]{GrRo16}}]\label{defn:graded_resc}
Let $(a_0,\dots,a_r)\in(\kk^\times)^{r+1}$ (where $\kk^{\times}$ denotes the nonzero elements of
$\kk$) and
$g\in\kk^P$.
Then $(a_0,\dots,a_r)\flatbin g$ is the $\kk$-labeling of $P$
formed by taking $g$ and multiplying the labels of all elements of rank $i$
by $a_i$.  This is called a \textbf{graded rescaling} of $g$ by $(a_0,\dots,a_r)$.
\end{defn}

\begin{ex}In the positive root poset $\Phi^+(A_3)$

\begin{center}
\begin{tikzpicture}[yscale=7/9]
\begin{scope}[shift={(0,0)}]
\draw[thick] (-0.3, 1.7) -- (-0.7, 1.3);
\draw[thick] (0.3, 1.7) -- (0.7, 1.3);
\draw[thick] (-1.3, 0.7) -- (-1.7, 0.3);
\draw[thick] (-0.7, 0.7) -- (-0.3, 0.3);
\draw[thick] (0.7, 0.7) -- (0.3, 0.3);
\draw[thick] (1.3, 0.7) -- (1.7, 0.3);
\node at (0,2) {$  z  $};
\node at (-1,1) {$  x  $};
\node at (1,1) {$  y  $};
\node at (-2,0) {$  u  $};
\node at (0,0) {$  v  $};
\node at (2,0) {$  w  $};
\node at (-3.1,1) {$(2,4,9)\;\flatbin$};
\node at (2.9,1) {$=$};
\end{scope}
\begin{scope}[shift={(5.7,0)}]
\draw[thick] (-0.3, 1.7) -- (-0.7, 1.3);
\draw[thick] (0.3, 1.7) -- (0.7, 1.3);
\draw[thick] (-1.3, 0.7) -- (-1.7, 0.3);
\draw[thick] (-0.7, 0.7) -- (-0.3, 0.3);
\draw[thick] (0.7, 0.7) -- (0.3, 0.3);
\draw[thick] (1.3, 0.7) -- (1.7, 0.3);
\node at (0,2) {$  9z  $};
\node at (-1,1) {$  4x  $};
\node at (1,1) {$  4y  $};
\node at (-2,0) {$  2u  $};
\node at (0,0) {$  2v  $};
\node at (2,0) {$  2w  $.};
\end{scope}
\end{tikzpicture}
\end{center}
\end{ex}

\begin{prop}\label{prop:rank-tog-graded-rescaling}
Let $g\in\kk^P$ and $(a_0,\dots,a_r)\in(\kk^\times)^{r+1}$.
Then
$$\rktau{i}\big( (a_0,\dots,a_r)\flatbin g \big)=
\left( a_0,\dots,a_{i-1},\frac{1}{a_0\cdots a_r},a_{i+1},\dots,a_r \right) \flatbin\rktau{i}g.
$$
\end{prop}
The analogous result for birational order toggles~\cite[Prop.~39]{GrRo16} has 
$\ds \frac{a_{i-1}a_{i+1}}{a_{i}}$ in the $i$th position. 
\begin{proof}
Let $h=(a_0,\dots,a_r)\flatbin g$.  Let $v\in P$ have $\rk(v)=i$.  Then every maximal chain
$(y_0,\dots,y_r)$ in $P$ contains one element from each rank level.  Therefore, 
\begin{align*}
\Upsilon_v h &= \sum\limits_{(y_0,\dots,y_r)\in\mc_v(P)} h(y_0)h(y_1)\cdots h(y_{r-1}) h(y_r)\\
&= \sum\limits_{(y_0,\dots,y_r)\in\mc_v(P)} a_0 g(y_0) a_1 g(y_1)\cdots a_{r-1} g(y_{r-1}) a_r g(y_r)\\
&= a_0a_1\cdots a_r \sum\limits_{(y_0,\dots,y_r)\in\mc_v(P)} g(y_0) g(y_1)\cdots g(y_{r-1}) g(y_r)\\
&= a_0a_1\cdots a_r \Upsilon_v g.
\end{align*}
Then for every $v$ of rank $i$,
$$(\rktau{i}h)(v)=\frac{C}{\Upsilon_v h}=\frac{C}{a_0a_1\cdots a_r \Upsilon_v g}
=\frac{1}{a_0a_1\cdots a_r}(\rktau{i}g)(v),$$ 
while for any $x$ of rank $j\not=i$,
$(\rktau{i}h)(x)=h(x)=a_j g(x) = a_j (\rktau{i}g)(x)$.
Thus, $$\rktau{i}h = \left( a_0,\dots,a_{i-1},\frac{1}{a_0\cdots a_r},a_{i+1},\dots,a_r \right) \flatbin \rktau{i}g.$$
\end{proof}
A straightforward induction argument shows the following analogue of
\cite[Prop.~40]{GrRo16}. 
\begin{prop}\label{prop:BAR-graded-rescaling}
Let $g\in\kk^P$ and $(a_0,\dots,a_r)\in(\kk^\times)^{r+1}$.
Then
$$\BAR\big( (a_0,a_1, \dots,a_r)\flatbin g \big)=
\left( \frac{1}{a_0 a_1\cdots a_r}, a_0, a_1, \dots, a_{r-2}, a_{r-1} \right) \flatbin\BAR(g).
$$
\end{prop}



The key idea is that, like $\BOR$, applying $\BAR$ to a graded rescaling of $g$ yields a graded
rescaling of $\BAR(g)$.  

\subsection{Gyration}

For a finite graded poset $P$, Striker defined an
element of $\tog_\calj(P)$ called \emph{gyration},
which is conjugate to order-ideal rowmotion
$\rowJ$~\cite[\S6]{strikerRS}.
The name \emph{gyration} comes from  Wieland's action of the same name on alternating sign matrices~\cite{wieland}. 
Here we give a birational lifting of gyration using the
same definition, which remains conjugate to
$\BOR$ in $\tog_O(P)$ because the key algebraic properties of the order-ideal toggle group lift
to the birational realm. In fact this idea goes back to Cameron and Fon-Der-Flaass,
who considered (slightly more general) ``rank-permuted rowmotions,'' and showed that they
were all conjugate in the order-ideal toggle group~\cite[Lemma 2]{cameronfonder}.

\begin{defn}
\textbf{Birational order gyration} is the birational map
$\BOG: \kk^P \dra \kk^P$ that applies the birational order toggles for elements in even ranks first, then the odd ranks.
\end{defn}

For example, if $P$ has rank 7, then
$\BOG=\rkT{7}\rkT{5}\rkT{3}\rkT{1}\rkT{6}\rkT{4}\rkT{2}\rkT{0}$.
This is well-defined, since it does not matter in which order the elements of even rank are toggled, and similarly for odd rank.
This is because rank toggles $\rkT{i}, \rkT{j}$ commute when $|i-j|\not= 1$, where there are
no cover relations between an element of rank $i$ and an element of rank $j$. 
Since $\BOR$ and $\BOG$ are conjugate in $\tog_O(P)$,
they have the same order on any graded poset, and sometimes other properties can be transferred between the two maps.

On the other hand, birational \emph{antichain}
rank toggles never commute, so we need to specify a toggle order in the following definition.
This definition of $\BAG$ is the image of $\BOG$
under our explicit isomorphism between $\btog_O(P)$ and $\btog_A(P)$.

\begin{defn}
\textbf{Birational antichain gyration} is the birational map
$\BAG:\kk^P \dra \kk^P$ that first applies the antichain toggles for odd ranks starting from the bottom of the poset
up to the top, and then toggles the even ranks from the top of the poset
down to the bottom.
\end{defn}

For example, if $P$ has rank 7, then
$\BAG=\rktau{0}\rktau{2}\rktau{4}\rktau{6}\rktau{7}\rktau{5}\rktau{3}\rktau{1}$.
We omit the proof of the following theorem here, which is
completely analogous to~\cite[Theorem~2.34]{antichain-toggling} but now lifted to the birational realm.

\begin{thm}
The following diagram commutes.
\begin{center}
\begin{tikzpicture}[yscale=.6]
\node at (0,1.8) {$\kk^P$};
\node at (0,0) {$\kk^P$};
\node at (0,-1.8) {$\kk^P$};
\node at (3.25,1.8) {$\kk^P$};
\node at (3.25,0) {$\kk^P$};
\node at (3.25,-1.8) {$\kk^P$};
\draw[semithick, dashed, ->] (0,1.3) -- (0,0.5);
\node[left] at (0,0.9) {$\up^{-1}$};
\draw[semithick, dashed, ->] (0.7,-1.8) -- (2.5,-1.8);
\node[below] at (1.5,-1.8) {$\BOG$};
\draw[semithick, dashed, ->] (0.7,1.8) -- (2.5,1.8);
\node[above] at (1.5,1.8) {$\BAG$};
\draw[semithick, dashed, ->] (3.25,1.3) -- (3.25,0.5);
\node[right] at (3.25,0.9) {$\up^{-1}$};
\draw[semithick, dashed, ->] (0,-0.5) -- (0,-1.3);
\node[left] at (0,-0.9) {$\Theta$};
\draw[semithick, dashed, ->] (3.25,-0.5) -- (3.25,-1.3);
\node[right] at (3.25,-0.9) {$\Theta$};
\end{tikzpicture}
\end{center}
\end{thm}

\section{Noncommutative (skew field) dynamics}\label{sec:noncommutative}

\subsection{Introduction to skew field dynamics}

Darij Grinberg has conjectured that the periodicity of BOR-motion on certain nice posets
continues to hold even when extended to labelings of $P$ by elements that do not
necessarily commute.  Here we study the lifting to this setting of the
antichain perspective and relate it to the order perspective. We first recall Grinberg's original toggling definition of this map, which we call \textbf{NOR-motion}, and show that it is also given as a composition of three transfer maps as in the commutative case. 

Next we define the antichain analogues of toggling and \textbf{NAR-motion}, which can also be given in terms of the transfer maps. Along the way we give an explicit isomorphism between the group of order toggles and the group of antichain toggles in the noncommutative case.

Let $\bbs$ denote a \textbf{skew field} that contains an infinite field $\ff$ as a
subfield.  Any such skew field $\bbs$ of characteristic zero satisfies this condition, as it contains an isomorphic copy of $\qq$.
We will now work with $\bbs$-labelings in $\bbs^P:=\{f:P\rightarrow \bbs \}$.
We always require the generic constant $C\in \bbs$ to be in the
\emph{center} of $\bbs$ (i.e., $C$ commutes with every element of $\bbs$). The proofs in
this section specialize to show the results of \S\,\ref{ss:isoBR}. 

\begin{notation}
For $x\in\bbs$, write $\overline{x}$ for $x^{-1}$.
\end{notation}

\begin{notation}
The (commutative and associative) operation {\bf parallel sum} is defined by $x\parallelsum y=\overline{\overline{x}+\overline{y}}$.
We use $\sumpar$ as the analogue of $\sum$ with $+$ replaced with $\parallelsum$.
\end{notation}

The following {\bf reciprocity} relation, analogous to one in~\cite[\S5]{einpropp}, is easy to show.
\begin{prop}\label{prop:reciprocity}
For $x_1,\dots,x_n\in \bbs$,
$$\left(\sumparlimits{i=1}{n} x_i\right)
\left(\sum\limits_{i=1}^n \overline{x_i}\right)=
\left(\sum\limits_{i=1}^n \overline{x_i}\right)
\left(\sumparlimits{i=1}{n} x_i\right) =1.
$$
\end{prop}

\begin{remark}
It can be quite tricky to simplify expressions in a skew field until one gains some
experience.  
For example $x\parallelsum y = \overline{\overline{x} + \overline{y}}$ can equivalently be written as
\begin{itemize}
    \item $y\overline{(x+y)}x$ by multiplying on the left by $y\overline{y}$ and the right by $\overline{x}x$
    and using the property $\overline{AB}=\overline{B}\cdot
    \overline{A}$,
    \item or as $x\overline{(x+y)}y$ by multiplying on the left by $x\overline{x}$ and the right by $\overline{y}y$,
\end{itemize}
but is {\bf not equivalent} to
$yx\overline{(x+y)}$, $\overline{(x+y)}xy$, $xy\overline{(x+y)}$, or $\overline{(x+y)}yx$.
We can simplify $\overline{y}x\overline{(x+y)}y$ as
$$\overline{y}x\overline{(x+y)}y=
\overline{y}y\overline{(x+y)}x
=\overline{(x+y)}x.$$
Many other expressions are more challenging to rewrite in equivalent forms.  For example,
$$\overline{\overline{v}\cdot\overline{x}+
\overline{w}\cdot\overline{(x+y)}}=
(x+y)w\overline{(xv+xw+yw)}xv=
xv\overline{(xv+xw+yw)}(x+y)w$$
and
$$\overline{\overline{(v+w)}\cdot\overline{x}+\overline{w}
\cdot\overline{y}}=yw\overline{(xv+xw+yw)}x(v+w)=
x(v+w)\overline{(xv+xw+yw)}yw$$
are expressions that have arisen naturally in this study.  
\end{remark}
\begin{remark}\label{rem:partialMap}
When we move to the noncommutative setting, we no longer have the notions from algebraic
geometry and commutative 
algebra of \textit{Zariski
topology} and \textit{birational maps}, so we call the analogous maps \textbf{partial maps}.
These maps will not be defined when 
expressions in the denominator (i.e., expression we take the inverse of) become zero, so the domains need to be restricted somehow.
We do not try to address this issue formally, which would take us too far \emph{afield}.  In
particular, all the maps we consider are noncommutative analogues of our earlier birational
maps.  At a minimum, any equalities stated will hold as birational identities whenever we restrict the variables to lie
in the infinite subfield $\ff$. Practically speaking, they will hold in much greater generality.
\end{remark}

\begin{defn}[Darij Grinberg]
Let $v\in P$.  The \textbf{noncommutative order toggle} is the partial map $T_v:\bbs^P\dra
\bbs^P$ defined as follows.  For this definition, we extend any $f\in\bbs^P$ to a
$\widehat{P}$-labeling by setting $f\big(\widehat{0}\big)=1$ and $f\big( \widehat{1}
\big) = C$, recalling $C$ is in the center of $\bbs$.  Then 

$$\big(T_v(f)\big)(w) = \left\{\begin{array}{ll}
\left(\sum\limits_{u\in\widehat{P},u\lessdot v}f(u)\right)
\overline{f(v)} \left(\sumpar\limits_{u\in\widehat{P},u\gtrdot v}f(u)\right)
 &\text{if $w=v$}\\\vspace{-7 pt}\\
f(w) &\text{if $w\not=v$}
\end{array}\right.
$$
\end{defn}

Toggles in the noncommutative realm are {\bf not involutions}, so it is surprising they have any nice features at all.  
Initially, the term ``toggle'' was chosen
to be an involution, but we continue to use it here to keep terminology consistent with other realms.
We will call the inverse of a toggle an {\bf elggot}.

\begin{defn}
Let $v\in P$.  The \textbf{noncommutative order elggot} is the partial map $E_v:\bbs^P\dra
\bbs^P$ defined as follows.
Then 
$$\big(E_v(f)\big)(w) = \left\{\begin{array}{ll}
\left(\sumpar\limits_{u\in\widehat{P},u\gtrdot v}f(u)\right)
\overline{f(v)} \left(\sum\limits_{u\in\widehat{P},u\lessdot v}f(u)\right)
 &\text{if $w=v$}\\\vspace{-7 pt}\\
f(w) &\text{if $w\not=v$}
\end{array}\right.
$$
\end{defn}
A straightforward computation shows that $T_{v}$ and $E_{v}$ are inverse partial
maps.  
Order toggles and elggots commute with each other in the skew-field setting exactly when they do 
in the (commutative) birational realm.  We omit the elementary proof.

\begin{prop}
Let $u,v\in P$.  If neither $u$ nor $v$ covers the other, then
$T_uT_v=T_vT_u$, $E_uE_v=E_vE_u$, $T_uE_v=E_vT_u$, and
$E_uT_v=T_vE_u$.
\end{prop}

\begin{defn}[Darij Grinberg]
Let $(x_1,x_2,\dots,x_n)$ be any linear extension of a finite poset $P$.  Then the partial
map $\NOR=T_{x_1} T_{x_2} \cdots T_{x_n}$ is called \textbf{noncommutative order rowmotion
(NOR-motion)}. 
\end{defn}

\begin{conj}[Darij Grinberg]\label{conj:NOR-prod-chains}
On $[a]\times[b]$, $\NOR$ has order $a+b$.
\end{conj}

\subsection{Transfer maps in the noncommutative realm}

\begin{defn}
Let $f\in\bbs^P$.  We define {\bf complement} $\Theta$, {\bf down transfer} $\down$, {\bf up
transfer} $\up$, {\bf inverse down transfer} $\down^{-1}$, and {\bf inverse up transfer}
$\up^{-1}$ 
as follows.  
These specialize to Definition~\ref{def:bir-trans} when $\bbs$ is actually a field.  
\begin{align*}
    (\Theta f)(x) &= C \cdot \overline{f(x)}\\
    (\down f)(x) &= f(x) \cdot \overline{\sum\limits_{y\lessdot x}f(y)} \quad \big(\text{with } f\big(\widehat{0}\big)=1\big)\\
    (\up f)(x) &= 
    \overline{\sum\limits_{y\gtrdot x}f(y)} \cdot f(x) \quad \big(\text{with } f\big(\widehat{1}\big)=1\big)\\
    \left(\down^{-1}f\right)(x) &=
    \sum_{\widehat{0}\lessdot y_1\lessdot y_2 \lessdot \cdots \lessdot y_k=x} f(y_k)\cdots f(y_2)f(y_1)  
        =f(x) \cdot \sum\limits_{y\lessdot x}\left(\down^{-1}f\right)(y)\\
    \left(\up^{-1}f\right)(x) &=
    \sum_{x=y_1\lessdot y_2 \lessdot \cdots \lessdot y_k\lessdot \widehat{1}} f(y_k)\cdots f(y_2)f(y_1)
          =\sum\limits_{y\gtrdot x}\left(\up^{-1}f\right)(y) \cdot f(x)
\end{align*}
\end{defn}

\begin{thm}[Analogue of~{\cite[Thm.~6.2]{einpropp}}]\label{thm:NOR-transfer}
For any finite poset $P$, $\NOR=\Theta \circ \up^{-1} \circ \down$.
\end{thm}

\begin{proof}
We will prove $(\Theta\up^{-1}\down f)(x)=(\NOR f)(x)$
for all $x\in P$ for which either side is defined (recall these are partial maps).  We induct on $P$ from top to bottom.
Let $x\in P$ and assume
every $y>x$ satisfies the induction hypothesis.

Then $$(\down f)(x) = f(x)
    \overline{\sum\limits_{y\lessdot x}f(y)}.$$
Now we apply $\up^{-1}$ to both sides.  Using the recursive
description for $\up^{-1}$, we obtain
\begin{align*}
    \left(\up^{-1}\down f\right)(x) &=
    \sum\limits_{y\gtrdot x}\left(\up^{-1}\down f\right)(y)\cdot (\down f)(x)\\
    &= \sum\limits_{y\gtrdot x}\left(\up^{-1}\down f\right)(y) \cdot f(x) \cdot \overline{\sum\limits_{y\lessdot x}f(y)}\\
    &= C\; \overline{\sumparlimitlower{y\gtrdot x} \left(\Theta \up^{-1}\down f\right)(y)} \cdot f(x) \cdot \overline{\sum\limits_{y\lessdot x}f(y)}.
\end{align*}
Next we apply $\Theta$ to both sides, which yields
\begin{align*}
    \left(\Theta \up^{-1}\down f\right)(x) &=
    C\cdot \overline{
    C\cdot \overline{\sumparlimitlower{y\gtrdot x}\left(\Theta \up^{-1}\down f\right)(y)} \cdot f(x) \cdot \overline{\sum\limits_{y\lessdot x}f(y)}}\\
    &=
    \left(\sum\limits_{y\lessdot x}f(y)\right) \cdot
    \overline{f(x)} \cdot \left( \sumparlimitlower{y\gtrdot x}\left(\Theta \up^{-1}\down f\right)(y) \right)\\
    &=
    \left(\sum\limits_{y\lessdot x}f(y)\right) \cdot
    \overline{f(x)} \cdot \left( \sumparlimitlower{y\gtrdot x}\left(\NOR f\right)(y) \right)
\end{align*}
by the induction hypothesis.
Note that the last expression equals $(\NOR f)(x)$
because it is exactly what we obtain
at $x$ just before we toggle at $x$ (where we have already toggled elements $y\gtrdot x$ but {\bf not} elements $y\lessdot x$).
\end{proof}

\subsection{Noncommutative Antichain Toggling and Rowmotion}

\begin{defn}
Let $v\in P$.  The \textbf{noncommutative antichain toggle} is the partial map $\tau_v:\bbs^P \dra \bbs^P$ defined as follows:\\\begin{footnotesize}
$\big(\tau_v(g)\big)(x) = 
\left\{\begin{array}{ll}
C \cdot \overline{
\sum\left\{\underbrace{g(y_{c-1})\cdots g(y_1)}_{\text{(indices decrease by 1)}}
\underbrace{g(y_{k})\cdots g(y_c)}_{\text{(indices decrease by 1)}} : \widehat{0}\lessdot y_1\lessdot y_2 \lessdot
\cdots \lessdot y_k\lessdot \widehat{1}, y_c=v
\right\}
} &\text{if $x=v$}\\\vspace{-4.5pt}\\
g(x) &\text{if $x\not=v$.}
\end{array}
\right.$\end{footnotesize}

The \textbf{noncommutative antichain elggot} is the partial map $\e_v:\bbs^P \dra \bbs^P$ defined as follows:\\\begin{footnotesize}
$\big(\e_v(g)\big)(x) = 
\left\{\begin{array}{ll}
C \cdot \overline{
\sum\left\{\underbrace{g(y_c)\cdots g(y_1)}_{\text{(indices decrease by 1)}}
\underbrace{g(y_{k})\cdots g(y_{c+1})}_{\text{(indices decrease by 1)}} : \widehat{0}\lessdot y_1\lessdot y_2
\lessdot
\cdots \lessdot y_k\lessdot \widehat{1}, y_c=v
\right\}
} &\text{if $x=v$}\\\vspace{-4.5pt}\\
g(x) &\text{if $x\not=v$.}
\end{array}
\right.$\end{footnotesize}
Let $\ntog_A(P)$ be the group generated by all noncommutative antichain toggles $\tau_v$ for $v\in P$, and let $\ntog_O(P)$ be the group generated by all noncommutative order toggles $T_v$.
\end{defn}

Commutativity of toggles and elggots is the same as in the (commutative) birational realm.

\begin{prop}
Let $u,v\in P$.  If $u \parallel v$, then
$\tau_u \tau_v= \tau_v \tau_u$, $ \e_u \e_v= \e_v \e_u$, $ \tau_u \e_v= \e_v \tau_u$.
\end{prop}
\begin{proof}
Analogous to the proof of Proposition~\ref{prop:basic-tau-prop}(2). 
\end{proof}

\begin{ex}
Consider the poset $P=[2]\times[3]$ below, with the generic labeling $g\in\kk^P$ by 
$u,v,w,x,y,z\in\bbs$. 
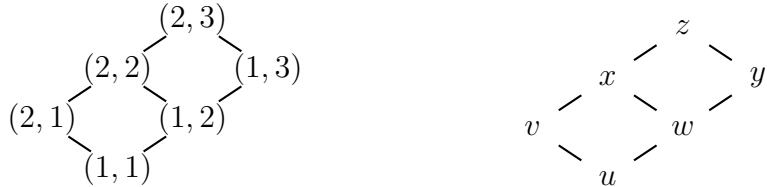
\begin{center}
\begin{tikzpicture}[yscale=2/3]
\node at (0,0) {$(1,1)$};
\node at (-1,1) {$(2,1)$};
\node at (1,1) {$(1,2)$};
\node at (0,2) {$(2,2)$};
\node at (2,2) {$(1,3)$};
\node at (1,3) {$(2,3)$};
\draw[thick] (-0.35,0.35) -- (-0.65,0.65);
\draw[thick] (0.35,0.35) -- (0.65,0.65);
\draw[thick] (1.35,1.35) -- (1.65,1.65);
\draw[thick] (0.35,2.35) -- (0.65,2.65);
\draw[thick] (-0.35,1.65) -- (-0.65,1.35);
\draw[thick] (0.35,1.65) -- (0.65,1.35);
\draw[thick] (1.35,2.65) -- (1.65,2.35);
\end{tikzpicture}
\hskip 1in
\begin{tikzpicture}[yscale=2/3]
\node at (0,0) {$u$};
\node at (-1,1) {$v$};
\node at (1,1) {$w$};
\node at (0,2) {$x$};
\node at (2,2) {$y$};
\node at (1,3) {$z$};
\draw[thick] (-0.35,0.35) -- (-0.65,0.65);
\draw[thick] (0.35,0.35) -- (0.65,0.65);
\draw[thick] (1.35,1.35) -- (1.65,1.65);
\draw[thick] (0.35,2.35) -- (0.65,2.65);
\draw[thick] (-0.35,1.65) -- (-0.65,1.35);
\draw[thick] (0.35,1.65) -- (0.65,1.35);
\draw[thick] (1.35,2.65) -- (1.65,2.35);
\end{tikzpicture}
\end{center}
\begin{itemize}
\item If we apply the toggle $\tau_{(1,1)}$,
we would change the label at $(1,1)$ to
$$C\cdot \overline{zxvu+zxwu+zywu}=
C\cdot \overline{u}\cdot\overline{(xv+xw+yw)}
\cdot\overline{z}.$$
\item If instead we apply the toggle $\tau_{(2,1)}$,
we would change the label at $(2,1)$ to
$$C\cdot \overline{uzxv}
=C\cdot \overline{v}\cdot \overline{x}\cdot \overline{z}\cdot \overline{u}.$$
\item If instead we apply the toggle $\tau_{(1,2)}$,
we would change the label at $(1,2)$ to
$$C\cdot \overline{uzxw+uzyw}
=C\cdot \overline{w}\cdot \overline{(x+y)}\cdot \overline{z}\cdot \overline{u}.$$
\item If instead we apply the toggle $\tau_{(2,2)}$,
we would change the label at $(2,2)$ to
$$C\cdot \overline{vuzx+wuzx}
=C\cdot \overline{x}\cdot \overline{z}\cdot \overline{u}\cdot \overline{(v+w)}.$$
\end{itemize}
\end{ex}

\begin{defn}
Let $(x_1,x_2,\dots,x_n)$ be any linear extension of a finite poset $P$.  Then the partial map $\NAR=\tau_{x_n}\cdots \tau_{x_2} \tau_{x_1}$, i.e., toggling at each element of $P$
from bottom to top, is called \textbf{noncommutative antichain rowmotion (NAR-motion)}.
\end{defn}

\begin{example}\label{ex:NAR[2]x[3]}
On the poset $P=[2]\times[3]$, $\NAR$ has order 5.
In Figure~\ref{fig:NAR[2]x[3]}, we show an orbit beginning with a generic labeling.
\end{example}

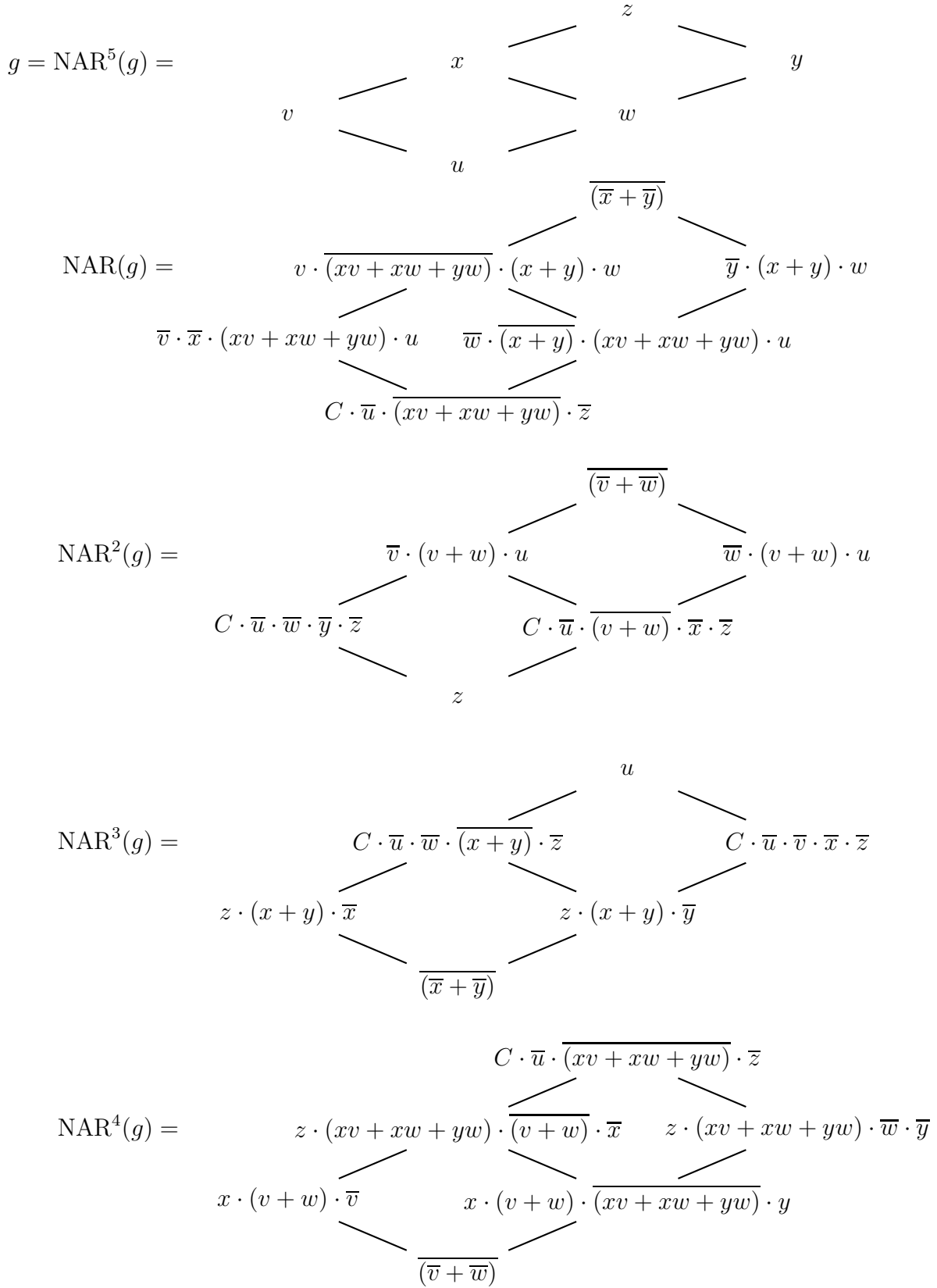
\begin{figure}
    \centering
\begin{tikzpicture}[xscale=26/9,yscale=11/9]
\begin{scope}[shift={(0,-0.6)},yscale=8/11]
\node at (-2.16,2) { $g=\NAR^5(g)=$};
\draw[thick] (-0.3, 1.7) -- (-0.7, 1.3);
\draw[thick] (0.3, 1.7) -- (0.7, 1.3);
\draw[thick] (-0.7, 0.7) -- (-0.3, 0.3);
\draw[thick] (0.7, 0.7) -- (0.3, 0.3);
\draw[thick] (0.7, 2.7) -- (0.3, 2.3);
\draw[thick] (1.3, 2.7) -- (1.7, 2.3);
\draw[thick] (1.7, 1.7) -- (1.3, 1.3);
\node at (1,3) {$  z  $};
\node at (0,2) {$  x  $};
\node at (2,2) {$  y  $};
\node at (-1,1) {$  v  $};
\node at (1,1) {$  w  $};
\node at (0,0) {$  u  $};
\end{scope}
\begin{scope}[shift={(0,-4)}]
\node at (-2,2) { $\NAR(g)=$};
\draw[thick] (-0.3, 1.7) -- (-0.7, 1.3);
\draw[thick] (0.3, 1.7) -- (0.7, 1.3);
\draw[thick] (-0.7, 0.7) -- (-0.3, 0.3);
\draw[thick] (0.7, 0.7) -- (0.3, 0.3);
\draw[thick] (0.7, 2.7) -- (0.3, 2.3);
\draw[thick] (1.3, 2.7) -- (1.7, 2.3);
\draw[thick] (1.7, 1.7) -- (1.3, 1.3);
\node at (1,3) {$  \overline{\left(\overline{x}+\overline{y}\right)}  $};
\node at (0,2) {$  v \cdot \overline{(xv+xw+yw)} \cdot (x+y) \cdot w  $};
\node at (2,2) {$  \overline{y} \cdot (x+y) \cdot w  $};
\node at (-1,1) {$  \overline{v} \cdot \overline{x} \cdot (xv+xw+yw) \cdot u   $};
\node at (1,1) {$  \overline{w} \cdot \overline{(x+y)} \cdot (xv+xw+yw) \cdot u  $};
\node at (0,0) {$  C \cdot \overline{u} \cdot \overline{(xv+xw+yw)} \cdot \overline{z}  $};
\end{scope}
\begin{scope}[shift={(0,-8)}]
\node at (-2,2) { $\NAR^2(g)=$};
\draw[thick] (-0.3, 1.7) -- (-0.7, 1.3);
\draw[thick] (0.3, 1.7) -- (0.7, 1.3);
\draw[thick] (-0.7, 0.7) -- (-0.3, 0.3);
\draw[thick] (0.7, 0.7) -- (0.3, 0.3);
\draw[thick] (0.7, 2.7) -- (0.3, 2.3);
\draw[thick] (1.3, 2.7) -- (1.7, 2.3);
\draw[thick] (1.7, 1.7) -- (1.3, 1.3);
\node at (1,3) {$  \overline{\left(\overline{v}+\overline{w}\right)}  $};
\node at (0,2) {$  \overline{v} \cdot (v+w) \cdot u  $};
\node at (2,2) {$  \overline{w} \cdot (v+w) \cdot u  $};
\node at (-1,1) {$  C\cdot \overline{u}\cdot\overline{w}
\cdot\overline{y}\cdot\overline{z}  $};
\node at (1,1) {$  C\cdot \overline{u}\cdot\overline{(v+w)}
\cdot\overline{x}\cdot\overline{z}  $};
\node at (0,0) {$  z  $};
\end{scope}
\begin{scope}[shift={(0,-12)}]
\node at (-2,2) { $\NAR^3(g)=$};
\draw[thick] (-0.3, 1.7) -- (-0.7, 1.3);
\draw[thick] (0.3, 1.7) -- (0.7, 1.3);
\draw[thick] (-0.7, 0.7) -- (-0.3, 0.3);
\draw[thick] (0.7, 0.7) -- (0.3, 0.3);
\draw[thick] (0.7, 2.7) -- (0.3, 2.3);
\draw[thick] (1.3, 2.7) -- (1.7, 2.3);
\draw[thick] (1.7, 1.7) -- (1.3, 1.3);
\node at (1,3) {$  u  $};
\node at (0,2) {$  C\cdot \overline{u} \cdot \overline{w} \cdot \overline{(x+y)} \cdot \overline{z}  $};
\node at (2,2) {$  C\cdot \overline{u} \cdot \overline{v} \cdot \overline{x} \cdot \overline{z}  $};
\node at (-1,1) {$  z \cdot (x+y) \cdot \overline{x}  $};
\node at (1,1) {$ z \cdot (x+y) \cdot \overline{y}  $};
\node at (0,0) {$\overline{\left(\overline{x}+\overline{y}\right)} $};
\end{scope}
\begin{scope}[shift={(0,-16)}]
\node at (-2,2) { $\NAR^4(g)=$};
\draw[thick] (-0.3, 1.7) -- (-0.7, 1.3);
\draw[thick] (0.3, 1.7) -- (0.7, 1.3);
\draw[thick] (-0.7, 0.7) -- (-0.3, 0.3);
\draw[thick] (0.7, 0.7) -- (0.3, 0.3);
\draw[thick] (0.7, 2.7) -- (0.3, 2.3);
\draw[thick] (1.3, 2.7) -- (1.7, 2.3);
\draw[thick] (1.7, 1.7) -- (1.3, 1.3);
\node at (1,3) {$  C\cdot \overline{u} \cdot \overline{(xv+xw+yw)} \cdot \overline{z}$};
\node at (0,2) {$  z \cdot (xv+xw+yw) \cdot \overline{(v+w)}\cdot\overline{x}  $};
\node at (2,2) {$  z \cdot (xv+xw+yw) \cdot \overline{w}\cdot\overline{y}  $};
\node at (-1,1) {$  x \cdot (v+w) \cdot \overline{v} $};
\node at (1,1) {$  x \cdot (v+w) \cdot \overline{(xv+xw+yw)} \cdot y  $};
\node at (0,0) {$ \overline{\left(\overline{v}+\overline{w}\right)}  $};
\end{scope}
\end{tikzpicture}
\caption{An orbit of $\NAR$ starting on a generic labeling $g\in\bbs^P$, for $P=[2]\times[3]$.  We observe that the order of $\NAR$, like $\BAR$, is 5 on this poset.}
\label{fig:NAR[2]x[3]}
\end{figure}
We now define specific elements of $\ntog_A(P)$ that mimic the action of order toggles.
\begin{defn}
For $v\in P$, let $T_v^* =\e_{v_1}\e_{v_2}\cdots \e_{v_k} \tau_v \tau_{v_k}\cdots \tau_{v_2}\tau_{v_1}\in\ntog_A(P)$
where $v_1,\dots,v_k$ are the elements of $P$ covered by $v$.  (In the case that $v$ is a minimal element of $P$, $k=0$ and $T_v^*=\tau_v$.)
Let $E_v^*=\left(T_v^*\right)^{-1} = \e_{v_1}\e_{v_2}\cdots \e_{v_k} \e_v \tau_{v_k}\cdots \tau_{v_2}\tau_{v_1}.$
\end{defn}

The following lemma (whose proof is straightforward from the definitions) will be used in the proof of Theorem~\ref{thm:T-star-NC}.

\begin{lem}\label{lem:meteor gorge}
Let $g\in\bbs^P$ and $v\in P$.
Then
\begin{align*}
    (\tau_v g)(v)&=C\cdot \overline{(\down^{-1}g)(v)\cdot (\up^{-1}g)(v)} \cdot g(v)\\
    &=C\cdot \overline{(\up^{-1}g)(v)} \cdot \overline{(\down^{-1}g)(v)}\cdot g(v),\\
     (\e_v g)(v)&=C\cdot g(v)\cdot \overline{(\down^{-1}g)(v)\cdot (\up^{-1}g)(v)}\\
    &=C\cdot g(v) \cdot \overline{(\up^{-1}g)(v)} \cdot \overline{(\down^{-1}g)(v)}.   
\end{align*}
\end{lem}
The following result uses the transfer maps to explicitly mimic the action of order toggles using the products of antichain toggles defined above and similarly for elggots.
\begin{thm}[Analogue of Theorem~\ref{thm:T-star}]\label{thm:T-star-NC}
Let $v\in P$.  Then the following diagrams commute on the domains in which the maps are defined.
\begin{center}
\begin{tikzpicture}[yscale=.6]
\node at (0,1.8) {$\bbs^P$};
\node at (0,0) {$\bbs^P$};
\node at (0,-1.8) {$\bbs^P$};
\node at (3.25,1.8) {$\bbs^P$};
\node at (3.25,0) {$\bbs^P$};
\node at (3.25,-1.8) {$\bbs^P$};
\draw[semithick, dashed, ->] (0,1.3) -- (0,0.5);
\node[left] at (0,0.9) {$\up^{-1}$};
\draw[semithick, dashed, ->] (0.7,-1.8) -- (2.5,-1.8);
\node[below] at (1.5,-1.8) {$T_v$};
\draw[semithick, dashed, ->] (0.7,1.8) -- (2.5,1.8);
\node[above] at (1.5,1.8) {$T_v^*$};
\draw[semithick, dashed, ->] (3.25,1.3) -- (3.25,0.5);
\node[right] at (3.25,0.9) {$\up^{-1}$};
\draw[semithick, dashed, ->] (0,-0.5) -- (0,-1.3);
\node[left] at (0,-0.9) {$\Theta$};
\draw[semithick, dashed, ->] (3.25,-0.5) -- (3.25,-1.3);
\node[right] at (3.25,-0.9) {$\Theta$};
\begin{scope}[shift=({7,0)})]
\node at (0,1.8) {$\bbs^P$};
\node at (0,0) {$\bbs^P$};
\node at (0,-1.8) {$\bbs^P$};
\node at (3.25,1.8) {$\bbs^P$};
\node at (3.25,0) {$\bbs^P$};
\node at (3.25,-1.8) {$\bbs^P$};
\draw[semithick, dashed, ->] (0,1.3) -- (0,0.5);
\node[left] at (0,0.9) {$\up^{-1}$};
\draw[semithick, dashed, ->] (0.7,-1.8) -- (2.5,-1.8);
\node[below] at (1.5,-1.8) {$E_v$};
\draw[semithick, dashed, ->] (0.7,1.8) -- (2.5,1.8);
\node[above] at (1.5,1.8) {$E_v^*$};
\draw[semithick, dashed, ->] (3.25,1.3) -- (3.25,0.5);
\node[right] at (3.25,0.9) {$\up^{-1}$};
\draw[semithick, dashed, ->] (0,-0.5) -- (0,-1.3);
\node[left] at (0,-0.9) {$\Theta$};
\draw[semithick, dashed, ->] (3.25,-0.5) -- (3.25,-1.3);
\node[right] at (3.25,-0.9) {$\Theta$};
\end{scope}
\end{tikzpicture}
\end{center}
\end{thm}

\begin{proof}
The right commutative diagram clearly follows from the left, so we will only prove the left one.

Let $g\in\bbs^P$.  We must show that
$\Theta \up^{-1} (T_v^*g) = T_v(\Theta\up^{-1} g)$.

Suppose $v$ is a minimal element of $P$.  Then $T_v^*=\tau_v$.  Note from the definitions that 
$$(\Theta\up^{-1} g)(x) =
C\cdot \overline{\sum\left\{ g(y_k)\cdots g(y_2)g(y_1):
    x=y_1\lessdot y_2 \lessdot \cdots \lessdot y_k\lessdot \widehat{1} \right\}}    .$$
Thus, $(\Theta\up^{-1} g)(x)$ only depends on the labels $g(y)$ for any $y\geq x$.
Let $x\not=v$.  Since $\tau_v$ only affects the label at $v$, and $x\not\leq v$ (by minimality of $v$), it follows that 
$$
\big(\Theta\up^{-1} (\tau_v g)\big)(x)=
(\Theta\up^{-1} g)(x)=
\big(T_v(\Theta\up^{-1} g)\big)(x).$$
Now we must confirm that
$\big(\Theta\up^{-1} (\tau_v g)\big)(v)=\big(T_v(\Theta\up^{-1} g)\big)(v).$
We have
\begin{align*}
&\phantom{=l}\big(\Theta\up^{-1} (\tau_v g)\big)(v)\\
&= C\cdot \overline{\big(\up^{-1} (\tau_v g)\big)(v)}\\
&= C\cdot \overline{\sum\limits_{y\gtrdot v}\big(\up^{-1}(\tau_v g)\big)(y)\cdot(\tau_v g)(v)}\\
&= C\cdot \overline{\sum\limits_{y\gtrdot v}\big(\up^{-1}g\big)(y)\cdot(\tau_v g)(v)}\\
&= C\cdot \overline{\sum\limits_{y\gtrdot v}\big(\up^{-1}g\big)(y)\cdot C\cdot\overline{(\up^{-1}g)(v)}}
&\text{(apply definitions of $\tau_v$ and $\up^{-1}$ for minimal $v$)}
\\
&=C\cdot \overline{(\up^{-1}g)(v)}\left(\overline{C\cdot \sum\limits_{y\gtrdot v}\big(\up^{-1}g\big)(y)}\right)\\
&=(\Theta\up^{-1}g)(v)\left(\sumparlimitlower{y\gtrdot v}\big(\Theta\up^{-1}g\big)(y)\right)\\
&=\big(T_v(\Theta\up^{-1} g)\big)(v)
&\text{(since $v$ is minimal so there is no $y\lessdot v$ in $P$).}
\end{align*}

Now assume $v$ is not minimal in $P$.  Let $v_1,\dots,v_k$ be the elements that $v$ covers.  Let
\begin{align*}
&&g' &= \tau_v\tau_{v_k}\cdots \tau_{v_1} g,
&g'' &= \e_{v_1}\cdots \e_{v_k} g' = T_v^* g,\\
f &= \Theta\up^{-1}g,
&f' &= \Theta\up^{-1}g',
&f'' &= \Theta\up^{-1}g''.\\
\end{align*}

The goal is to show that $f''=T_v f$.
Note that $g$, $g'$, and $g''$ can only possibly differ in the labels of $v,v_1,v_2,\cdots,v_k$.  From the definitions of $\Theta$ and $\up^{-1}$, we note that
$T_v f$ and $f''$ can only possibly differ in the labels of elements $\leq v$.

We begin by proving $f''(v)=(T_v f)(v)$.  Since $v_1,\cdots,v_k$ are pairwise incomparable
(so each chain can only contain at most one of them), for $1\leq j\leq k$,
\begin{align*}
\left(\tau_{v_1}\cdots \tau_{v_k} g\right)(v_j)
&=
C\cdot \overline{(\down^{-1}g)(v_j)\cdot(\up^{-1}g)(v_j)}\cdot g(v_j)
&\text{(from Lemma~\ref{lem:meteor gorge})}\\
&=
C\cdot \overline{\overline{g(v_j)}\cdot(\down^{-1}g)(v_j)\cdot(\up^{-1}g)(v_j)}\\
&=C\cdot \overline{\sum\limits_{y\lessdot v_j}(\down^{-1}g)(v_j)\cdot(\up^{-1}g)(v_j)}\\
&=C\cdot \overline{(\up^{-1}g)(v_j)}\cdot\overline{\sum\limits_{y\lessdot v_j}(\down^{-1}g)(v_j)}\\
&=f(v_j)\cdot\overline{\sum\limits_{y\lessdot v_j}(\down^{-1}g)(v_j)}.
\end{align*}
We restate the above fact
\begin{equation}\label{eq:the lost city}
   \left(\tau_{v_1}\cdots \tau_{v_k} g\right)(v_j)
   = f(v_j)\cdot\overline{\sum\limits_{y\lessdot v_j}(\down^{-1}g)(v_j)}
\end{equation}
as an equation we will reference later.

Then to get $g'(v)$, we apply $\tau_v$ to $\left( \tau_{v_1} \cdots \tau_{v_k}g \right)(v)$.  Lemma~\ref{lem:meteor gorge} gives
\begin{align*}
g''(v)=g'(v)&=
C\cdot \overline{\big(\down^{-1}( \tau_{v_1} \cdots \tau_{v_k}g)\big)(v)\cdot\big(\up^{-1}( \tau_{v_1} \cdots \tau_{v_k}g)\big)(v)}\cdot ( \tau_{v_1} \cdots \tau_{v_k}g)(v)
\\
&= C\cdot
\overline{\overline{( \tau_{v_1} \cdots \tau_{v_k}g)(v)}\cdot\big(\down^{-1}( \tau_{v_1} \cdots \tau_{v_k}g)\big)(v)\cdot\big(\up^{-1}( \tau_{v_1} \cdots \tau_{v_k}g)\big)(v)}
\\
&= C\cdot
\overline{\left(\sum\limits_{v_j\lessdot v}\big(\down^{-1}( \tau_{v_1} \cdots \tau_{v_k}g)\big)(v)\right)\cdot\big(\up^{-1}g\big)(v)}
\\
&=C\cdot \overline{\big(\up^{-1}g\big)(v)}\cdot
\overline{\sum\limits_{v_j\lessdot v}\big(\down^{-1}( \tau_{v_1} \cdots \tau_{v_k}g)\big)(v)}
\\
&=
f(v)\cdot
\overline{\sum\limits_{v_j\lessdot v}\big(\down^{-1}( \tau_{v_1} \cdots \tau_{v_k}g)\big)(v)}
\\
&=
f(v)\cdot
\overline{\sum\limits_{v_j\lessdot v}\left(\big(\tau_{v_1} \cdots \tau_{v_k}g\big)(v_j)
\sum\limits_{y\leq v_j}\big(\down^{-1}( \tau_{v_1} \cdots \tau_{v_k}g)\big)(y)\right)}
\\
&=
f(v)\cdot
\overline{\sum\limits_{v_j\lessdot v}\left(\big(\tau_{v_1} \cdots \tau_{v_k}g\big)(v_j)
\sum\limits_{y\leq v_j}\big(\down^{-1}g\big)(y)\right)}
\\
&=
f(v)\cdot
\overline{\sum\limits_{v_j\lessdot v}\underbrace{f(v_j)}_
\text{from Eq.~(\ref{eq:the lost city})}}.
\end{align*}

Then using the recursive descripion of $\up^{-1}$,
\begin{align*}
f''(v) &= C\cdot \overline{(\up^{-1}g'')(v)}
=
C\cdot \overline{\sum\limits_{y\gtrdot v}(\up^{-1}g'')(v)\cdot g''(v)}
=
C\cdot \overline{\sum\limits_{y\gtrdot v}(\up^{-1}g)(y)\cdot g''(v)}\\
&=
\overline{g''(v)}\cdot C\cdot  \overline{\sum\limits_{y\gtrdot v}(\up^{-1}g)(y)}
=
\overline{g''(v)}\cdot \sumparlimitlower{y\gtrdot v}(\Theta\up^{-1}g)(y)
=
\overline{f(v)\cdot
\overline{\sum\limits_{v_j\lessdot v}f(v_j)}}\cdot \sumparlimitlower{y\gtrdot v}f(y)\\
&=\sum\limits_{v_j\lessdot v}f(v_j) \cdot\overline{f(v)}\cdot
\sumparlimitlower{y\gtrdot v}f(y)
=
(T_v f)(v).
\end{align*}

We noted earlier in the proof that
$(T_v f)(x)=f''(x)=f(x)$ for all $x>v$.  Now we have proven
$(T_v f)(v)=f''(v)$.  What remains to be shown is that
$(T_v f)(x)=f''(x)=f(x)$ for all $x<v$, which we will now prove
using downward induction on $x$.  We begin with the base case $x\lessdot v$.  For $1\leq j\leq k$,
\begin{align*}
g''(v_j) &= (\e_{v_j}g')(v_j)\\
&=
C\cdot g'(v_j) \cdot \overline{(\down^{-1}g')(v_j)\cdot(\up^{-1}g')(v_j)}
& \text{(from Lemma~\ref{lem:meteor gorge})}
\\
&= C\cdot
\overline{(\down^{-1}g')(v_j)\cdot(\up^{-1}g')(v_j) \cdot
\overline{g'(v_j)}}
\\
\end{align*}\begin{align*}
&= C\cdot
\overline{\left(g'(v_j)\sum\limits_{y\lessdot v_j}\left(\down^{-1}g'\right)(y)\right)\cdot
\left(\sum\limits_{y\gtrdot v_j}\left(\up^{-1}g'\right)(y)\cdot g'(v_j)\right) \cdot
\overline{g'(v_j)}}
\\
&= C\cdot
\overline{g'(v_j)\cdot\sum\limits_{y\lessdot v_j}\left(\down^{-1}g'\right)(y)\cdot
\sum\limits_{y\gtrdot v_j}\left(\up^{-1}g'\right)(y)}
\\
&=
\overline{
f(v_j)\cdot\overline{\sum\limits_{y\lessdot v_j}(\down^{-1}g)(v_j)}
\cdot\sum\limits_{y\lessdot v_j}\left(\down^{-1}g'\right)(y)
\cdot\sum\limits_{y\gtrdot v_j}\left(\up^{-1}g'\right)(y)}
&  \text{(from Eq.~(\ref{eq:the lost city}))}
\\
&=
\overline{
f(v_j)\sum\limits_{y\gtrdot v_j}\left(\up^{-1}g'\right)(y)}
\\
&=
\overline{
f(v_j)\sum\limits_{y\gtrdot v_j}\left(\up^{-1}g''\right)(y)}
\end{align*}\vspace{-0.1 in}\\
where the last equality is because $\up^{-1}g', \up^{-1}g''$ are the same for $y>v_j$.
Using the above fact in the fourth equality below,
\begin{align*}
f''(v_j)
&=C\cdot \overline{(\up^{-1}g'')(v_j)}
=\overline{\sum\limits_{y\gtrdot v_j}(\up^{-1}g'')(y)\cdot g''(v_j)}
= C\cdot 
\overline{g''(v_j)}\cdot \overline{\sum\limits_{y\gtrdot v_j}(\up^{-1}g'')(y)}\\
&=f(v_j)\sum\limits_{y\gtrdot v_j}\left(\up^{-1}g''\right)(y)\cdot \overline{\sum\limits_{y\gtrdot v_j}(\up^{-1}g'')(y)}
= f(v_j)=(T_v f)(v_j).
\end{align*}
Now let $x<v$ and $x\not\in\{v_1,\dots,v_k\}$.
Assume that $(T_v f)(y)=f(y)=f''(y)$ for every $y$ covering $x$ (which cannot include $y=v$ because $x\not\in\{v_1,\dots,v_k\}$).
Also since $x\not\in\{v,v_1,\dots,v_k\}$, recall that $g(x)=g''(x)$.  So

\begin{align*}
f''(x) &=
C\cdot \overline{(\up^{-1}g'')(x)}
=
C\cdot \overline{\sum\limits_{y\gtrdot x}(\up^{-1}g'')(y)
\cdot g''(x)}
=
\overline{\sum\limits_{y\gtrdot x}(\Theta f'')(y)
\cdot g''(x)}\\
&=
\overline{\sum\limits_{y\gtrdot x}(\Theta f)(y)
\cdot g(x)}
=
\overline{\sum\limits_{y\gtrdot x}(\up^{-1}g)(y)
\cdot g(x)}
=
\overline{(\up^{-1}g)(x)}
=f(x)
=(T_v f)(x).
\end{align*}
\end{proof}

We continue the group-theoretic approach of~\cite{antichain-toggling} to prove
$\NAR=\down\circ\Theta\circ\up^{-1}$.
The proofs here are similar to those
in~\cite{antichain-toggling}, but modified as
toggles are no longer involutions.  The next two definitions and theorem allow us to mimic antichain toggles by order toggles.  

\begin{defn}
For $S\subseteq P$  let $\eta_S:=T_{x_1}T_{x_2}\cdots T_{x_k}$
where $(x_1,x_2,\dots,x_k)$ is a linear extension of the subposet 
$\{x\in P\;|\;x<y \text{ for some }y\in S\}$.  (In the special case where every element of $S$ is minimal in $P$,
$\eta_S$ is the identity.)  For $v\in P$, we write $\eta_v:=\eta_{\{v\}}$. 
\end{defn}

\begin{defn}
For $v\in P$, define $\tau_v^*\in\ntog_O(P)$ as $\tau_v^* := \eta_v T_v \eta_v^{-1}$.  Let $\e_v^*=\left(\tau_v^*\right)^{-1} = \eta_v E_v \eta_v^{-1}$.
\end{defn}

\begin{thm}[Analogue of~\ref{thm:tau-star} and~{\cite[Thm.~2.19]{antichain-toggling}}]\label{thm:tau-star-NC}
Let $v\in P$.  Then the following diagrams commute on the domains in which the maps are defined.
\begin{center}
\begin{tikzpicture}[yscale=.6]
\node at (0,1.8) {$\bbs^P$};
\node at (0,0) {$\bbs^P$};
\node at (0,-1.8) {$\bbs^P$};
\node at (3.25,1.8) {$\bbs^P$};
\node at (3.25,0) {$\bbs^P$};
\node at (3.25,-1.8) {$\bbs^P$};
\draw[semithick, dashed, ->] (0,1.3) -- (0,0.5);
\node[left] at (0,0.9) {$\up^{-1}$};
\draw[semithick, dashed, ->] (0.7,-1.8) -- (2.5,-1.8);
\node[below] at (1.5,-1.8) {$\tau_v^*$};
\draw[semithick, dashed, ->] (0.7,1.8) -- (2.5,1.8);
\node[above] at (1.5,1.8) {$\tau_v$};
\draw[semithick, dashed, ->] (3.25,1.3) -- (3.25,0.5);
\node[right] at (3.25,0.9) {$\up^{-1}$};
\draw[semithick, dashed, ->] (0,-0.5) -- (0,-1.3);
\node[left] at (0,-0.9) {$\Theta$};
\draw[semithick, dashed, ->] (3.25,-0.5) -- (3.25,-1.3);
\node[right] at (3.25,-0.9) {$\Theta$};
\begin{scope}[shift=({7,0)})]
\node at (0,1.8) {$\bbs^P$};
\node at (0,0) {$\bbs^P$};
\node at (0,-1.8) {$\bbs^P$};
\node at (3.25,1.8) {$\bbs^P$};
\node at (3.25,0) {$\bbs^P$};
\node at (3.25,-1.8) {$\bbs^P$};
\draw[semithick, dashed, ->] (0,1.3) -- (0,0.5);
\node[left] at (0,0.9) {$\up^{-1}$};
\draw[semithick, dashed, ->] (0.7,-1.8) -- (2.5,-1.8);
\node[below] at (1.5,-1.8) {$\e_v^*$};
\draw[semithick, dashed, ->] (0.7,1.8) -- (2.5,1.8);
\node[above] at (1.5,1.8) {$\e_v$};
\draw[semithick, dashed, ->] (3.25,1.3) -- (3.25,0.5);
\node[right] at (3.25,0.9) {$\up^{-1}$};
\draw[semithick, dashed, ->] (0,-0.5) -- (0,-1.3);
\node[left] at (0,-0.9) {$\Theta$};
\draw[semithick, dashed, ->] (3.25,-0.5) -- (3.25,-1.3);
\node[right] at (3.25,-0.9) {$\Theta$};
\end{scope}
\end{tikzpicture}
\end{center}
\end{thm}

To prove Theorem~\ref{thm:tau-star-NC}, we first need a lemma.

\begin{lemma}[Analogue of~{\cite[Lemma~2.21]{antichain-toggling}}]\label{lem:tau-star-NC}
Let $v_1,\dots,v_k$ be pairwise incomparable elements of $P$.  Then for $1\leq i\leq k$,
$$\tau_{v_1}^*\tau_{v_2}^*\dots \tau_{v_i}^* = \eta_{\{v_1,\dots,v_i\}} T_{v_1}T_{v_2}\dots T_{v_i} \eta_{\{v_1,\dots,v_i\}}^{-1}.$$
\end{lemma}

\begin{proof}
This proof is similar to that of~\cite[Lemma 2.21]{antichain-toggling}.

This claim is true by definition for $i=1$ and we proceed inductively.  Suppose it is true for some given $i\leq k-1$.  Let
\begin{itemize}
\item $x_1,\dots,x_a$ be the elements that are both less than $v_{i+1}$ and less than at least one of $v_1,\dots,v_i$,
\item $y_1,\dots,y_b$ be the elements that are less than at least one of $v_1,\dots,v_i$ but not less than $v_{i+1}$,
\item $z_1,\dots,z_c$ be the elements that are less than $v_{i+1}$ but not less than any of $v_1,\dots,v_i$.
\end{itemize}
Clearly, it is possible for one or more of the sets $\{x_1,\dots,x_a\}$, $\{y_1,\dots,y_b\}$, and $\{z_1,\dots,z_c\}$ to be empty.  For example, if $b=0$, then the product $T_{y_1}\cdots T_{y_b}$ is just the identity.

Note than none of $y_1,\dots,y_b$ are less than any of $x_1,\dots,x_a$ because any element less than some $x_j$ is automatically less than $v_{i+1}$.  By similar reasoning, none of $z_1,\dots,z_c$ are less than any of $x_1,\dots,x_a$ either.
Also any pair $y_m,z_n$ are incomparable, because $y_m \leq z_n$ would imply $y_m<v_{i+1}$, while $z_n \leq y_m$ would imply $z_n$ is less than some $v_j$.
By transitivity and the pairwise incomparability of $v_1,\dots,v_{i+1}$, each $y_m$ is incomparable with $v_{i+1}$, and each $z_m$ is incomparable with any of $v_1,\dots,v_i$.

We will pick the indices so that $(x_1,\dots, x_a)$, $(y_1,\dots,y_b)$, and $(z_1,\dots,z_c)$ are linear extensions of the subposets $\{x_1,\dots,x_a\}$, $\{y_1,\dots,y_b\}$, and $\{z_1,\dots,z_c\}$, respectively.  Then we have the following
\begin{itemize}
\item $(x_1,\dots,x_a,y_1,\dots,y_b)$ is a linear extension of $\big\{p\in P\;|\; p<q,q\in\{v_1,\dots, v_i\}\big\}$.
\begin{itemize}
\item This yields $\eta_{\{v_1,\dots,v_i\}}=T_{x_1}\cdots T_{x_a}T_{y_1}\cdots T_{y_b}$.
\end{itemize}
\item $(x_1,\dots,x_a,z_1,\dots,z_c)$ is a linear extension of $\{p\in P \;|\; p<e_{i+1}\}$.
\begin{itemize}
\item This yields $\eta_{v_{i+1}}=T_{x_1}\cdots T_{x_a}T_{z_1}\cdots T_{z_c}$.
\end{itemize}
\item $(x_1,\dots,x_a,y_1,\dots,y_b,z_1,\dots,z_c)$ and 
$(x_1,\dots,x_a,z_1,\dots,z_c,y_1,\dots,y_b)$ are both linear extensions of $\big\{p\in P \;|\; p<q,q\in\{v_1,\dots, v_{i+1}\}\big\}$.
\begin{itemize}
\item This yields $\eta_{\{v_1,\dots,v_{i+1}\}}=T_{x_1}\cdots T_{x_a}T_{y_1}\cdots T_{y_b}T_{z_1}\cdots T_{z_c}$.
\end{itemize}
\end{itemize}

Using the induction hypothesis,
\begin{align*}
&  \tau_{v_1}^*\dots \tau_{v_i}^* \tau_{v_{i+1}}^* \\
&=
\eta_{\{v_1,\dots,v_i\}} T_{v_1}\mydots T_{v_i} \eta_{\{v_1,\dots,v_i\}}^{-1}\eta_{v_{i+1}} T_{v_{i+1}} \eta_{v_{i+1}}^{-1}\\
&=
T_{x_1}\mydots T_{x_a} T_{y_1} \mydots T_{y_b} T_{v_1} \mydots T_{v_i} E_{y_b} \mydots E_{y_1} E_{x_a}\mydots E_{x_1} T_{x_1}\mydots T_{x_a} T_{z_1} \mydots T_{z_c} T_{v_{i+1}} E_{z_c} \mydots E_{z_1} E_{x_a}\mydots E_{x_1}\\
&=
T_{x_1}\mydots T_{x_a} T_{y_1} \mydots T_{y_b} T_{v_1} \mydots T_{v_i} E_{y_b} \mydots E_{y_1} T_{z_1} \mydots T_{z_c} T_{v_{i+1}} E_{z_c} \mydots E_{z_1} E_{x_a}\mydots E_{x_1}\\
&=
T_{x_1}\mydots T_{x_a} T_{y_1} \mydots T_{y_b} T_{v_1} \mydots T_{v_i} T_{z_1} \mydots T_{z_c} T_{v_{i+1}} E_{z_c} \mydots E_{z_1} E_{y_b} \mydots E_{y_1} E_{x_a}\mydots E_{x_1}\\
&=
T_{x_1}\mydots T_{x_a} T_{y_1} \mydots T_{y_b} T_{z_1} \mydots T_{z_c} T_{v_1} \mydots T_{v_i} T_{v_{i+1}} E_{z_c} \mydots E_{z_1} E_{y_b} \mydots E_{y_1} E_{x_a}\mydots E_{x_1}\\
&=
\eta_{\{v_1,\dots,v_{i+1}\}} T_{v_1}\mydots T_{v_i}
T_{v_{i+1}}\eta_{\{v_1,\dots,v_{i+1}\}}^{-1}
\end{align*}
where each commutation above is between toggles/elggots for pairwise incomparable elements.
\end{proof}

We are now ready to prove Theorem~\ref{thm:tau-star-NC}.

\begin{proof}[Proof of Theorem~\ref{thm:tau-star-NC}]
The right commutative diagram clearly follows from the left, so we will simply prove the left.

We use induction on $v$.  If $v$ is a minimal element of $P$, then $\tau_v^*=T_v$ and $T_v^*=\tau_v$, so the diagram commutes by Theorem~\ref{thm:T-star-NC}.

Now suppose $v$ is not minimal.  Let $v_1,\dots,v_k$ be the elements of $P$ covered by $v$, and suppose that the theorem is true for every $v_i$.  That is, for every $A\in\bbs^P$ with $I=\Theta\up^{-1}A$, we have $\Theta\up^{-1}\big(\tau_{v_i}(A)\big)=\tau_{v_i}^*(I)$.

Then $$\Theta\up^{-1}\big(\tau_{v_1}\tau_{v_2}\dots \tau_{v_k}(A)\big)=\tau_{v_1}^*\tau_{v_2}^*\dots \tau_{v_k}^*(I)
= \eta_{\{v_1,\dots,v_k\}} T_{v_1}T_{v_2}\dots T_{v_k} \eta_{\{v_1,\dots,v_k\}}^{-1}(I)$$
and 
$$\Theta\up^{-1}\big(\e_{v_1}\e_{v_2}\dots \e_{v_k}(A)\big)=\e_{v_1}^*\e_{v_2}^*\dots \e_{v_k}^*(I)
= \eta_{\{v_1,\dots,v_k\}} E_{v_1}E_{v_2}\dots E_{v_k} \eta_{\{v_1,\dots,v_k\}}^{-1}(I)$$
by Lemma~\ref{lem:tau-star-NC}.

Throughout this proof, we let $A\in\bbs^P$ and
$I=\Theta\up^{-1}A\in\bbs^P$.  Think $A$ for ``antichain'' and $I$ as referring to ``the order ideal generated by $A$''
if we were in the combinatorial realm.

From the definition of $T_v^*$, it follows that $\tau_{v_1}\tau_{v_2}\cdots \tau_{v_k} T_v^* \e_{v_k}\cdots\e_{v_2} \e_{v_1} 
=\tau_v$.
Then
\begin{align*}
  \Theta\up^{-1}\big(\tau_v(A)\big)&=  
\Theta\up^{-1}\big(\tau_{v_1}\tau_{v_2}\cdots \tau_{v_k} T_v^* \e_{v_k}\cdots\e_{v_2} \e_{v_1}  (A)\big)\\
&=
\eta_{\{v_1,\dots,v_k\}} T_{v_1}T_{v_2}\cdots T_{v_k} \eta_{\{v_1,\dots,v_k\}}^{-1} T_v \eta_{\{v_1,\dots,v_k\}} E_{v_k}\cdots E_{v_2}E_{v_1} \eta_{\{v_1,\dots,v_k\}}^{-1}(I)
\end{align*}
by Theorem~\ref{thm:T-star-NC} (for $T_v^*$) and the induction hypothesis (for $\tau_{v_1}\tau_{v_2}\cdots \tau_{v_k}$ and
$\e_{v_k}\cdots\e_{v_2} \e_{v_1}$).  Thus, it suffices to show that
\begin{equation}\label{eq:tau-star-NC}
\eta_{\{v_1,\dots,v_k\}} T_{v_1}T_{v_2}\cdots T_{v_k} \eta_{\{v_1,\dots,v_k\}}^{-1} T_v \eta_{\{v_1,\dots,v_k\}} E_{v_k}\cdots E_{v_2}E_{v_1} \eta_{\{v_1,\dots,v_k\}}^{-1}=\tau_v^*.
\end{equation}
The toggles in the product $\eta_{\{v_1,\dots,v_k\}}$ correspond to elements strictly less than $v_1,\dots,v_k$; none of these cover nor are covered by $v$.  Thus
we can commute $T_v$ with $\eta_{\{v_1,\dots,v_k\}}$ on the left side of~(\ref{eq:tau-star-NC}) and then cancel $\eta_{\{v_1,\dots,v_k\}}^{-1} \eta_{\{v_1,\dots,v_k\}}$.
Thus the left side of~(\ref{eq:tau-star-NC}) is
$$\eta_{\{v_1,\dots,v_k\}} T_{v_1}T_{v_2}\cdots T_{v_k} T_v E_{v_k} \cdots E_{v_2}E_{v_1} \eta_{\{v_1,\dots,v_k\}}^{-1}.$$
Note that $$\{x\in P \;|\; x<v\}=\{x\in P \;|\; x<y \text{ for some
}y\in\{v_1,\dots,v_k\}\}\cup \{v_1,\dots,v_k\}$$ where the union is disjoint and that
$v_1,\dots,v_k$ are maximal elements of this set.  Thus for any linear extension
$(x_1,\dots,x_n)$ of $\big\{x\in P \;|\; x<y \text{ for some }y\in\{v_1,\dots,v_k\}\big\}$, a linear extension of 
$\{x\in P \;|\; x<v\}$ is
$(x_1,\dots,x_n,v_1,\dots,v_k)$.
So
$\eta_{\{v_1,\dots,v_k\}} T_{v_1}T_{v_2}\cdots T_{v_k}
=\eta_v$
and
$E_{v_k}\cdots E_{v_2}E_{v_1}\eta_{\{v_1,\dots,v_k\}}^{-1}
=\eta_v^{-1}$
which means the left side of~(\ref{eq:tau-star-NC}) is $\eta_v T_v \eta_v^{-1}= \tau_v^*$,
the same as the right side.
\end{proof}

The following is a corollary of Theorems~\ref{thm:T-star-NC} and~\ref{thm:tau-star-NC}.

\begin{cor}\label{cor:iso-NC}
There is an isomorphism from $\ntog_A(P)$ to $\ntog_O(P)$ given by $\tau_v\mapsto \tau_v^*$, with inverse given by $T_v\mapsto T_v^*$.
\end{cor}

\begin{thm}\label{thm:NAR-transfer}
For any finite poset $P$,  $\NAR=\down\circ \Theta \circ \up^{-1}$.
\end{thm}

Proving Theorem~\ref{thm:NAR-transfer}
is equivalent to proving the following diagram commutes on the domains in which the maps are defined.
This is because of Theorem~\ref{thm:NOR-transfer}
which says $\NOR=\Theta \circ \up^{-1} \circ \down$.
\begin{center}
\begin{tikzpicture}[yscale=2/3]
\node at (0,1.8) {$\bbs^P$};
\node at (0,0) {$\bbs^P$};
\node at (0,-1.8) {$\bbs^P$};
\node at (3.25,1.8) {$\bbs^P$};
\node at (3.25,0) {$\bbs^P$};
\node at (3.25,-1.8) {$\bbs^P$};
\draw[semithick, dashed, ->] (0,1.3) -- (0,0.5);
\node[left] at (0,0.9) {$\up^{-1}$};
\draw[semithick, dashed, ->] (0.7,-1.8) -- (2.5,-1.8);
\node[below] at (1.5,-1.8) {$\NOR$};
\draw[semithick, dashed, ->] (0.7,1.8) -- (2.5,1.8);
\node[above] at (1.5,1.8) {$\NAR$};
\draw[semithick, dashed, ->] (3.25,1.3) -- (3.25,0.5);
\node[right] at (3.25,0.9) {$\up^{-1}$};
\draw[semithick, dashed, ->] (0,-0.5) -- (0,-1.3);
\node[left] at (0,-0.9) {$\Theta$};
\draw[semithick, dashed, ->] (3.25,-0.5) -- (3.25,-1.3);
\node[right] at (3.25,-0.9) {$\Theta$};
\end{tikzpicture}
\end{center}

Since $\NOR=\Theta \circ \up^{-1} \circ \down$,
this leads to the following simpler commutative diagram.

\begin{center}
\begin{tikzpicture}[yscale=2/3]
\node at (0,1.8) {$\bbs^P$};
\node at (0,0) {$\bbs^P$};
\node at (3.25,1.8) {$\bbs^P$};
\node at (3.25,0) {$\bbs^P$};
\draw[semithick, dashed, ->] (0,1.3) -- (0,0.5);
\node[left] at (0,0.9) {$\down$};
\draw[semithick, dashed, ->] (0.7,0) -- (2.5,0);
\node[below] at (1.5,0) {$\NAR$};
\draw[semithick, dashed, ->] (0.7,1.8) -- (2.5,1.8);
\node[above] at (1.5,1.8) {$\NOR$};
\draw[semithick, dashed, ->] (3.25,1.3) -- (3.25,0.5);
\node[right] at (3.25,0.9) {$\down$};
\end{tikzpicture}
\end{center}

\begin{proof}
This proof is similar to the proof
of~\cite[Theorem 3.21]{antichain-toggling}.

Let $(x_1,x_2,\dots,x_n)$ be any linear extension of a finite poset $P$.  By the definitions,
$\NAR=\tau_{x_n}\cdots \tau_{x_2}\tau_{x_1}$
and
$\NOR=T_{x_1} T_{x_2} \cdots T_{x_n}$.

Using the isomorphism from $\ntog_A(P)$ to $\ntog_O(P)$ given by $\tau_v \mapsto \tau_v^*$,
it suffices to show that
$\tau_{x_n}^*\cdots \tau_{x_2}^* \tau_{x_1}^*=\NOR = T_{x_1} T_{x_2} \cdots T_{x_n}$.
We will use induction to prove that $\tau_{x_k}^*\cdots \tau_{x_2}^* \tau_{x_1}^*
=T_{x_1} T_{x_2} \cdots T_{x_k}$ for $1\leq k\leq n$.

For the base case, $\tau_{x_1}^*=T_{x_1}$ since $x_1$ is a minimal element of $P$.
For the induction hypothesis, let $1\leq k\leq n-1$ and assume that
$\tau_{x_k}^*\cdots \tau_{x_2}^* \tau_{x_1}^*
=T_{x_1} T_{x_2} \cdots T_{x_k}$.
Then 
\begin{equation}\label{eq:future frenzy}
\tau_{x_{k+1}}^*\tau_{x_k}^*\cdots \tau_{x_2}^* \tau_{x_1}^*
= \eta_{x_{k+1}} T_{x_{k+1}} \eta_{x_{k+1}}^{-1} T_{x_1} T_{x_2} \cdots T_{x_k}.
\end{equation}
Let $(y_1,\dots,y_{k'})$ be a linear extension of the subposet $\{y\in P \;|\; y<x_{k+1}\}$ of $P$.
Then since $(x_1,\dots,x_n)$ is a linear extension of $P$, all of $y_1,\dots,y_{k'}$ must be in $\{x_1,\dots,x_k\}$.
Furthermore, any element less than one of $y_1,\dots,y_{k'}$ must be less than $x_{k+1}$ so none of the elements of $\{x_1,\dots,x_k\}$ outside of $\{y_1,\dots,y_{k'}\}$ are less than any of $y_1,\dots,y_{k'}$.
Therefore, we can name these elements in such a way that $(y_1,\dots,y_{k'}, y_{k'+1}, \dots, y_k)$ is a linear extension of $\{x_1,\dots,x_k\}$.
Noting again that any two linear extensions of a poset differ by a sequence of swaps between
adjacent incomparable elements~\cite{etienne-84}, 
toggles of incomparable elements commute so $T_{x_1}T_{x_2}\cdots T_{x_k} = T_{y_1}\cdots T_{y_{k'}}T_{y_{k'+1}}\cdots T_{y_k}$.
From Eq.~(\ref{eq:future frenzy}) and $\eta_{x_{k+1}}=T_{y_1}\cdots T_{y_{k'}}$, we obtain
\begin{align*}
\tau_{x_{k+1}}^*\tau_{x_k}^*\cdots \tau_{x_2}^* \tau_{x_1}^* &=
\eta_{x_{k+1}} T_{x_{k+1}} \eta_{x_{k+1}}^{-1}
T_{x_1} T_{x_2} \cdots T_{x_k}
\\&=
T_{y_1}\cdots T_{y_{k'}} T_{x_{k+1}} E_{y_{k'}}\cdots E_{y_1} T_{y_1}\cdots T_{y_{k'}}T_{y_{k'+1}}\cdots T_{y_k}
\\&=
T_{y_1}\cdots T_{y_{k'}} T_{x_{k+1}} T_{y_{k'+1}}\cdots T_{y_k}
\\&=
T_{y_1}\cdots T_{y_{k'}} T_{y_{k'+1}}\cdots T_{y_k}  T_{x_{k+1}}
\\&=
T_{x_1} T_{x_2} \cdots T_{x_k} T_{x_{k+1}}.
\end{align*}
In the fourth equality above, we could move $T_{x_{k+1}}$ to the right of $T_{y_{k'+1}}\cdots T_{y_k}$ because $x_{k+1}$ is incomparable with each of $y_{k'+1},\dots, y_k$.  This is because none of these are less than $x_{k+1}$ by design nor greater than $x_{k+1}$ by position within the linear extension $(x_1,\dots,x_n)$ of $P$.

By induction, we have $\tau_{x_n}^*\cdots \tau_{x_2}^* \tau_{x_1}^* = T_{x_1} T_{x_2} \cdots T_{x_n} = \NOR = \Theta\circ\up^{-1}\circ \down$ so
$$\tau_{x_n}\cdots \tau_{x_2} \tau_{x_1} = \NAR
=\down\circ\Theta\circ\up^{-1}.$$
\end{proof}

From Theorem~\ref{thm:NAR-transfer}, and the ensuing commutative diagrams, the orders of $\NAR$ and $\NOR$ are equal on any poset.  So Grinberg's Conjecture~\ref{conj:NOR-prod-chains}
is equivalent to the claim that $\NAR$ has order $a+b$ on
$[a]\times[b]$. Although we do not resolve this conjecture here, we hope that giving another approach from the antichain perspective may be helpful in studying these questions.  

Furthermore, it appears at every step of the process, the labels can be written in a way that is no more complicated than in the (commutative) birational realm.  What we mean by that is they can be written in a way that contains every factor from the birational realm (multiplied in a certain order) and does {\bf not} require extra factors that would cancel in the commutative realm.
Compare Figures~\ref{fig:BAR[2]x[3]} and~\ref{fig:NAR[2]x[3]} for $P=[2]\times[3]$.

We can extend the main results on graded rescalings from \S\ref{ss:g-rescaling} to the
noncommutative setting as long as each component of the rescaling vector $(a_{0},\dots ,a_{r})$
lies in the \emph{center} of $\bbs$.  Under this assumption
Propositions~\ref{prop:rank-tog-graded-rescaling} and~\ref{prop:BAR-graded-rescaling} go
through with $\BAR$ replaced by $\NAR$. (The analogous results for $\NOR$ are true as well.)
We omit the details.

\bibliography{bibliography}
\bibliographystyle{halpha}
\end{document}